\documentclass{amsart}
\usepackage[utf8]{inputenc}
\usepackage{amsmath,amsfonts,amsthm,amssymb,mathrsfs}
\usepackage[margin=1.25in]{geometry}
\usepackage{verbatim}
\usepackage{esint}
\usepackage{xcolor}
\usepackage{caption}
\usepackage[citecolor=blue,colorlinks, linkcolor=blue]{hyperref}
\usepackage[normalem]{ulem}
\usepackage{comment}
\usepackage{tikz, multicol}
\usepackage{wrapfig}
\usepackage{graphicx}
\usepackage[T1]{fontenc}

\newtheorem{theorem}{Theorem}[section]
\newtheorem{definition}[theorem]{Definition}

\newtheorem{lemma}[theorem]{Lemma}

\newtheorem{corollary}[theorem]{Corollary}
\newtheorem{proposition}[theorem]{Proposition}
\newtheorem{remark}[theorem]{Remark}
\newtheorem{question}{Question}

\newcommand{\eps}{\varepsilon}
\newcommand{\vp}{\varphi}

\newcommand{\al}{\alpha}
\newcommand{\be}{\beta}
\newcommand{\ga}{\gamma}
\newcommand{\de}{\delta}

\newcommand{\Ga}{\Gamma}
\newcommand{\te}{\theta}
\newcommand{\la}{\lambda}

\newcommand{\si}{\sigma}

\DeclareMathOperator{\conv}{conv}

\newcommand{\disp}{\displaystyle}
\newcommand{\iny}{\infty}

\newcommand{\su}{\subseteq}

\newcommand{\N}{\mathbb{N}}
\newcommand{\R}{\mathbb{R}}

\newcommand{\Ha}{\mathcal{H}}
\newcommand{\dist}{\operatorname{dist}}
\newcommand{\diam}{\operatorname{diam}}

\newcommand{\Fav}{\operatorname{Fav}}
\newcommand{\Lip}{\operatorname{Lip}}

\newcommand{\abs}[1]{\left\vert#1\right\vert}
\newcommand{\set}[1]{\left\{#1\right\}}
\newcommand{\brac}[1]{\left[#1\right]}
\newcommand{\pr}[1]{\left( #1 \right) }

\newcommand{\pb}[1]{\left( #1 \right] }
\newcommand{\Ho}[1]{\mathcal{H}^1 \left( #1 \right) }

\title[Self-similar sets and Lipschitz graphs]{Self-similar sets and Lipschitz graphs}

\author[B.\,Davey]{Blair Davey}
\address{Department of Mathematical Sciences, Montana State University, Bozeman, MT 59717, USA}
\email{blairdavey@montana.edu}

\author[S.\,Ghinassi]{Silvia Ghinassi}
\address{Math Department, Shoreline Community College, Shoreline, WA 98133, USA}
\email{sghinassi@shoreline.edu}

\author[B.\,Wilson]{Bobby Wilson}
\address{Department of Mathematics, University of Washington, Seattle, WA 98195, USA}
\email{blwilson@uw.edu}

\date{\today}


\begin{document}
\begin{abstract}
    We investigate and quantify the distinction between rectifiable and purely unrectifiable $1$-sets in the plane.
    That is, given that purely unrectifiable $1$-sets always have null intersections with Lipschitz images, we ask whether these sets intersect with Lipschitz images at a dimension that is close to one.
    In an answer to this question, we show that one-dimensional attractors of iterated function systems that satisfy the open set condition have subsets of dimension arbitrarily close to one that can be covered by Lipschitz graphs. 
    Moreover, the Lipschitz constant of such graphs depends explicitly on the difference between the dimension of the original set and the subset that intersects with the graph.
\end{abstract}
\maketitle
\tableofcontents

\section{Introduction}

Rectifiable sets have been widely studied in the last century, following the pioneering work of Besicovitch, Federer, and Marstrand \cite{besicovitch1928,federer1947,marstrand1954}, among many others. 
They have since been characterized and utilized as a class of ``nice'' sets for many problems in analysis. 
Their ``bad'' counterpart, purely unrectifiable sets, have also been studied for their remarkable measure-theoretic properties. 
Attempting to provide a complete reference list for such statements is beyond our scope; see, for instance, \cite{mattila, falconer1986} for an exposition of such topics.

Here we focus on one-dimensional sets in the plane with finite and positive measure ($1$-sets); those sets $E \su \R^2$ with $\dim(E) = 1$ and $0 < \mathcal{H}^1(E) < \iny$.
Here $\dim(\cdot)$ denotes the Hausdorff dimension and $\mathcal{H}^s$ denotes the $s$-dimensional Hausdorff measure, see Definition \ref{def:Hausdorff}.
Let $\mathcal{I}$ denote the collection of all Lipschitz images in the plane.
That is, we say $I \in \mathcal{I}$ if there exists a Lipschitz function $g:[0,1] \to \mathbb{R}^2$ such that $I = g([0,1])$.
We say that a $1$-set $E$ is \textit{rectifiable} if there exists a countable collection of Lipschitz images $\set{I_i}_{i=1}^\iny \su \mathcal{I}$ for which 
$$\Ho{E \setminus \bigcup_{i=1}^\iny I_i} = 0.$$
On the other hand, $E$ is said to be \textit{purely unrectifiable} if for every Lipschitz image $I \in \mathcal{I}$, it holds that
$$\Ho{E \cap I} = 0.$$
Therefore, it is natural to ask: Do purely unrectifiable sets see Lipschitz images at a lower dimension?
 In particular, we ask the following:

\begin{question}
\label{imagesQ}
If $E \su \R^2$ is a purely unrectifiable $1$-set, how large is
    \begin{align*}
        \sup \set{\dim\pr{E \cap I} : I \in \mathcal{I}}?
    \end{align*}
\end{question}

In Section \ref{Section:lines}, we prove that if $E$ is an Alfhors regular, purely unrectifiable $1$-set, then 
    \begin{align*}
        \sup \set{\dim\pr{E \cap I} : I \in \mathcal{I}}=1.
    \end{align*} 
That is, we use a $\beta$-number computation to show that there exists $I \in \mathcal{I}$ such that $\dim\pr{E \cap I} =1$.

As shown in \cite{garnett2010}, and expanded upon in \cite{BS23} and \cite{Bad19}, if one considers rectifiable measures that are not absolutely continuous with respect to Hausdorff measure, then there exist measures with support equal to $\R^2$ that are carried by one-dimensional sets and are Lipschitz image rectifiable but have measure zero when intersected with any Lipschitz graph.
However, as long as one focuses their attention on rectifiable or purely unrectifiable \emph{sets} (as we do), we can replace Lipschitz images by Lipschitz graphs in the above definitions. 
The natural question then becomes: What if we restrict to Lipschitz graphs instead of images? 
Let $\mathcal{G}$ denote the set of one-dimensional Lipschitz graphs in the plane.
Now we ask the following question:

\begin{question}
\label{graphsQ}
If $E \su \R^2$ is a purely unrectifiable $1$-set, how large is
$$\sup \set{\dim\pr{E \cap \Ga} : \Ga \in \mathcal{G}}?$$
\end{question}

As a starting point, we consider straight lines as a ``toy example'' for graphs, and pose the following question:

\begin{question}
\label{linesQ}
If $E \su \R^2$ is a purely unrectifiable $1$-set, what can be said about
$$\sup \set{\dim\pr{E \cap L} : L \in A(2, 1)}?$$ 
Here $A(2, 1)$ denotes the collection of all lines in $\mathbb{R}^2$.
\end{question}

The main goal of this paper is to quantify the difference between rectifiable and purely unrectifiable sets and their relationships with the families of sets discussed above. 
As we show in Section \ref{Section:lines}, the heart of the matter will be Question \ref{graphsQ}, as Questions \ref{imagesQ} and \ref{linesQ} are somewhat less involved.
In the case of lines, the answer depends on the specific geometry of the set; while for the case of Lipschitz images, we show that for many purely unrectifiable sets, the intersection with the set has dimension 1, of course, with zero measure. 
This result for Lipschitz images is given in Proposition \ref{imagesprop}.
Our answer to Question \ref{graphsQ} regarding Lischitz graphs is given in Theorem \ref{mainthm}.

An important class of purely unrectifiable sets arises as the attractors of dynamical systems known as iterated function systems (IFS).
Many well-known fractals are attractors of IFS which satisfy the open set condition. 
See \cite{schief1994separation, schief1996self} for more details about the necessity of the open set condition.
In this paper, we focus on $1$-dimensional attractors of iterated function systems that satisfy the open set condition (see Section \ref{Section:prelim} for definitions). 
By assuming the open set condition, the sets that we work with necessarily have two fundamental properties: self-similarity and Ahlfors regularity, both of which are key factors in our constructions. 

A natural question in the theory of IFS is whether their attractors can be parametrized by nice curves. 
In this direction, Hata \cite{hata1985structure} showed that if the attractor of an IFS is connected, then it is the continuous image of $[0,1]$. 
More recently, Badger and Vellis \cite{badger2020holder} improved the aforementioned result by showing that if an attractor is connected, then it is the image of a H\"older curve. 
Shaw and Vellis \cite{shaw2023parametrizability} provide sufficient conditions for the attractor of an infinite IFS to be parametrized by a H\"older curve.
In \cite{BK24}, similar problems in metric spaces are explored for both H\"older and Lipschitz maps.

When working with the attractor of an IFS, our primary goal in the Lipschitz graph contructions will be to identify what we refer to as ``good directions''. 
These will be directions over which, at a given step of the construction, one can construct rectilinear graphs whose neighborhoods cover a significant number of the ``pieces'' of the IFS. 
This idea is reminiscent of the application of Dilworth's lemma in the papers of Alberti, Cs\"ornyei, and Preiss \cite{alberti2005structure, alberti2010differentiability} where they construct neighborhoods of 1-Lipschitz graphs that contain significant subsets of finite point sets.
Dilworth's lemma in this context is not strong enough for us to reach our final conclusion, so we are required to use information about the structure of projections of the IFS. 
To that end, we turn to quantitative projection theorems. 

In \cite{tao2009quantitative, davey2022quantification}, the authors prove quantitative Besicovitch projection theorems and they demonstrate that the extent to which a set cannot be covered by Lipschitz graph neighborhoods leads to an upper estimate on Favard length (the average size of the projections) of a neighborhood of the given set. 
More recently, in \cite{CDOV24}, the authors establish quantitative Lipschitz covering results for sets with large Favard length.
Lower bounds on the Favard length (due to Mattila \cite{mattila1990}, see also Bongers \cite{bongers2019geometric} where convexity arguments are used to re-prove some of the results in \cite{mattila1990}), combined with the ideas in \cite{tao2009quantitative, davey2022quantification} would imply that most sets can be covered (locally) with some Lipschitz graph neighborhoods. 
However, the locality condition and the fact that the proofs are not constructive limit their adaptability to our setting. 
Therefore, in this article, we have detailed a completely novel algorithm for the identification of good directions and the construction of the Lipschitz graphs.

We do not know if the bounds on the Lipschitz constants in Theorem \ref{mainthm} are sharp.  
It would be interesting to know whether these results extend to a broader class of sets, such as higher-dimensional, higher-codimensional, or not self-similar sets. 
In these more general cases, obstacles with varying degrees of difficulty appear, including the fact that the regularization of self-similar sets \emph{\`a la} Peres and Shmerkin \cite{peres2009resonance} is not as strong and the process of constructing higher dimensional Lipschitz graphs is not as simple as the process for one-dimensional Lipschitz graphs.

\subsection{Organization of the article}

$\quad$ \\ 

The organization of the paper closely follows the steps we took in approaching the problem: 
In Section \ref{Section:lines}, we first address Questions \ref{imagesQ} and \ref{linesQ}. 
Section \ref{Section:prelim} provides definitions and preliminary results for iterated function systems, which can be of interest on their own. 
In Section \ref{Section:graph}, we demonstrate a general Lipschitz graph construction that parametrizes the $\limsup$ set of a suitable family of well-separated nested compact sets. 
In the sections that follow, we provide answers to Question \ref{graphsQ}. 

In Section \ref{Section:4corner}, we present two graph constructions for the $4$-corner Cantor set, $\mathcal{C}_4$. The set, also known as Garnett-Ivanov Cantor set, appears in works of Ivanov, translated in \cite{Iva84}, and later in Garnett \cite{Gar70}, who both proved the set has zero analytic capacity (while having positive length). 
Analogous sets, with scaling larger than $\frac14$, were first considered by Denjoy \cite{denjoy1932sur} who proved those have positive continuous analytic capacity and positive Newtonian capacity.

The first construction described in Section \ref{Section:4corner} is \emph{ad hoc}, while the second one illustrates the general method that is used in Section \ref{Section:norota}.

In Section \ref{Section:norota}, we consider the attractors of rotation-free iterated function systems. 
We first introduce the Favard length, and by using the universal lower bound on its decay (due to Mattila \cite{mattila1990}), we are able to find a substantial and well-separated set, in the vein of \cite{tao2009quantitative, davey2022quantification}. 

More specifically, since the Favard length of each neighborhood of the attractor is substantial, then for each neighborhood, there must be at least one good angle onto which the orthogonal projection is substantial.
By recasting the attractor of an IFS as a limit of its ``generations'', we see that every generation has a good angle onto which its projection is substantial.
From this observation, we use a Vitali-type argument to show that the substantial projection can be ``nearly'' covered by a well-separated set.
This reduction to a substantial ``near cover'' corresponds to a sub-IFS of an iteration of the original IFS with a relatively high similarity dimension.
Since the sub-IFS is also rotation-free, then all of its generations have well-separated projections onto the fixed angle.
Using self-similarity, we recursively build the Lipschitz graph over the attractor of this sub-IFS and carefully track its Lipschitz constant.

In the rotational case, presented in Section \ref{Section:rota}, the construction is much more delicate.
Due to the presence of rotations, an angle that may be good at one scale could fail to be good at other scales, making it difficult to choose a projection angle.
To overcome this challenge, we rely on ergodic theory. 
The first step in the construction is to reduce the original IFS to a uniform sub-IFS (with a loss of dimension), and here we follow a result of Peres and Shmerkin \cite{peres2009resonance}.
In the second step, we use additional tools from \cite{peres2009resonance} to show that every generation of the uniform sub-IFS has lots of good projection angles.
The ideas here are reminiscent of Mattila's lower bound in \cite{mattila1990}, but the additional quantitative information is important for our application.
We then use the maximal ergodic theorem to argue that there is at least one angle onto which most generations have good projection properties.
With the good angle in hand, we then mimic the techniques in the previous section to build the graph.

Since it is absent from the literature, our Appendix \ref{appendix} includes a proof, due to Davies, of the fact that $\mathcal{H}^1(\mathcal{C}_4)=\sqrt{2}$.
A different proof of the same fact can be found, in French, in \cite{marion1979}.

\subsection{Notation}

$\quad$ \\ 

For a set $E \su \R^d$, we write $\text{int}(E)$ to denote it interior, $\overline{E}$ to denote its closure, $\conv(E)$ to denote its convex hull, and $\diam(E)$ to denote its diameter.
Given $r > 0$, $E(r)$ denotes the closed $r$-neighborhood of $E$; that is, $E(r) = \set{y \in \R^d : \abs{x - y} \le r \text{ for some } x \in E}$.
We write $B(x,r) = \set{y \in \R^d : \abs{x - y} \le r}$ to denote the closed ball of radius $r$ centered at $x \in \R^d$.
A cube $Q \su \R$ has side length denoted by $\ell(Q)$.
Given a cube $Q$ and a constant $c > 0$, we use the notation $c \, Q$ to denote the cube with the same center at $Q$ and side length $\ell(c \, Q) = c \, \ell(Q)$.
If $E$ is a finite set, then we write $\# E$ to denote the number of elements in $E$.
For a Lebesgue measurable set $E \su \mathbb{R}^d$, $|E|$ will denote the Lebesgue measure of $E$.
Let $P_\te : \R^2 \to \R$ denote the orthogonal projection onto a line of angle $\te$.
That is, $P_\te(z) = x\cos \te  + y\sin \te $ for any point $z = (x, y) \in \R^2$.
We may also write $P_x$ and $P_y$ to denote the projections onto that $x$- and $y$-axes, respectively.
For inequalities, we write $A \lesssim B$ if there exists a constant $c > 0$ such that $A \leq c B$.

\subsection{Funding and acknowledgements}

$\quad$ \\ 

B.D. was partially supported by the NSF LEAPS-MPS DMS-2137743 and NSF CAREER DMS-2236491.
S.G. was partially supported by NSF DMS-1854147 and would like to thank Giovanni Alberti, Camillo De Lellis, and Hong Wang for several helpful conversations on similar questions in the early days of the COVID-19 pandemic. 
B.W. was supported by NSF CAREER DMS-2142064.

\section{Intersections with Lines and Lipschitz Images} 
\label{Section:lines}

Here we provide answers to Questions \ref{linesQ} and \ref{imagesQ}.
We begin with a few definitions that allow us to describe graphs and $s$-sets.

\begin{definition}[Lipschitz functions]
   For $A \su \R$, let $g : A \to \R$ and set $\Gamma = \{(x,g(x)) \mid x \in A\}$ to be its graph.
   We say that $g$ is \textbf{Lipschitz} if there exists $\lambda >0$ such that, for all $x,y \in A$,
   \begin{equation*}
   |f(x)-f(y)| \leq \lambda |x-y|.
   \end{equation*}
   We say that $g$ is \textbf{biLipschitz} if there exists $\lambda >0$ such that, for all $x,y \in A$,
   \begin{equation*}
   \frac{1}{\lambda} |x-y| \leq |f(x)-f(y)| \leq \lambda |x-y|.
   \end{equation*}
If $g$ is Lipschitz, then we define the \textbf{Lipschitz constant} as
\begin{equation*}
\Lip(\Gamma) = \inf\{ \lambda\in \R \mid |g(x)-g(y)| \leq \lambda |x-y|, \text{ for all $x,y \in A$}\}.
\end{equation*}
\end{definition}

\begin{definition}[Hausdorff measure and dimension]
\label{def:Hausdorff}
For any $E \su \R^d$, $\de > 0$ and $s \ge 0$, define 
$$\mathcal{H}^s_\de(E) = \inf \set{ \sum _{i = 1}^\iny \diam\pr{U_i}^s : E \su\bigcup_{i=1}^\iny U_i, \diam\pr{U_i} < \de}.$$
The \textbf{$s$-dimensional Hausdorff measure} of a set $E \su \R^d$ is defined as
$$\mathcal{H}^s(E) = \lim_{\de \downarrow 0} \mathcal{H}^s_\de(E).$$
The \textbf{Hausdorff dimension} is defined as
    \begin{equation*}
        \dim(E):= \sup\{s\geq 0~:~ \mathcal{H}^{s}(E)>0\}.
    \end{equation*}    
\end{definition}

\begin{definition}[Ahlfors regularity]
\label{def AReg}
  We say that a set $E \su \R^d$ is \textbf{$s$-Ahlfors regular} if there exist constants $a,b >0$ such that, for all $x \in E$ and $r \in (0,\diam (E))$, we have
  \begin{equation*}
  a r^s \leq \mathcal{H}^s(E \cap B(x,r)) \leq b r^s.
  \end{equation*}
  We refer to $a$ and $b$ as the \textbf{Ahlfors lower and upper constants}, respectively.
\end{definition}

\subsection{Intersections of purely unrectifiable 1-sets with lines}

$\quad$ \\ 

In response to Question \ref{linesQ}, the following discussion shows that the answer depends on the specific geometric structure of the $1$-set $E$.

The following result shows that the intersection dimension can be zero. 

\begin{proposition}[Simple Venetian blind construction]
    There exists a set $E \su \brac{0,1}^2$ such that $\dim(E) \ge 1$ and
        \begin{align*}
            \sup \set{\dim(E \cap L) : L \in A(2, 1)} =0,
        \end{align*}
    where we recall that $A(2, 1)$ denotes the collection of all lines in $\mathbb{R}^2$.
\end{proposition}

This proof is adapted from the proof of \cite[Theorem 5.11]{falconer1986}.

\begin{proof}
Consider a sequence of integers $\{m_k\}_{k=1}^{\infty}$ such that $m_k \to \infty$. 
Define two sets $A, B \su [0, 1]$ by 
	\begin{align*}
		A&:= \left\{ \sum_{j=1}^{\infty} \frac{a_j}{2^j}~:~ a_j \in  \{0, 1\}, \, a_j= 0 \mbox{ for } m_{2k} \leq j < m_{2k+1} \mbox{ for all } k \right\}\\
		B&:= \left\{ \sum_{j=1}^{\infty} \frac{b_j}{2^j}~:~ b_j \in  \{0, 1\}, \, b_j= 0 \mbox{ for } m_{2k-1} \leq j < m_{2k} \mbox{ for all } k  \right\}.
	\end{align*}
If $m_k=10^{10^k}$, then as shown in \cite{falconer1986}, $\dim\pr{A} = \dim\pr{B} = 0$.

Set $E = A \times B$.
Recall that $P_\te$ denotes the orthogonal projection onto the line of angle $\te$.
Since $A+B= [0, 1]$, then $\dim\pr{P_{\frac \pi 4}\pr{E}} = 1$ and it follows that $\dim(E) \geq 1$. 

Whenever $L \in A(2,1)$ is not a horizontal line, since the projection of $L$ onto the $x$-axis is a biLipschitz map, then it holds that 
$$\dim(E\cap L) = \dim(P_0(E \cap L)) \le \dim(P_0(E)) = \dim(A) = 0.$$
Similarly, for any non-vertical line $L \in A(2,1)$, 
$$\dim(E\cap L) = \dim(P_{\frac \pi 2}(E \cap L)) \le \dim(P_{\frac \pi 2}(E)) = \dim(B) = 0.$$
Therefore, for any line $L \in A(2,1)$, $\dim(E \cap L) = 0$, and the conclusion follows.
\end{proof}

By example, we can show that other values in $\pr{0,1}$ may be achieved as the dimension of intersection with lines.

Let $\mathcal{C}_4$ be the 4-corner Cantor set, defined as $\disp \mathcal{C}_4= \bigcap_{n=0}^{\infty}C_n$, where $C_0=[0,1]^2$, $C_1$ is the union of the $4$ squares at the corners of side-length $\frac14$, and similarly $C_n$ is the union of $4^n$ squares of side-length $4^{n}$.
Since any line that intersects $C_0$ can intersect at most two of the squares in $C_1$, then an iterative argument shows that
$$\sup \set{\dim\pr{\mathcal{C}_4 \cap L} : L \in A(2,1)} = \frac 1 2.$$
In particular, $\mathcal{C}_4$ covers the mid-range of Question \ref{linesQ}.

For $k \ge 4$, let $\mathcal{C}_k$ be a k-square Cantor set, defined as $\disp \mathcal{C}_k= \bigcap_{n=0}^{\infty}C_n$, where  $C_0=[0,1]^2$, and $C_1$ is the union of $k$ squares, each of side-length $\frac 1 k$, $4$ of which are placed in the corners of $C_0$, and the remaining $k - 4$ are evenly spaced along the $x$-axis.
Then $C_n$ is defined recursively to consist of $k^n$ squares of side length $k^{-n}$.
Images of the first two iterations of $\mathcal{C}_6$ are provided in Figure \ref{C6Image}. 
If $L$ is the horizontal line through the origin, then $\disp \dim\pr{\mathcal{C}_k \cap L} = \frac {\log\pr{k-2}}{\log k}$.

In particular, by choosing $k \gg 1$, we may make the intersection dimension arbitrarily close to $1$.
\begin{figure}[ht]
\centering
\begin{tikzpicture}[scale=0.4]

\fill[black]
 (0, 0) -- (0, 1.5) --
(0, 1.5) -- (1.5, 1.5) --
(1.5, 1.5) -- (1.5, 0) -- 
   cycle; 

\fill[black]
(0, 7.5) -- (0, 9) --
(0, 9) -- (1.5, 9) --
(1.5, 9) -- (1.5, 7.5) -- 
   cycle;

\fill[black]
(2.5, 0) -- (2.5, 1.5) --
(2.5, 1.5) -- (4, 1.5) --
(4, 1.5) -- (4, 0) -- 
   cycle;

\fill[black]
(5, 0) -- (5, 1.5) --
(5, 1.5) -- (6.5, 1.5) --
(6.5, 1.5) -- (6.5, 0) -- 
   cycle;

\fill[black]
 (7.5, 7.5) -- (7.5, 9) --
(7.5, 9) -- (9, 9) --
(9, 9) -- (9, 7.5) -- 
   cycle;
   
\fill[black]
 (7.5, 0) -- (7.5, 1.5) --
(9, 1.5) -- (9, 0) -- 
(9, 0) -- (7.5, 0) -- 
   cycle;

\end{tikzpicture}
\hspace{1cm}
\begin{tikzpicture}[scale=0.066666666]%0.3/6

%bottom row 
\fill[black]
 (0, 0) -- (0, 1.5) --
(0, 1.5) -- (1.5, 1.5) --
(1.5, 1.5) -- (1.5, 0) -- 
%(1.5, 0) -- (0, 0) -- 
   cycle; 

\fill[black]
(0, 7.5) -- (0, 9) --
(0, 9) -- (1.5, 9) --
(1.5, 9) -- (1.5, 7.5) -- 
%(1.5, 7.5) -- (0, 7.5) -- 
   cycle;

   \fill[black]
(2.5, 0) -- (2.5, 1.5) --
(2.5, 1.5) -- (4, 1.5) --
(4, 1.5) -- (4, 0) -- 
%(1.5, 7.5) -- (0, 7.5) -- 
   cycle;

    \fill[black]
(5, 0) -- (5, 1.5) --
(5, 1.5) -- (6.5, 1.5) --
(6.5, 1.5) -- (6.5, 0) -- 
   cycle;

   \fill[black]
 (7.5, 7.5) -- (7.5, 9) --
(7.5, 9) -- (9, 9) --
(9, 9) -- (9, 7.5) -- 
%(9, 7.5) -- (7.5, 7.5) -- 
   cycle;
   
\fill[black]
 (7.5, 0) -- (7.5, 1.5) --
%(8, 1) -- (8, 1) --
(9, 1.5) -- (9, 0) -- 
(9, 0) -- (7.5, 0) -- 
   cycle;

   %%
%BR
   \fill[black]
 (52.5, 0) -- (52.5, 1.5) --
(52.5, 1.5) -- (54, 1.5) --
(54, 1.5) -- (54, 0) -- 
%(1.5, 0) -- (0, 0) -- 
   cycle; 
   
%TR
\fill[black]
(52.5, 7.5) -- (52.5, 9) --
(52.5, 9) -- (54, 9) --
(54, 9) -- (54, 7.5) -- 
%(1.5, 7.5) -- (0, 7.5) -- 
   cycle;

%T2R
   \fill[black]
(51.5, 0) -- (51.5, 1.5) --
(51.5, 1.5) -- (50, 1.5) --
(50, 1.5) -- (50, 0) -- 
%(1.5, 7.5) -- (0, 7.5) -- 
   cycle;

%T2L
    \fill[black]
(49, 0) -- (49, 1.5) --
(49, 1.5) -- (47.5, 1.5) --
(47.5, 1.5) -- (47.5, 0) -- 
%(1.5, 7.5) -- (0, 7.5) -- 
   cycle;

   \fill[black]
 (46.5, 7.5) -- (46.5, 9) --
(46.5, 9) -- (45, 9) --
(45, 9) -- (45, 7.5) -- 
%(9, 7.5) -- (7.5, 7.5) -- 
   cycle;

     \fill[black]
 (46.5, 0) -- (46.5, 1.5) --
(46.5, 1.5) -- (45, 1.5) --
(45, 1.5) -- (45, 0) -- 
%(9, 7.5) -- (7.5, 7.5) -- 
   cycle;

%top row (y = y + 45

\fill[black]
 (0, 0+45) -- (0, 1.5+45) --
(0, 1.5+45) -- (1.5, 1.5+45) --
(1.5, 1.5+45) -- (1.5, 0+45) -- 
%(1.5, 0) -- (0, 0) -- 
   cycle; 

\fill[black]
(0, 7.5+45) -- (0, 9+45) --
(0, 9+45) -- (1.5, 9+45) --
(1.5, 9+45) -- (1.5, 7.5+45) -- 
%(1.5, 7.5) -- (0, 7.5) -- 
   cycle;

   \fill[black]
(2.5, 0+45) -- (2.5, 1.5+45) --
(2.5, 1.5+45) -- (4, 1.5+45) --
(4, 1.5+45) -- (4, 0+45) -- 
%(1.5, 7.5) -- (0, 7.5) -- 
   cycle;

    \fill[black]
(5, 0+45) -- (5, 1.5+45) --
(5, 1.5+45) -- (6.5, 1.5+45) --
(6.5, 1.5+45) -- (6.5, 45) -- 
%(1.5, 7.5) -- (0, 7.5) -- 
   cycle;

   \fill[black]
 (7.5, 7.5+45) -- (7.5, 9+45) --
(7.5, 9+45) -- (9, 9+45) --
(9, 9+45) -- (9, 7.5+45) -- 
%(9, 7.5) -- (7.5, 7.5) -- 
   cycle;
   
\fill[black]
 (7.5, 0+45) -- (7.5, 1.5+45) --
%(8, 1) -- (8, 1) --
(9, 1.5+45) -- (9, 0+45) -- 
(9, 0+45) -- (7.5, 0+45) -- 
   cycle;

   %%
%BR
   \fill[black]
 (52.5, 0+45) -- (52.5, 1.5+45) --
(52.5, 1.5+45) -- (54, 1.5+45) --
(54, 1.5+45) -- (54, 0+45) -- 
%(1.5, 0) -- (0, 0) -- 
   cycle; 
   
%TR
\fill[black]
(52.5, 7.5+45) -- (52.5, 9+45) --
(52.5, 9+45) -- (54, 9+45) --
(54, 9+45) -- (54, 7.5+45) -- 
%(1.5, 7.5) -- (0, 7.5) -- 
   cycle;

%T2R
   \fill[black]
(51.5, 45) -- (51.5, 1.5+45) --
(51.5, 1.5+45) -- (50, 1.5+45) --
(50, 1.5+45) -- (50, 45) -- 
%(1.5, 7.5) -- (0, 7.5) -- 
   cycle;

%T2L
    \fill[black]
(49, 45) -- (49, 1.5+45) --
(49, 1.5+45) -- (47.5, 1.5+45) --
(47.5, 1.5+45) -- (47.5, 45) -- 
%(1.5, 7.5) -- (0, 7.5) -- 
   cycle;

   \fill[black]
 (46.5, 7.5+45) -- (46.5, 9+45) --
(46.5, 9+45) -- (45, 9+45) --
(45, 9+45) -- (45, 7.5+45) -- 
%(9, 7.5) -- (7.5, 7.5) -- 
   cycle;

     \fill[black]
 (46.5, 0+45) -- (46.5, 1.5+45) --
(46.5, 1.5+45) -- (45, 1.5+45) --
(45, 1.5+45) -- (45, 0+45) -- 
%(9, 7.5) -- (7.5, 7.5) -- 
   cycle;

%top left
%x = x +15

\fill[black]
 (0+15, 0) -- (0+15, 1.5) --
(0+15, 1.5) -- (1.5+15, 1.5) --
(1.5+15, 1.5) -- (1.5+15, 0) -- 
%(1.5, 0) -- (0, 0) -- 
   cycle; 

\fill[black]
(0+15, 7.5) -- (0+15, 9) --
(0+15, 9) -- (1.5+15, 9) --
(1.5+15, 9) -- (1.5+15, 7.5) -- 
%(1.5, 7.5) -- (0, 7.5) -- 
   cycle;

   \fill[black]
(2.5+15, 0) -- (2.5+15, 1.5) --
(2.5+15, 1.5) -- (4+15, 1.5) --
(4+15, 1.5) -- (4+15, 0) -- 
%(1.5, 7.5) -- (0, 7.5) -- 
   cycle;

    \fill[black]
(5+15, 0) -- (5+15, 1.5) --
(5+15, 1.5) -- (6.5+15, 1.5) --
(6.5+15, 1.5) -- (6.5+15, 0) -- 
%(1.5, 7.5) -- (0, 7.5) -- 
   cycle;

   \fill[black]
 (7.5+15, 0) -- (7.5+15, 1.5) --
(7.5+15, 1.5) -- (9+15, 1.5) --
(9+15, 1.5) -- (9+15, 0) -- 
%(9, 7.5) -- (7.5, 7.5) -- 
   cycle;
   
\fill[black]
 (7.5+15, 0+7.5) -- (7.5+15, 1.5+7.5) --
%(8, 1) -- (8, 1) --
(9+15, 1.5+7.5) -- (9+15, 0+7.5) -- 
(9+15, 0+7.5) -- (7.5+15, 0+7.5) -- 
   cycle;

   % x = x+30

   \fill[black]
 (0+30, 0) -- (0+30, 1.5) --
(0+30, 1.5) -- (1.5+30, 1.5) --
(1.5+30, 1.5) -- (1.5+30, 0) -- 
%(1.5, 0) -- (0, 0) -- 
   cycle; 

\fill[black]
(0+30, 7.5) -- (0+30, 9) --
(0+30, 9) -- (1.5+30, 9) --
(1.5+30, 9) -- (1.5+30, 7.5) -- 
%(1.5, 7.5) -- (0, 7.5) -- 
   cycle;

   \fill[black]
(2.5+30, 0) -- (2.5+30, 1.5) --
(2.5+30, 1.5) -- (4+30, 1.5) --
(4+30, 1.5) -- (4+30, 0) -- 
   cycle;

    \fill[black]
(5+30, 0) -- (5+30, 1.5) --
(5+30, 1.5) -- (6.5+30, 1.5) --
(6.5+30, 1.5) -- (6.5+30, 0) -- 
   cycle;

   \fill[black]
 (7.5+30, 7.5) -- (7.5+30, 9) --
(7.5+30, 9) -- (9+30, 9) --
(9+30, 9) -- (9+30, 7.5) -- 
   cycle;
   
\fill[black]
 (7.5+30, 0) -- (7.5+30, 1.5) --
(9+30, 1.5) -- (9+30, 0) -- 
(9+30, 0) -- (7.5+30, 0) -- 
   cycle;

\end{tikzpicture}
\caption{
\label{C6Image}
The first two iterations in the definition of $\mathcal{C}_{6}$}
\end{figure}
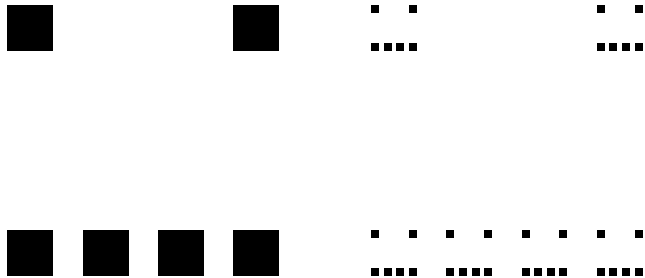

\subsection{Intersections of purely unrectifiable 1-sets with Lipschitz images}

$\quad$ \\ 

In response to Question \ref{imagesQ}, we use the Analyst's Traveling Salesperson Theorem to prove that for Ahlfors regular $1$-sets, we can always find a Lipschitz image whose intersection with the set has dimension $1$.
Recall that $\mathcal{I}$ denotes the collection of all Lipschitz images in the plane.
To state the theorem, we first need to define dyadic cubes and $\be$-numbers.

\begin{definition}[Dyadic cubes]
We let $\mathcal{Q}$ denote the collection of dyadic cubes; that is,
$$\mathcal{Q} = \set{[k 2^{-n}, (k+1) 2^{-n}) \times [j 2^{-n}, (j+1) 2^{-n}) : k, j, n \in \mathbb{Z}}.$$
The collection of all dyadic cubes $Q$ for which the side length is $2^{-n}$, written $\ell(Q) = 2^{-n}$, is denoted by
$$\mathcal{Q}_n = \set{[k 2^{-n}, (k+1) 2^{-n}) \times [j 2^{-n}, (j+1) 2^{-n}) : k, j \in \mathbb{Z}}.$$    
\end{definition}

Note that if $Q \in \mathcal{Q}_n$, then for any $k \in \N$, $\pr{2k+1} \cdot Q$ can be realized as a union of cubes in $\mathcal{Q}_n$.

\begin{definition}[$\beta$-numbers]
    Given $E \su \R^2$ and $Q \in \mathcal{Q}$, define
    \begin{align*}
        \beta_E(Q): = \begin{cases}
            \disp \inf_{L \in A(2, 1)} \sup_{x \in E \cap Q} \diam(Q)^{-1} \operatorname{dist}(x, L) & \text{ if } E \cap Q \neq \varnothing \\
            0 & \text{ if } E \cap Q = \varnothing.
        \end{cases}
    \end{align*}
\end{definition}

\begin{theorem}[Analyst's Traveling Salesperson Theorem, Theorem 1 in \cite{jones1990rectifiable}]
\label{atst}
    A bounded set $E \su \R^2$ is contained in a Lipschitz image $I \in \mathcal{I}$ if and only if
    \begin{equation}
    \label{def:betaNumber}
        \beta^2(E):= \sum_{\substack{Q \in \mathcal{Q}}} \beta^2_E(3Q) \diam(Q) < \infty.
     \end{equation}
    More precisely, if $\beta^2(E)< \infty$, then there exist a curve $I \in \mathcal{I}$ so that $I \supseteq E$ and
    $\mathcal{H}^1(I) \lesssim \diam(E) + \beta^2(E)$. 
    If $I \in \mathcal{I}$ is any curve containing $E$, then $\diam(E) + \beta^2(E) \lesssim \mathcal{H}^1(I)$.
\end{theorem}

\begin{proposition}[Lipschitz images result]
\label{imagesprop}
    For any bounded Ahlfors regular $1$-set $E \su \mathbb{R}^2$, there exists $I \in \mathcal{I}$ such that $\disp \dim(E \cap I) = 1$.
\end{proposition}

\begin{proof}
    Let $E \su \R^2$ be an Ahlfors regular $1$-set with Ahlfors lower and upper constants $a$ and $b$, respectively.
    For each $n\in \N$, let $\mathcal{E}_n \su \mathcal{Q}_n$ denote the collection of dyadic cubes of side length $2^{-n}$ that intersect $E$ in a substantial set.
    That is, 
    \begin{equation}
    \label{bigQnE}
        \mathcal{E}_n = \set{q \in \mathcal{Q}_n : \mathcal{H}^1(E \cap q) \ge \frac a {9} 2^{-n}} .
    \end{equation}
    For any cube $Q \su \R^2$ (not necessarily dyadic), we define the collections 
    \begin{equation}
    \label{dyadicQ}
      \mathcal{Q}_n(Q) = \set{q \in \mathcal{Q}_n : q \su Q}  
    \end{equation}
    and 
    \begin{equation}
    \label{dyadicbigQ}
      \mathcal{E}_n(Q) = \set{q \in \mathcal{E}_n : q \su Q}  
    \end{equation}
    and note that $\mathcal{E}_n(Q) \su \mathcal{Q}_n(Q)$.
    Given $Q \in \mathcal{E}_m$, if $n>m$, as shown in Lemma \ref{cubeCountLemma}, there exist constants $c_1, C_1$, depending solely on $a$ and $b$, so that
    \begin{equation}
    \label{cubeCountBound}
      \max\set{1, c_1 2^{n-m}}  \le \# \mathcal{E}_n((1+2^{1-(n-m)})\cdot Q) \le  C_1 2^{n-m} .
    \end{equation}
    In fact, as shown in Lemma \ref{cubeCountLemma}, for any $q \in \mathcal{Q}_n\pr{Q}$ with $\Ha^1\pr{E \cap q} > 0$, either $q \in \mathcal{E}_n$, or $q$ is adjacent to a cube in $\mathcal{E}_n$. 
    The scaling of $Q$ by $(1+2^{1-(n-m)})$ in \eqref{cubeCountBound} accounts for this adjacency issue. 
    Without loss of generality, we assume that $C_1 \ge 1$.

    We now begin the construction of a $1$-dimensional set $F \su E$ that is contained in a Lipschitz image. 
    By Theorem \ref{atst}, $F$ is contained in a Lipschitz images iff $\beta^2(F) < \infty$.
    Since $\beta_{F}(3Q) \leq 1$ for any $Q$, then we see from \eqref{def:betaNumber} that it suffices to construct the set $F$ in such a way that the set of neighboring cubes $\mathcal{N} = \{ Q \in \mathcal{Q} ~:~ 3Q \cap F \neq \varnothing\}$ has the property that
    \begin{align*}
        \sum_{\substack{Q \in \mathcal{N}}} \diam(Q) < \infty.
    \end{align*}
    On the other hand, we need to have enough cubes to get a sufficiently large dimension estimate. 
    To that end, % with $\mathcal{F}_n = \{ Q \in \mathcal{Q}_n ~:~ Q \cap F \neq \varnothing\}$,
    we also ensure that
    \begin{align*}
        \liminf_{n \to \infty} \frac{\log \left(\#\{ Q \in \mathcal{Q}_n ~:~ Q \cap F \neq \varnothing\}\right)}{\log(2^n)}\geq 1,
    \end{align*}
    then we apply Proposition \ref{prop dimboundcor}.
    We use collections of squares, $\mathcal{F}_n \su \mathcal{Q}_n$, to define the set $F$.
    
    We construct the collections $\mathcal{F}_n$ by induction.
    Let $\{N_k\}_{k=0}^{\infty}$ be the increasing sequence of integers that is recursively defined by $N_0 = 0$ and, for $k \in \N$, 
    \begin{align*}
         N_{k}= N_{k-1}  + k.
    \end{align*}
    Let $\la = \frac{c_1}{10^2C_1}$, where $c_1, C_1$ appear in \eqref{cubeCountBound}.
    Choose $\mathcal{F}_1 = \mathcal{F}_{N_1} \su \mathcal{E}_1$ to be a maximal set of cubes with the property that whenever $Q_1, Q_2 \in \mathcal{F}_{1}$, it holds that $5 Q_1 \cap 5 Q_2 = \varnothing$.
    We refer to this property as the \textit{5-separation condition}.
    Let $M = \#\mathcal{F}_1$.  
    We construct the collections $\mathcal{F}_n \su \mathcal{Q}_n$ recursively and show by induction that for all $k \in \N$,  $\mathcal{F}_{N_k} \su \mathcal{E}_{N_k}$ is a collection of cubes for which the 5-separation condition holds and
    \begin{equation}
    \label{NkCount}
      \#\mathcal{F}_{N_k} \in \brac{M \la^{k-1} 2^{N_{k-1}}, M 2^{N_{k-1}}}. 
    \end{equation}
    In addition, for every $n \in (N_{k}, N_{k+1})$, we show that
    \begin{equation}
    \label{interiorCount}
      \#\mathcal{F}_{n} \leq M C_2 2^{n - k},
    \end{equation}
    where $C_2$ depends only on $a$ and $b$.
    
    Since $N_1 = 1$, then \eqref{NkCount} holds with $k = 1$ and the 5-separation condition also holds by construction.
    Assume that $\mathcal{F}_{N_k} \su \mathcal{E}_{N_k}$ has been defined so that \eqref{NkCount} holds and the 5-separation condition holds.
    For each $n \in (N_k, N_{k+1}]$, the collections $\mathcal{F}_n$ are constructed in the following order:
    we first use $\mathcal{F}_{N_k}$ to construct an auxiliary set $\mathcal{G}_{N_{k+1}-1}$, then we use $\mathcal{G}_{N_{k+1}-1}$ to define $\mathcal{F}_{N_{k+1}}$, and then we go back and use $\mathcal{F}_{N_{k+1}}$ to define $\mathcal{F}_n$ for each $n \in (N_{k}, N_{k+1})$.
    Since we are dealing with cubes at many scales, we will use $R$ and $q$ to denote cubes in $\mathcal{F}_{N_k}$ and $\mathcal{F}_{N_{k+1}}$, respectively, and we will reserve $Q$ to denote cubes in $\mathcal{Q}_n$ for $n \in (N_k, N_{k+1})$.
    
    We start with the construction of $\mathcal{G}_n$, where $n = N_{k+1} - 1$.
    Let $\mathcal{G}^0_{n} \su \mathcal{E}_{n}$ be the subcollection  defined by 
    \begin{align*}
        \mathcal{G}^0_{n} := \bigcup_{R \in \mathcal{F}_{N_k}} \mathcal{E}_{n}((1+2^{1-k})\cdot R).
    \end{align*}
    By the 5-separation condition on $\mathcal{F}_{N_k}$, this union is of disjoint sets, so the bounds from \eqref{cubeCountBound} and the inductive hypothesis described by \eqref{NkCount} show that 
    \begin{align*}
    \#\mathcal{G}^0_{n} 
    &= \sum_{R \in \mathcal{F}_{N_k}} \#\mathcal{E}_{n}((1+2^{1-k})\cdot R)
    \in \brac{M c_1 \la^{k-1} 2^{n - k}, M C_1 2^{n - k}}.
    \end{align*}
    Now define $\mathcal{G}_n \su \mathcal{G}^0_{n}$ to be a maximal subcollection of cubes with the 5-separation condition. 
    Since $\mathcal{G}_n$ is maximal and the set is in the plane, $\#\mathcal{G}_n \geq \frac{1}{10^2}\#\mathcal{G}_n^{0}$ and we see that
    \begin{align*}
        \#\mathcal{G}_{n} 
        &
        \in \brac{M c_1  10^{-2} \la^{k-1} 2^{n - k}, M C_1 2^{n - k}}.
    \end{align*}
    
    Next, we construct $\mathcal{F}_{N_{k+1}} \su \mathcal{E}_{N_{k+1}}$.
    Define $\mathcal{F}^0_{N_{k+1}}$ by choosing one cube, $q \in \mathcal{E}_{N_{k+1}}(2 Q)$, for each of the $Q \in \mathcal{G}_{N_{k+1}-1}$. 
    Since each $Q \in \mathcal{E}_{N_{k+1}-1}$, then the lower bound in \eqref{cubeCountBound} implies that such a $q = q(Q)$ always exists.
    Because $\mathcal{G}_{N_{k+1}-1}$ has the $5$-separation condition, then $Q \mapsto q(Q)$ is injective and $\#\mathcal{F}^0_{N_{k+1}} = \#\mathcal{G}_{N_{k+1}-1}$.
    Recalling that $C_1 \ge 1$, we reduce this set to $\mathcal{F}_{N_{k+1}}$ by selecting cubes so that $\#\mathcal{F}_{N_{k+1}} \le \frac 1 {C_1} \#\mathcal{F}^0_{N_{k+1}}$ and then
    \begin{align*}
        \#\mathcal{F}_{N_{k+1}} \in \brac{M \la^k 2^{N_{k}}, M 2^{N_{k}}},
    \end{align*}
    which establishes \eqref{NkCount}. 
    
    For each $q \in \mathcal{F}_{N_{k+1}}$, there exists a unique $Q \in \mathcal{G}_{N_{k+1} - 1}$ for which $q \su 2Q$, which implies that $3q \su 3Q$ and $5 q \su 5 Q$.
    As $\mathcal{G}_{N_{k+1} - 1}$ satisfies the 5-separation condition by construction, then $\mathcal{F}_{N_{k+1}}$ inherits the 5-separation condition, as required.
    And since each $Q \in \mathcal{G}_{N_{k+1} - 1}$ satisfies $Q \su \pr{1 + 2^{1-k}} \cdot R$ for some $R \in \mathcal{F}_{N_{k}}$, then $3Q \su 3 R$.
    Combining these observations shows that
    \begin{equation}
    \label{cubesContained}
      \bigcup_{q \in \mathcal{F}_{N_{k+1}}} 3 q 
      \su \bigcup_{Q \in \mathcal{G}_{N_{k+1}-1}} 3 Q 
      \su \bigcup_{R \in \mathcal{F}_{N_{k}}} 3 R.  
    \end{equation}
    Therefore, with $\disp F_k := \bigcup_{R \in\mathcal{F}_{N_{k}}} 3 R$, it holds that $F_{k+1} \su F_k$ and by the 5-separation condition, these unions are of disjoint sets.
    We note that each $F_k$ can be realized as a union of dyadic cubes in $\mathcal{Q}_{N_k}$.
    
    Finally, for $n \in (N_k, N_{k+1})$, define
        \begin{align*}
            \mathcal{F}_{n}= \bigcup_{q \in \mathcal{F}_{N_{k+1}}}\{Q \in \mathcal{Q}_n~:~  Q \cap 3 q \neq \varnothing\}.
        \end{align*} 
    That is, $\mathcal{F}_{n}$ is the smallest collection of cubes in $\mathcal{Q}_n$ that covers $F_{k+1}$.
    The observation in \eqref{cubesContained} shows that $\disp \bigcup_{Q \in \mathcal{F}_{n}} Q \su F_k = \bigcup_{R \in \mathcal{F}_{N_{k}}} 3 R$ which implies that $\disp \bigcup_{Q \in \mathcal{F}_{n}} 3 Q \su \bigcup_{R \in \mathcal{F}_{N_{k}}} 5 R$.
    Let $\disp G_k = \bigcup_{R \in \mathcal{F}_{N_{k}}} 5 R$ and note that the union is of disjoint sets by the $5$-separation condition.
    For each $R \in \mathcal{F}_{N_{k}}$, there exists $x \in E \cap R$, so Ahlfors regularity implies that $\Ha^1\pr{E \cap 5 R} \le \Ha^1\pr{E \cap B(x, 3 \sqrt 2 \cdot 2^{-N_k})} \le 3 b \sqrt 2 \cdot 2^{-N_k}$.
    Therefore,
    \begin{equation}
    \label{GkUpperBd}
        \Ha^1\pr{E \cap G_k}
        = \sum_{R \in \mathcal{F}_{N_{k}}} \Ha^1\pr{E \cap 5 R}
        \le 3 b \sqrt 2 \cdot 2^{-N_k} \# \mathcal{F}_{N_{k}}
        \le 3 M b \sqrt 2 \cdot 2^{-k},
    \end{equation}
    where we have applied the upper bound from \eqref{NkCount}.
    
    Define $\disp \mathcal{F}_{n}'= \bigcup_{q \in \mathcal{F}_{N_{k+1}}}\{Q \in \mathcal{Q}_n~:~  q \su Q \} \su \mathcal{F}_n$.
    If $Q \in \mathcal{F}_n \setminus \mathcal{F}_n'$, then $Q$ must be a neighbor to some $Q' \in \mathcal{F}_n'$.
    Since each cube has eight neighbors, then $\# \mathcal{F}_{n}' \ge \frac 1 9 \#\mathcal{F}_{n}$. 
    Since $\mathcal{F}_{N_{k+1}} \su \mathcal{E}_{N_{k+1}}$, then for each $Q \in \mathcal{F}_{n}'$, we have $\Ha^1\pr{E \cap Q} \ge \frac a {9} 2^{-N_{k+1}} > 0$. 
    Following the argument in the proof of Lemma \ref{cubeCountLemma}, either $Q \in \mathcal{E}_n$, or $Q$ has a neighbor in $\mathcal{E}_n$, so $\Ha^1\pr{E \cap 3 Q} \ge \frac a {9} 2^{-n}$ for every $Q \in \mathcal{F}'_{n}$.
    Thus, because $\disp G_k \supseteq \bigcup_{Q \in \mathcal{F}_{n}} 3 Q \supseteq \bigcup_{Q \in \mathcal{F}_{n}'} 3 Q$, we see that
    \begin{equation}
    \label{GkLowerBd}
        \Ha^1\pr{E \cap G_k}
        %\ge \Ha^1\pr{E \cap \bigcup_{Q \in \mathcal{F}_{n}} 3 Q}
        \ge \Ha^1\pr{E \cap \bigcup_{Q \in \mathcal{F}'_{n}} 3 Q} 
        \ge \frac 1 9 \sum_{Q \in \mathcal{F}'_{n}}  \Ha^1\pr{E \cap 3 Q} 
        \ge \frac a {9^2} 2^{-n} \# \mathcal{F}'_{n}
        \ge \frac a {9^3} 2^{-n} \# \mathcal{F}_{n},
    \end{equation}
    where the $\frac 1 9$ accounts for possible overlaps.
    Combining \eqref{GkUpperBd} with \eqref{GkLowerBd} shows that \eqref{interiorCount} holds with $C_2 = \frac{9^3 \sqrt 2b }a$.
    The inductive argument is complete.

    Recall that $F_{k+1} \su F_k$ and define
    $$F := \bigcap_{k=1}^{\infty} F_k = \bigcap_{k=1}^{\infty} \bigcup_{R \in \mathcal{F}_{N_{k}}} 3 R.$$
    Since $\mathcal{F}_{N_{k}} \su \mathcal{E}_{N_{k}}$ for every $k \in \N$ and
    $$E = \bigcap_{n = 1}^\iny \bigcup_{\set{Q \in \mathcal{Q}_n : Q \cap E \ne \varnothing}}  3 Q
    \supseteq \bigcap_{n = 1}^\iny \bigcup_{Q \in \mathcal{E}_n} 3Q
    \supseteq \bigcap_{k = 1}^\iny \bigcup_{Q \in \mathcal{E}_{N_k}} 3Q,$$
    then $F \su E$.
    
    Let $\mathcal{N} = \set{Q \in \mathcal{Q} : F \cap 3Q \ne \varnothing}$, the set of all dyadic cubes that are a neighbor to a cube that intersects $F$.
    Note that if $Q \in \mathcal{Q} \setminus \mathcal{N}$, then $\beta^2_F(3Q) = 0$.
    Define $\mathcal{N}_{n} := \mathcal{N} \cap \mathcal{Q}_{n}$.
    Observe that for any $k \in \N$,
    \begin{align*}
    \mathcal{N}_{N_k}
    &= \set{Q \in \mathcal{Q}_{N_k} : F \cap 3Q \ne \varnothing} 
    \su \set{Q \in \mathcal{Q}_{N_k} : F_k \cap 3Q \ne \varnothing}  
    = \set{Q \in \mathcal{Q}_{N_k} : Q \su  \bigcup_{R \in \mathcal{F}_{N_k}} 5 R}.
    \end{align*}
    In particular, $\# \mathcal{N}_{N_k} \le 25 \# \mathcal{F}_{N_k}$.
    On the other hand, if $n \in \pr{N_k, N_{k+1}}$ for some $k \in \N$, then because $\disp F_{k+1} \su \bigcup_{Q \in \mathcal{F}_n} Q$, we see that
    $$\mathcal{N}_{n}
    := \mathcal{N} \cap \mathcal{Q}_{n} 
    %= \set{Q \in \mathcal{Q}_{n} : F \cap 3Q \ne \varnothing} 
    \su \set{Q' \in \mathcal{Q}_{n} : \pr{\bigcup_{Q \in \mathcal{F}_n} Q} \cap 3Q' \ne \varnothing} 
    \su \set{Q' \in \mathcal{Q}_n : Q' \su \bigcup_{Q \in \mathcal{F}_n} 3 Q }$$
    and then $\# \mathcal{N}_n \le 9 \# \mathcal{F}_n$.
    Now we have
    \begin{align*}
        \beta^2(F) 
        &= \sum_{\substack{Q \in \mathcal{Q}}} \be_{F}^2(3Q) \diam(Q)
        \le\sum_{\substack{Q \in \mathcal{N}}}\diam(Q) 
        = \sum_{k=1}^{\infty} \sum_{n = N_k}^{N_{k+1} - 1} \sum_{Q \in \mathcal{N}_n} \diam(Q) \\
        &= \sum_{k=1}^{\infty} \sum_{n = N_k}^{N_{k+1} - 1} \# \mathcal{N}_n 2^{-n}
        \le 25 \sum_{k=1}^{\infty} \# \mathcal{F}_{N_k} 2^{-N_k}
        + 9 \sum_{k=1}^{\infty} \sum_{n = N_k+1}^{N_{k+1} - 1} \# \mathcal{F}_n 2^{-n}
        \\
        &\le 25 \sum_{k=1}^{\infty} M 2^{-k}
        + 9 \sum_{k=1}^{\infty} \sum_{n = N_k+1}^{N_{k+1} - 1} M C_2 2^{-k}
        = M \sum_{k=1}^{\infty} \pr{25 + 9 C_2 k} 2^{-k} 
        < \iny.
    \end{align*}
    Therefore, since $F \su E$ and $E$ is bounded, Theorem \ref{atst} implies that there exists $I \in \mathcal{I}$ such that $F \su I$.  
    
    By construction, $\set{F_k}_{k=1}^\iny$ is nested and each $F_k$ is a union of $\# \mathcal{F}_{N_k}$ disjoint cubes with diameter $3 \sqrt 2 \cdot 2^{- N_k}$.
    Thus, we may use Proposition \ref{prop dimboundcor} to estimate the dimension of $F$.
    Since $N_k = N_{k-1} + k = \frac{k(k+1)}{2}$ and $\# \mathcal{F}_{N_{k-1}} \ge M \la^{k-2} 2^{N_{k-2}}$, then 
    \begin{align*}
         \liminf_{k \to \infty}  \frac{\log \left(\#\mathcal{F}_{N_{k-1}} \right)}{\log\pr{\frac{2^{N_k}}{3\sqrt{2}}}}
         &\geq \liminf_{k \to \infty} \frac{\log M + (k-2) \log \la + N_{k-2} \log 2}{N_k \log 2 - \log(3\sqrt{2})} \\
         &= 1 + \liminf_{k \to \infty} \frac{k \log\pr{\frac \la 4} + \log\pr{\frac{6M\sqrt{2}}{\la^2}}}{\frac{k(k+1)}{2} \log 2 - \log(3\sqrt{2})}
          = 1
    \end{align*}
    and it follows that $\dim(F) \geq 1$.  
    Since $F \su E$ and $F \su I$, then 
    $$1 \le \dim(F) = \dim\pr{F \cap I} \le \dim\pr{E \cap I} \le \dim\pr{E} = 1,$$
    and we conclude that $\dim\pr{E \cap I} = 1$, as required.
\end{proof}

\begin{lemma}[Cube count lemma]
\label{cubeCountLemma}
    Let $E \su \R^2$ be an Ahlfors regular $1$-set with Ahlfors lower and upper constants $a$ and $b$, respectively.
    Define $\mathcal{E}_n$ and $\mathcal{E}_n(Q)$ as in \eqref{bigQnE} and \eqref{dyadicbigQ}, respectively.
    There exist constants $c_1, C_1$, depending only on $a$ and $b$, so that for any $Q \in \mathcal{E}_m$ and any $n>m$, it holds that
    \begin{align*}
      \max\set{1, c_1 2^{n-m}}  \le \# \mathcal{E}_n((1+2^{1-(n-m)})\cdot Q) \le  C_1 2^{n-m}.
    \end{align*}
\end{lemma}

\begin{proof}
    Let $Q \in \mathcal{E}_m$.
    Since $E \cap Q \ne \varnothing$, then there exists an $x \in E \cap Q$, and for any such $x$, it holds that $\pr{1 + 2^{1-(n-m)}} \cdot Q \su B\pr{x, \frac 3 {\sqrt 2} 2^{-m}}$.
    Ahlfors regularity implies that
    \begin{align*}
        \frac {3b} {\sqrt 2} 2^{-m}
        &\ge \Ha^1\pr{E \cap \pr{1 + 2^{1-(n-m)}} \cdot Q}
        = \sum_{q \in \mathcal{Q}_n\pr{\pr{1 + 2^{1-(n-m)}} \cdot Q}} \Ha^1\pr{E \cap q} \\
        &\ge \sum_{q \in \mathcal{E}_n\pr{\pr{1 + 2^{1-(n-m)}} \cdot Q}} \Ha^1\pr{E \cap q}
        \ge \# \mathcal{E}_n\pr{\pr{1 + 2^{1-(n-m)}} \cdot Q} \frac{a}{9} 2^{-n}
    \end{align*}
   which gives the upper bound with $C_1 = \frac{27 b }{\sqrt 2 a}$.
    
    If $q \in \mathcal{Q}_n\pr{\pr{1 + 2^{1-(n-m)}} \cdot Q}$ (as defined in \eqref{dyadicQ}) and $\Ha^1\pr{E \cap q} > 0$, then either $q$ is in $\mathcal{E}_n\pr{\pr{1 + 2^{1-(n-m)}} \cdot Q}$ or $q \in \mathcal{E}'_n\pr{\pr{1 + 2^{1-(n-m)}} \cdot Q}$, where 
    $$\mathcal{E}'_n\pr{Q} = \set{q \in \mathcal{Q}_n\pr{Q} : \Ha^1\pr{E \cap q} \in \pr{0, \frac{a}{9} 2^{-n}}}.$$
    If $q \in \mathcal{E}'_n\pr{Q}$, then for any $x \in E \cap q$, $\disp B(x, 2^{-n}) \su q \cup \bigcup_{k = 1}^8 q_k$, where each $q_k$ denotes one of the eight neighbors of $q$.
    Therefore, Ahlfors lower regularity shows that
    \begin{align*}
        a 2^{-n}
        &\le \Ha^1\pr{E \cap B(x, 2^{-n})}
        \le \Ha^1\pr{E \cap \pr{q \cup \bigcup_{k = 1}^8 q_k}}
        < \frac a {9} 2^{-n} + \sum_{k = 1}^8 \Ha^1\pr{E \cap q_k}
    \end{align*}
    from which it follows from pigeonholing that $\Ha^1\pr{E \cap q_k} > \frac a {9} \ell(q)$ for some $k$.
    That is, for every $q \in \mathcal{E}'_n\pr{Q}$, there exists a neighbor $q' \in \mathcal{E}_n\pr{\pr{1 + 2^{1-(n-m)}} \cdot Q}$.
    Since $\Ha^1\pr{E \cap Q} > 0$, then we must have that $\# \mathcal{E}_n(\pr{1 + 2^{1-(n-m)}} \cdot Q) \ge 1$.
    Since each $q'$ can be a neighbor to at most eight cubes, then $\# \mathcal{E}'_n(Q) \le 8 \# \mathcal{E}_n(\pr{1 + 2^{1-(n-m)}} \cdot Q)$.
    It follows that
    \begin{align*}
        \frac{a}{9} 2^{-m}
        &\le \Ha^1\pr{E \cap Q}
        \le \sum_{q \in \mathcal{E}'_n(Q)} \Ha^1\pr{E \cap q}
        + \sum_{q \in \mathcal{E}_n(Q)} \Ha^1\pr{E \cap q} \\
        &< \sum_{q \in \mathcal{E}'_n(Q)} \frac{a}{9} 2^{-n}
        + \sum_{q \in \mathcal{E}_n(Q)} \Ha^1\pr{E \cap B\pr{x_q, 2^{-n + \frac 1 2}} } \\
        &\le \# \mathcal{E}_n\pr{\pr{1 + 2^{1-(n-m)}} \cdot Q} \pr{ \frac{8 a}{9} 
        +  \sqrt 2 b} 2^{-n}
    \end{align*}
    which gives the lower bound with $c_1 = \frac a { 8 a 
        +  9 \sqrt 2 b}$.   
\end{proof}

The remainder of the article will be dedicated to answering Question \ref{graphsQ}, that is, to understanding whether a Lipschitz \emph{graph} sees a high-dimensional subset of a purely unrectifiable $1$-set. 
Our results are summarized in the statement below.

\begin{theorem}[Lipschitz graph theorem]
\label{mainthm}
    Let $C \su \mathbb{R}^2$ be the $1$-dimensional attractor of an iterated function system that satisfies the open set condition. 
    There exist constants $b_1, b_2$ such that for every $\varepsilon>0$, there exists a Lipschitz graph $\Gamma$ that satisfies 
        \begin{align*}
            \dim( C \cap \Gamma) \geq 1-\varepsilon
        \end{align*}
    and
    \begin{align*}
        \Lip(\Gamma) \leq b_1\exp(b_2\varepsilon^{-1}\log(\varepsilon^{-1})).
    \end{align*}
\end{theorem}

We establish this result in four settings below.
The first two settings are different versions of the theorem in the case where $C = \mathcal{C}_4$, the 4-corner Cantor set; see Propositions \ref{construction1} and \ref{construction2}.
Note that in those two cases, the bound on the Lipschitz constant is much smaller, as shown in the table below.
Next, in Theorem \ref{rotFreeGraph}, we establish a version of this result for attractors of rotation-free iterated function systems.
Finally, Theorem \ref{rotationalCase} applies to attractors of general iterated function systems.

The following table summarizes the four results described by Theorem \ref{mainthm}.
The bounds on the Lipschitz constants appear in the third column. 
Here we use $c$ to denote a universal constant.
The constant $c_0$ depends on the number of maps in the iterated function system (denoted by $M$ in the fourth row) and the largest scale factor (which is $\max\set{r_1, \ldots, r_M}$ using the notation in the fourth row).
Finally, if $K$ denotes the convex hull of the fixed points of the IFS, then $\disp \nu = \inf_{\te \in \mathbb{S}^1} \abs{P_\te\pr{K}}$.

\begin{center}
\renewcommand{\arraystretch}{1.7}
\captionsetup{type=table}
\begin{tabular}{|l|l|l|}
\hline
Result & Set Type & Lipschitz constants bound \\
\hline
Propositions \ref{construction1} & $\mathcal{C}_4$, ad hoc construction & $c \eps^{-2}$ \\ 
Propositions \ref{construction2} & $\mathcal{C}_4$, generic construction & $c 2^{\frac 1 \eps}$ \\ 
Theorem \ref{rotFreeGraph} & attractor of rotation-free IFS & $\disp \frac{\diam(K)}{\nu} \exp\brac{c_0 \eps^{-1} \log\pr{\eps^{-1}}}$,  \\
Theorem \ref{rotationalCase} & attractor of rotational IFS & $\disp \frac{\diam(K)}{\prod_{\substack{k=1 \\ \ }}^M r_k}  \max\set{\frac 1 \nu, 1}\exp\brac{20 M \eps^{-1} \log\pr{\eps^{-1}}}$  \\

\hline
\end{tabular} 
\captionof{table}{A summary of Theorem \ref{mainthm}.}
\label{tableofResults}
\end{center}

The proofs of Theorems \ref{rotFreeGraph} and \ref{rotationalCase} bear many similarities, but that of Theorem \ref{rotationalCase} is arguably more complex.
A natural question to ask is whether the constant derived in the proof of Theorem \ref{rotationalCase} is ever smaller than the one from Theorem \ref{rotFreeGraph}.
That is, if we are working with a rotation-free IFS, is there ever a reason to apply the more complicated theorem?
To answer this question, we examine the relationship between $c_0$ and $20 M$.

As shown in Theorem \ref{rotFreeGraph}, after conflating notation, $c_0 = \log\pr{4M r_M^{-1}} \max\set{1, \frac{3}{\log\pr{r_M^{-1}}}}$.
If $c_0 = \log\pr{4M r_M^{-1}}$, then because $r_M \ge \frac 1 M$, we see that $c_0 \le \log\pr{4 M^2} < 20 M$ since $M \ge 3$.
In this case, the rotation-free approach wins.
On the other hand, if $c_0 = 3\frac{\log\pr{4 M r_M^{-1}} }{\log\pr{r_M^{-1}}} = 3\pr{1 + \frac{\log\pr{4 M}}{\log\pr{r_M^{-1}}}}$, i.e. $\log\pr{r_M^{-1}} \le 3$, then
\begin{align*}
    c_0 \ge 20 M
    & \iff
    \log\pr{4 M} \ge \pr{\frac{20}{3} M - 1} \log\pr{r_M^{-1}}
    \iff
    r_M \ge \exp\pr{- \frac{\log \pr{4M}}{\frac{20}{3} M - 1}}.
\end{align*}
In other words, the constant from Theorem \ref{rotationalCase} is smaller than the one from Theorem \ref{rotFreeGraph} if we are in the case where the largest scale factor is very close to $1$.
Given that the first step in the proof of Theorem \ref{rotationalCase} is a uniformization process, it makes sense that when the scale factors have a lot of variability, then the uniformization step is useful.

\section{Preliminaries on Iterated Function Systems} \label{Section:prelim}

Here we collect a number of definitions and results that we use below in our graph constructions.

\begin{definition}[IFS]
    An \textbf{iterated function system (IFS)}, $\disp \{f_j\}_{j=1}^N$, is a finite family of maps $f_j \colon \mathbb{R}^2 \to \mathbb{R}^2$ of the form 
    \begin{equation}
    \label{IFSDef}
        f_j(x) = r_j A_jx+ z_j,
    \end{equation}
    where $r_j \in (0, 1)$, $A_j \in O(2)$ is a rotation matrix, and $z_j \in \mathbb{R}^2$. 
    We follow the convention that
    \begin{equation*}
  %  \label{rjOrdered}
        0 < r_1 \le \ldots \le r_N < 1.
    \end{equation*}
    If $r_1 = \ldots = r_N =: r$ and $A_1 = \ldots = A_N =: A$, then we call $\disp \{f_j\}_{j=1}^N$ a \textbf{uniform iterated function system (UIFS)}, we call $r$ the \textbf{scale factor}, and we call $A$ the \textbf{rotation matrix}.
    The unique compact set $C \su \R^2$ satisfying
        \begin{align*}
 %       \label{CantorSetDef}
            C= \bigcup_{j=1}^N f_j(C)
        \end{align*}
    is called the \textbf{attractor} of the system, $\{f_j\}_{j=1}^N$. 
    If $A_j = \operatorname{Id}$ for all $j = 1, \dots, N$, then we say that the IFS is \textbf{rotation-free}.
    The unique positive real number $s$ for which $\disp \sum_{j=1}^N r_j^s = 1$ is called the \textbf{similarity dimension} of the IFS.
\end{definition}

For $E\su \R^d$, $\conv(E)$ denotes its convex hull, defined as
$$\conv\pr{E} := \set{\sum_{k=1}^n \lambda_k x_k  :n \in \N, x_k \in E, \la_k \in \brac{0,1}, \sum_{k=1}^n \la_k=1}.$$
Note that if $E$ is compact, then so is $\conv(E)$.

If $K \su \R^2$ is convex and compact, then the image of the map $\zeta : \brac{0, \pi} \to \R_+$ defined by $\zeta\pr{\te} = \Ha^1(P_\te(K))$ is a closed interval.
This allows us to introduce the following non-degeneracy conditions.

\begin{definition}[Non-degeneracy]
A compact, convex set $K \su \R^2$ is $\nu$\textbf{-non-degenerate}, or simply \textbf{non-degenerate}, if 
\begin{equation}
\label{infProj}
     \nu =\nu(K):= \inf \set{|P_\te(K)| : \te \in \mathbb{S}^1} > 0.
 \end{equation}
 We say that an IFS $\{f_j\}_{j=1}^N$ with attractor $C$ is \textbf{$\nu$-non-degenerate} if $\conv(C)$ is $\nu$-non-degenerate.
\end{definition}

\begin{remark}
    An IFS is non-degenerate if and only if its attractor is not contained in a line. 
    If the attractor is contained in a line, then we already know how to construct a Lipschitz graph that intersects the attractor in a high-dimensional set. 
    Therefore, we will focus on non-degenerate iterated function systems.
\end{remark}

\begin{lemma}[John Ellipsoid]
\label{lem innerball}
If $K \su \mathbb{R}^d$ is a compact, convex, $\nu$-non-degenerate set, then there exists a ball, $B$, of diameter $\frac \nu d$ such that $B \su K$.
\end{lemma}

\begin{proof}
    Condition \eqref{infProj} implies that $K$ has nonempty interior. 
    Therefore, John's Theorem \cite{john1948extremum} implies that there exists an ellipsoid, $\mathcal{E}$, with center $x$ such that 
        \begin{align*}
            \tfrac{1}{d}\cdot (\mathcal{E}-x)+x \su K \su \mathcal{E}.
        \end{align*}
    Since $K \su \mathcal{E}$, then $\inf \set{|P_\te(\mathcal{E})| : \te \in \mathbb{S}^{d-1}} \geq \nu$.
    The symmetry of $\mathcal{E}$ implies that there is a ball, $B_0$, with diameter at least $\nu$ and center $x$ contained in $\mathcal{E}$. Therefore,
        \begin{equation*}
            B:=\tfrac{1}{d}\cdot (B_0-x)+x \su \tfrac{1}{d}\cdot (\mathcal{E}-x)+x \su K. \qedhere
        \end{equation*}      
\end{proof}
 
As is often done with well-known fractals (e.g. the Sierpi\'nski triangle, the 4-corner Cantor set), it will be helpful to realize the attractors as an infinite intersection of a collection of nested compact sets.
To define this nested collection, we need an initial set, $K$.
For the Sierpi\'nski triangle, $K$ is a triangle, while for the 4-corner Cantor set, $K$ is the unit square.

Given an IFS $\set{f_j}_{j = 1}^N$, for any $n \in \N$, we write $j^{(n)} = \pr{j_1, j_2, \ldots, j_n} \in \set{1, \ldots , N}^n$ to denote an $n$-sequence of elements in $\set{1,\ldots ,N}$.
We use this notation to describe iterated functions; that is,
\begin{equation}
\label{iteratedFunc}
  f_{j^{(n)}} = f_{j_1} \circ f_{j_2} \circ \ldots \circ f_{j_n}.
\end{equation}

\begin{definition}[Generations of an IFS]\label{def generation}
    Given an IFS $\set{f_j}_{j = 1}^N$ and a convex, compact set $K$, we define the \textbf{generations of $\set{f_j}_{j = 1}^N$ with respect to $K$} as $\set{C_n}_{n=0}^\iny$, where $C_0 = K$ and for each $n \in \N$,
    $$C_n
    = \bigcup_{j^{(n)}} f_{j^{(n)}}(K)
    =: \bigcup_{j^{(n)}} K_{j^{(n)}}.$$
    We call $K$ the \textbf{initial set} and each $C_n$ is the \textbf{$n^{th}$ generation}.
\end{definition}

Recall that for any choice of initial set $K$, the generations $\set{C_n}_{n=0}^\iny$ converge to $C$ in the Hausdorff metric as $n \to \infty$; see \cite[Theorem 8.3]{falconer1986}, for example.
Therefore, we could take $K$ to be any set.
However, as we will see, it will be convenient to choose $K$ to be optimal in the following ways:
\begin{enumerate}
    \item 
    \label{C0Large}
    The initial set $K$ is large enough for the generations $\set{C_n}_{n = 0}^\iny$ to be nested.
    In this case, $\disp C = \bigcap_{n =0}^{\iny} C_n$.
    \item 
    \label{C0Small}
    The initial set $K$ is small enough for each $C_n$ to have a uniformly bounded number of overlaps.
\end{enumerate}
 
The next result shows that $C$ can be realized as the infinite intersection of iterations of functions in the IFS applied to $\conv(C)$.
In particular, this lemma shows that $\conv(C)$ is large enough for its images to be nested, so it satisfies the condition \ref{C0Large} above.

\begin{lemma}[Nested generations]
\label{generLemma}
    Let $\{f_j\}_{j=1}^N$ be an IFS with attractor $C$ and let $\set{C_n}_{n=0}^\iny$ denote the generations of $\{f_j\}_{j=1}^N$ with respect to $\conv(C)$.
    Then $\set{C_n}_{n=0}^\iny$ is nested and $\disp C = \bigcap_{n=0}^{\infty} C_n$.
\end{lemma}

\begin{proof}
For each $n \in \N$, observe that $C_n$ is recursively defined as
\begin{equation*}
%\label{CjRecursDef}
C_n = \bigcup_{j=1}^N f_j\pr{C_{n-1}}.
\end{equation*}
To show that $\set{C_n}_{n=0}^\iny$ is indeed nested, it suffices to show that $C_1 \su C_0$, then appeal to self-similarity.

For any $x_0 \in C_0 = \conv(C)$, there exists $n \in \N$, $x_1, \ldots, x_n \in C$, $\la_1, \ldots, \la_n \in [0,1]$ with $\disp \sum_{k=1}^n \la_k = 1$ so that $\disp x_0 = \sum_{k=1}^n \lambda_k x_k$.
Since $f_j(x) = r_j A_j x + z_j$ is affine, then 
\begin{equation*}
%\label{Kcontainment}
f_j\pr{x_0} 
= r_j A_j \pr{\sum_{k=1}^n \la_k x_k} + z_j
= \sum_{k=1}^n \la_k  \pr{r_j A_j x_k + z_j}
= \sum_{k=1}^n \la_k f_j\pr{ x_k}  \in \conv\pr{f_j(C)}
\end{equation*}
which implies that $f_j(C_0) \su \conv\pr{f_j(C)}$.
Therefore,
\begin{align*}
C_1 
&= \bigcup_{j=1}^N f_j\pr{C_0}
\su \bigcup_{j=1}^N \conv\pr{f_j(C)}
\su \conv\pr{\bigcup_{j=1}^N f_j(C)}
= \conv\pr{C} 
= C_0,
\end{align*}
so we conclude that $\set{C_n}_{n = 0}^\iny$ is nested.
It then follows from \cite[Theorem 2.6]{falconer1997} that $\disp C = \bigcap_{n = 0}^\iny C_n$, as required.
\end{proof}

We may apply the previous result to any sub-IFS to get a nested collection of generations. 
For any $m \in \N$, we use the notation 
\begin{equation}
    \label{IFSiter}
    \set{f_{j^{(m)}} } = \set{f_{j^{(m)}} : j^{(m)} \in \set{1, \ldots, N}^m}
\end{equation}
to denote the $m^{th}$ iteration of the IFS $\set{f_j}_{j=1}^N$. 
The following corollary follows directly from the same argument above and the fact that a sub-IFS has strong enough self-similarity conditions to allow for the same proof.

\begin{corollary}[Nested sub-generations]
\label{generCor}
    Let $\{f_j\}_{j=1}^N$ be an IFS with attractor $C$.
    If $\set{\vp_\ell}_{\ell = 1}^M \su \set{f_{j^{(m)}}}$ for some $m \in \N$ and $\set{E_n}_{n=0}^\iny$ denotes the generations of $\set{\vp_\ell}_{\ell = 1}^M$ with respect to $\conv(C)$, then $\set{E_n}_{n=0}^\iny$ is nested and $\disp E = \bigcap_{n=0}^{\infty} E_n$ is the attractor for $\set{\vp_\ell}_{\ell = 1}^M$.
\end{corollary}

\begin{proof}
    For each $n \in \N$, $E_n$ is defined by $\disp E_n = \bigcup_{\ell=1}^M \varphi_{\ell}\pr{E_{n-1}}$.
    By arguments analogous to the previous proof, the result follows.
\end{proof}

We observe that if the IFS is rotation-free with attractor $C$, then $\conv(C)$ is generated by the $N$ fixed points of the IFS alone, as opposed to the (closure of the) fixed points of all of its iterations.
In other words, when the IFS is rotation-free, $\conv(C)$ is a polygon.

\begin{lemma}[Polygonal convex hull]
\label{polygonLemma}
    If $\set{f_j}_{j = 1}^N$ is a rotation-free IFS with fixed points $\set{x_j}_{j=1}^N$ and attractor $C$, then $\conv(C) = \conv\set{x_j : j = 1, \ldots, N}$.
\end{lemma}

\begin{proof}
Since $f_j(x) = r_j x + z_j$, then $f_j(x_j) = x_j$ if and only if $\disp x_j = \frac{z_j}{1 - r_j}$.
Fix some $m \in \N$ and observe that
\begin{align*}
f_{j^{(m)}}(x)
&= f_{j_1} \circ \ldots \circ f_{j_{m-1}}(r_{j_m} x + z_{j_m}) \\
&= f_{j_1} \circ \ldots \circ f_{j_{m-2}}(r_{j_{m-1}}\pr{r_{j_m} x + z_{j_m}} + z_{j_{m-1}}) \\
&= r_{j_1} r_{j_2} \ldots r_{j_m} x 
+ r_{j_1} \ldots r_{j_{m-1}} z_{j_m}  
+ r_{j_1} \ldots r_{j_{m-2}} z_{j_{m-1}}  
+ \ldots 
+ r_{j_1} z_{j_{2}} 
+ z_{j_1} \\
&= \pr{\prod_{\ell = 1}^m r_{j_\ell}} x
+ \sum_{k=1}^m  \pr{\prod_{\ell = 1}^{k-1} r_{j_\ell}} z_{j_k} 
\end{align*}
Therefore, $f_{j^{(m)}}(x) = x$ if and only if $\disp x \pr{1 - \prod_{\ell = 1}^m r_{j_\ell}} = \sum_{k=1}^m  \pr{\prod_{\ell = 1}^{k-1} r_{j_\ell}} z_{j_k} $ or
\begin{align*}
x 
&= \sum_{k=1}^m \pr{\frac{\prod_{\ell = 1}^{k-1} r_{j_\ell}}{1 - \prod_{\ell = 1}^m r_{j_\ell}}} z_{j_k} 
= \sum_{k=1}^m \brac{\frac{\pr{1 - r_{j_k}}\prod_{\ell = 1}^{k-1} r_{j_\ell}}{1 - \prod_{\ell = 1}^m r_{j_\ell}}} x_{j_k} .
\end{align*}
Since
\begin{align*}
\sum_{k=1}^m \pr{1 - r_{j_k}}\prod_{\ell = 1}^{k-1} r_{j_\ell}
&= \sum_{k=1}^m \prod_{\ell = 1}^{k-1} r_{j_\ell}
- \sum_{k=1}^m \prod_{\ell = 1}^{k} r_{j_\ell}
= 1 - \prod_{\ell = 1}^m r_{j_\ell},
\end{align*}
then we see that $x$, the fixed point of any $f_{j^{(m)}}$, is a convex combination of the fixed points $\set{x_j}_{i=1}^N$.
As $C$ is the closure of the fixed points of the set of iterations of the IFS, $\set{f_{j^{(m)}} : m \in \N}$, then the conclusion follows.
\end{proof}

A useful property of an IFS is the open set condition.

\begin{definition}[Open Set Condition]
    An iterated function system, $\{f_j\}_{j=1}^N$, satisfies the \textbf{open set condition} if there exists a non-empty open set $U \su \mathbb{R}^2$ such that
    \begin{align*}
        \bigcup_{j=1}^N f_j(U) \su U
    \end{align*}
    and 
    \begin{align*}
        f_j(U) \cap f_i(U) = \varnothing \quad \text{ for all } j \ne i.
    \end{align*} 
    If $U \cap C \ne \varnothing$, where $C$ is the attractor, then we say that the IFS satisfies the \textbf{strong open set condition}.
\end{definition}

Recall from \eqref{IFSiter} that we use the notation $\set{f_{j^{(m)}}}$ to denote the $m^{th}$ iteration of the IFS $\set{f_j}_{j=1}^N$.

\begin{lemma}[Open set inheritance]
\label{iterationOpenLemma}
    If $\{f_j\}_{j=1}^N$ satisfies the open set condition with the open set $U$, then for every $m \in \N$, $\set{f_{j^{(m)}}}$ also satisfies the open set condition with $U$.
\end{lemma} 

\begin{proof}
We establish this result by induction on $m$.
The case of $m = 1$ is immediate.
Assume that for all $m=1, \ldots, k$, we have
\begin{equation}
\label{mContainment}
  \bigcup_{j^{(m)}} f_{j^{(m)}}(U) \su U  
\end{equation}
and that for all $j^{(m)} \ne i^{(m)}$, it holds that
\begin{equation}
\label{mDisjoint}
  f_{j^{(m)}}(U) \cap f_{i^{(m)}}(U) = \varnothing.  
\end{equation}
Observe that
\begin{align*}
  \bigcup_{j^{(k+1)}} f_{j^{(k+1)}}(U)  
  &= \bigcup_{j^{(k)}} \bigcup_{i=1}^M f_{j^{(k)}}(f_i(U))  
  = \bigcup_{j^{(k)}} f_{j^{(k)}}\pr{\bigcup_{i=1}^M  f_i(U)}
  \su \bigcup_{j^{(k)}} f_{j^{(k)}}\pr{U}
  \su U,
\end{align*}
where we have used \eqref{mContainment} with $m = 1$ and then with $m = k$.
Property \eqref{mContainment} with $m = k+1$ has been shown.

Now we show the disjointness property described by \eqref{mDisjoint}.
Let $j^{(k+1)} \ne i^{(k+1)}$. 
If $j_1 = i_1$, then $\pr{j_2, \ldots, j_{k+1}} \ne \pr{i_2, \ldots, i_{k+1}}$, so by \eqref{mDisjoint} with $m = k$, we see that $f_{j_2} \circ \ldots \circ f_{j_{k+1}}(U) \cap f_{i_2} \circ \ldots \circ f_{i_{k+1}}(U) = \varnothing$.
As this condition holds under linear transformations, applying $f_{j_1}$ then shows that $f_{j^{(k+1)}}(U) \cap f_{i^{(k+1)}}(U) = \varnothing$ as well. 
If $j_1 \ne i_1$, since $f_{j_2} \circ \ldots \circ f_{j_{k+1}}(U) \su U$ and $f_{i_2} \circ \ldots \circ f_{i_{k+1}}(U) \su U$, then $f_{j^{(k+1)}}(U) \su f_{j_1}(U)$, $f_{i^{(k+1)}}(U) \su f_{i_1}(U)$.
In particular, $f_{j^{(k+1)}}(U) \cap f_{i^{(k+1)}}(U) \su f_{j_1}(U) \cap f_{i_1}(U) = \varnothing$, where we have used \eqref{mDisjoint} with $m = 1$.
In both cases, we have shown that \eqref{mDisjoint} holds for $m = k+1$, completing the proof.
\end{proof}

In the following proposition, we show that under the open set condition, we can bound the number of overlaps in the generations of an IFS with respect to the convex hull of its attractor.
In particular, this shows that $\conv(C)$ is small enough to satisfy the condition \ref{C0Small} from above.

\begin{proposition}[Separation properties]
\label{KAlmostDisjoint}
    Let $\{f_j\}_{j=1}^N$ be an IFS with attractor $C$ and generations $\set{C_n}_{n=1}^\iny$ with respect to $K = \conv(C)$.
    If $\{f_j\}_{j=1}^N$ satisfies the open set condition, then there exists $p \in \N \cup \set{0}$ so that for every $n \in \N$, $\disp C_{n} = \bigcup_{\ell=1}^{N^p} S_{\ell, n}$, where each $S_{\ell, n}$ is a union of convex sets with disjoint interiors.
    If $n \ge p$, then each $S_{\ell, n}$ contains exactly $N^{n-p}$ convex sets and therefore $C_{n}$ contains at least $N^{n-p}$ convex sets whose interiors are disjoint.
\end{proposition}

\begin{proof}
According to \cite[Theorem 2.2]{schief1994separation}, the open set condition is equivalent to the strong open set condition.
Let $U$ be the open set from the strong open set condition and let $x \in C \cap U$.
Since $U$ is open, there exists $\eps > 0$ so that $B(x, \eps) \su U$.
As $\disp \bigcap_{n=0}^\iny C_n = C$, then for every $n \in \N$, there exists some $j^{(n)} \in \set{1, \ldots, N}^n$ for which $x \in K_{j^{(n)}}$, one of the connected components of $C_n$. 
Since $\diam(K_{j^{(n)}}) = r_{j_1} \ldots r_{j_n} \diam(K)$, then $\disp \lim_{n \to \iny} \diam(K_{j^{(n)}}) = 0$.
Therefore, for $n$ sufficiently large, $K_{j^{(n)}} \su B(x, \eps) \su U$.
Let $p \in \N \cup \set{0}$ be the smallest number for which there exists some $j^{(p)}$ such that $\text{int}\pr{K_{j^{(p)}}} \su U$.

If $p = 0$, then $\text{int}(K) \su U$, so the open set condition implies that for every $n \in \N$, $\set{K_{j^{(n)}}}$ is a collection of $N^n$ sets with disjoint interiors.
In particular, $S_{1,n} = C_n$.
Assume $p \in \N$.
To ease notation, let $\set{\vp_\ell}_{\ell = 1}^{N^p} = \set{f_{k^{(p)}}}$ and assume that $K_{j^{(p)}} = \vp_1(K)$.
For each $\ell = 1, \ldots, N^p$, set $O_\ell = \text{int}\pr{\vp_\ell(K)}$.
By Lemma \ref{iterationOpenLemma}, for each $m \in \N$, $\set{f_{j^{(m)}}}$ satisfies the open set condition with $U$, and hence with $O_1 \su U$.
 
In particular, each collection $\set{f_{j^{(m)}}(O_1): j^{(m)} \in \set{1, \ldots , N}^m}$ consists of $N^m$ disjoint convex open sets.
Since $O_\ell = \vp_\ell(\text{int}(K)) = \vp_\ell \circ \vp_1^{-1}(O_1)$, then each $O_\ell$ is a translated, rotated and rescaled copy of $O_1$.
Therefore, for any $\ell \in \set{1, \ldots, N^p}$, we may define the open sets $E_{\ell,0} = O_\ell$ and for $m \in \N$,
\begin{align*}
  & E_{\ell, m} = \bigcup_{j^{(m)} } {f_{j^{(m)}}(O_\ell)} .
\end{align*}
By construction, each $E_{\ell, m}$ is a union of $N^m$ disjoint open sets.
Notice that $S_{\ell, p} := \overline{E_{\ell, 0}} = \vp_\ell(K)$ and
\begin{align*}
  S_{\ell, p+m}
  &:= \overline{E_{\ell, m}} 
  = \bigcup_{j^{(m)} } \overline{f_{j^{(m)}}(O_\ell)} 
  = \bigcup_{j^{(m)} } f_{j^{(m)}}\pr{\overline{O_\ell}}
  = \bigcup_{j^{(m)} } f_{j^{(m)}}\pr{\vp_\ell(K)}.
\end{align*}
Therefore, for every $n \ge p$, with $m = n - p$,
\begin{align*}
  C_n
  &= C_{p+m}
  = \bigcup_{i^{(p+m)}} K_{i^{(p+m)}}
  =\bigcup_{k^{(p)}}  \bigcup_{j^{(m)}} f_{j^{(m)}}\pr{ f_{k^{(p)}}(K)}
  = \bigcup_{\ell=1}^{N^p} \bigcup_{j^{(m)}} f_{j^{(m)}} \pr{\vp_\ell(K)}
  = \bigcup_{\ell=1}^{N^p} S_{\ell, n}.
\end{align*}
If $n < p$, then we may take each $S_{\ell,n}$ to contain a single component for all $\ell \le N^n$ and $S_{\ell,n} = \varnothing$ for all $\ell > N^n$.
\end{proof}

\begin{definition}[Overlapping index of an IFS]
\label{oIndex}
    Let $\{f_j\}_{j=1}^N$ be an IFS that satisfies the open set condition and let $p \in \N \cup \set{0}$ be the smallest integer for which Proposition \ref{KAlmostDisjoint} holds.
    We define the \textbf{overlapping index of the IFS} to be $\tau = N^p$.
\end{definition}

As defined, the overlapping index is an inherent property of the IFS.
In practice, we also need a notion of overlapping index associated to some choice of generations.

\begin{definition}[Overlapping index of generations]
\label{oIndexGen}
    Let $\{f_j\}_{j=1}^N$ be an IFS with generations $\set{C_n}_{n = 0}^\iny$.
    If there exists $\tau \in \N$ so that for every $n \in \N$, we have $\disp C_n = \bigcup_{\ell=1}^\tau S_{\ell,n}$, where each $S_{\ell, n}$ is a union of convex sets with disjoint interiors, then we call the smallest such $\tau$ the \textbf{overlapping index of the generations} $\set{C_n}_{n=0}^\iny$.
\end{definition}

We have the following consequence.

\begin{corollary}[Inheritance of overlapping index]
\label{overIterationCor}
    Let $\{f_j\}_{j=1}^N$ be an IFS with attractor $C$ and overlapping index $\tau$.
    For every $\set{\vp_\ell}_{\ell = 1}^M \su \set{f_{j^{(m)}}}$, the generations $\set{E_n}_{n=0}^\iny$ with respect to $\conv(C)$ have an overlapping index that is at most $\tau$.
\end{corollary}

\begin{proof}
    If $\set{C_n}_{n=0}^\iny$ denotes the generations of $\{f_j\}_{j=1}^N$ with respect to $\conv(C)$, then there exists $p \in \N \cup \set{0}$ so that $\tau = N^p$ is the overlapping index for the generations $\set{C_n}_{n=0}^\iny$.
    Therefore, for each $n \in \N$, $\disp C_n = \bigcup_{\ell=1}^\tau S_{\ell, n}$, where each $S_{\ell, n}$ is a union of convex sets with disjoint interiors.
    Since $\set{\vp_\ell}_{\ell = 1}^M \su \set{f_{j^{(m)}}}$, then $E_n \su C_{nm}$ for each $n \in \N$ and we see that $\disp E_n = \bigcup_{\ell=1}^\tau T_{\ell,n}$, where $T_{\ell, n} = S_{\ell, nm} \cap E_n$.
    Since each $T_{\ell,n}$ is either empty or a union of disjoint sets with empty interiors, then the conclusion follows.
\end{proof}

Under the open set condition, we have the following.

\begin{proposition}[\cite{mattila} Theorem 4.14, \cite{falconer1986} Corollary 8.7]
\label{openset}
    Let $\{f_j\}_{j=1}^N$ be an IFS that satisfies the open set condition and has similarity dimension $s$.
    If $C$ denotes the attractor of $\{f_j\}_{j=1}^N$, then 
    \begin{enumerate}
        \item $0< \mathcal{H}^s(C)< \infty$.
        \item $\mathcal{H}^s(f_j(C) \cap f_i(C))=0$ for $i \neq j$.
        \item $C$ is $s$-Ahlfors regular.
    \end{enumerate}
\end{proposition}

The following result provides a way to estimate dimension in the absence of self-similarity.
The proof of this result emulates that of \cite[Theorem 2.5]{hata1986hausdorff}.
Note that this result was already used above in the proof of Proposition \ref{imagesprop}.

\begin{proposition}[Generalization of Hata's result]
\label{prop dimboundcor}
     Let $\{v_n\}_{n=1}^{\infty} \su \N$ be a sequence of positive integers and define the set of words of length $n$ to be
        \begin{align*}
            \Sigma_n:= \left\{w= (w_1 \cdots  w_n)~:~ 1\leq w_j \leq v_j \mbox{ for } 1 \leq j \leq n \right\}.
        \end{align*}
   For some $b\in (0, 1)$, let $\{K(w) : w \in \Sigma_n, n \in \N \}$ be a collection of non-degenerate convex sets satisfying
    \begin{enumerate}
        \item[(a)] $K(w_1\cdots w_n) \supseteq K(w_1 \cdots w_nw_{n+1})$ for any $(w_1 \cdots w_nw_{n+1}) \in \Sigma_{n+1}$;
        \item[(b)] $\operatorname{int}(K(w)) \cap \operatorname{int}(K(w')) = \varnothing$ for any $w \neq w' \in \Sigma_n$;
        \item[(c)] $\disp D_n := \max_{w \in \Sigma_n} \diam(K(w)) \to 0 $ as $n \to \infty$;
        \item[(d)] $\disp d_n:= \min_{w \in \Sigma_n} \diam(K(w))$ satisfies $ \frac{d_n}{D_n} > b$ for all $n$;
        \item[(e)] $ \inf \set{|P_\te(K(w))| : \te \in \mathbb{S}^{d-1}}\geq b \, d_n$ for any $w \in \Sigma_n$.
    \end{enumerate}
    With 
        \begin{align*}
            E := \bigcap_{n=1}^{\infty} \bigcup_{w \in \Sigma_n} K(w) \su \mathbb{R}^d,
        \end{align*}
    it holds that
        \begin{align*}
            \dim(E)  
            \ge \liminf_{n \to \infty} \frac{\log( v_{1}v_{2} \cdots v_{n-1})}{-\log d_n}. 
        \end{align*}
\end{proposition}

\begin{proof}
    Let $\mathcal{F}$ denote the collection of all sets that are finite unions of closed balls in $\mathbb{R}^d$. 
    Define the set-function $T_n : \mathcal{F} \rightarrow \N$ by
    $$T_n(F) = \#\{w \in \Sigma_n : F \cap K(w)\neq \varnothing\}.$$ 
    Since conditions (a) and (b) imply that $T_{n+1}(F) \le v_{n+1} T_n(F)$ for any $n \in \N$ and any $F \in \mathcal{F}$, then the set function $\Phi: \mathcal{F} \rightarrow \mathbb{R}$ given by
    \begin{align*}
        \Phi(F):=\lim _{n \rightarrow \infty} \frac{T_n(F)}{v_1 \cdots v_n} 
    \end{align*}
    is the limit of a decreasing sequence, and hence well-defined.
    Since each $T_n$ is subadditive and monotonic, then so is $\Phi$.
    Moreover, $\Phi(F) \le 1$ for all $F \in \mathcal{F}$.
    In fact, if $F \supseteq E$, then $\Phi(F)=1$.

    Since $ \inf \set{|P_\te(K(w))| : \te \in \mathbb{S}^{d-1}} \geq b \, d_n$ for any $w \in \Sigma_n$,  Lemma \ref{lem innerball} and condition (d) imply that there exists a packing constant $c_0(b, d) > 0$ such that if $B$ is a ball with $\diam(B) < d_n$, then $T_n(B) \leq c_0$.
    Choose the smallest $\ell \in \N$ so that $c_0 \leq v_1\cdots v_{\ell-1}$.
    Then, for any ball $B$ with $\diam(B) < d_n$, it holds that 
    \begin{equation}
    \label{TnBound}
        T_n(B) \le v_1\cdots v_{\ell-1}. 
    \end{equation} 
    
    With $\ell$ as above, let 
    \begin{align*}
        \gamma:=\liminf_{n \rightarrow \infty} \frac{\log \pr{v_{\ell} \cdots v_{n-1}}}{-\log d_n}
    \end{align*}
    and assume that $\ga > 0$.
    For any $\delta \in \pr{0, \ga}$, there exists $N \in \N$ so that $d_n^{\gamma-\delta} v_{\ell} \cdots v_{n-1} \ge 1$ whenever $n \ge N$. 
    Therefore, there exists a constant $c(\delta)>0$ such that for any $n \ge \ell$,
    $$d_n^{\gamma-\delta} v_{\ell} \cdots v_{n-1} \geq c(\delta).$$ 
    We interpret $v_{\ell} \cdots v_{n-1} = 1$ when $n = \ell$.

    Consider now an arbitrary closed ball $B \in \mathcal{F}$ satisfying $\Phi(B)>0$. 
    Since $\Phi(B) \le 1$, then there exists a unique integer $N \ge \ell$ such that
    \begin{align*}
        \frac{1}{v_{\ell} \cdots v_{N-1}} \geq \Phi(B)>\frac{1}{v_{\ell} \cdots v_N} .
    \end{align*}
    If $\diam(B) < d_{N}$, then, by the definition of $\Phi$ and \eqref{TnBound}, we see that 
    \begin{align*}
        \Phi(B) \leq \frac{T_{N}(B)}{v_1 \cdots v_{N}} \leq  \frac{1}{v_{\ell} \cdots v_N},
    \end{align*}
    which is false.
    Therefore, $\diam(B) \geq d_{N}$ and we get
    \begin{align*}
        \Phi(B) 
        \leq \frac{1}{v_{\ell} \cdots v_{N-1}} 
        \leq \frac{\diam(B)^{\gamma-\delta}}{d_{N}^{\gamma-\delta} v_{\ell} \cdots v_{N-1}} 
        \leq \tfrac{1}{c(\delta)}\diam(B)^{\gamma-\delta}.
    \end{align*}
    For some $\eps > 0$, let $\set{B_m}_{m=1}^M$ be a finite cover of $E$, where each $B_m$ is a ball of diameter $\eps$.
    Then we have
    \begin{align*}
        \sum_{m=1}^M \diam(B_m)^{\gamma-\delta} 
        \geq c(\delta) \sum_{m=1}^M \Phi\left(B_m\right) 
        \geq c(\delta) \Phi\left(\bigcup_{m=1}^M B_m\right) 
        = c(\delta),
    \end{align*}
    where have used the subadditivity of $\Phi$.
    Since $E$ is compact, it follows  that $\mathcal{H}^{\gamma-\delta}_{\varepsilon}(E) \geq c(\delta)$
    and hence, $\mathcal{H}^{\gamma-\delta}(E)\geq c(\delta)$. 
    Since $\delta$ was arbitrary, we may conclude that $\dim(E) \geq \gamma$.
    To conclude, we note that 
    \begin{equation*}
        \ga =\liminf_{n \rightarrow \infty} \frac{\log \pr{v_{\ell} \cdots v_{n-1}}}{-\log d_n}
        =\liminf_{n \rightarrow \infty} \frac{\log \pr{v_{1} \cdots v_{n-1}}}{-\log d_n}.
    \end{equation*} \qedhere
\end{proof}

\section{Graph Construction Algorithm} 
\label{Section:graph}

In this section, we present and prove a general method for constructing a Lipschitz graph that intersects with the limit set of a nested sequence of compact sets.
Here we use the notation $P_x$ and $P_y$ to denote the orthogonal projections onto the $x$- and $y$-axis, respectively.

\begin{proposition}[Graph construction]
\label{prop graphconstruct}
Let $\{E_n\}_{n=1}^\iny \su \R^2$ be a nested sequence of compact sets with the property that for each $n \in \N$, $\disp E_n = \bigcup_{j=1}^{M_n} K_j^n$, where $\{K_j^n\}_{j=1}^{M_n}$ is a collection of closed convex sets with disjoint interiors.
Assume that there exist $\lambda>0$, and a decreasing sequence $\set{\si_n}_{n=1}^\iny \su \R_+$ with $\disp \lim_{n \to \iny} \si_n = 0$ so that for every $n \in \N$, the following hold:
\begin{align}
    & \frac{|P_{y}(K_j^n)|}{|P_x(K_j^n)|} \leq \lambda \text{ for all } j \in \{1, ..., M_n\}
    \label{convexBound} \\
    & \frac{|P_{y}(z_j-z_k)|}{|P_x(z_j-z_k)|} \leq \lambda \text{ for all } z_j \in \text{int}(K^n_{j}), z_k \in \text{int}(K^n_k), j, k \in \{1,..., M_n\}, j \ne k
    \label{connectorBound} \\
    & \diam(K^n_j) \leq \si_n.
    \label{diamBoundAbove}
\end{align}   
Then there exists a non-degenerate closed interval $I \su \mathbb{R}$ and a Lipschitz function  $g\colon I \to \mathbb{R}$ with graph $\Ga = \set{\pr{x, g(x)} : x \in I}$ such that $\Lip(\Ga) \le \la$ and
        \begin{align*}
            \bigcap_{n=1}^\iny E_n \su \Ga.
        \end{align*}
\end{proposition}

\begin{proof} 
  Let $K = \conv(E_1)$.
  Choose $z^\ell, z^r \in K$ so that $P_x(z^\ell) \le P_x(z) \le P_x(z^r)$ for all $z \in K$.
  That is, if we write $z^\ell = \pr{x^\ell, y^\ell}$, $z^r = \pr{x^r, y^r}$ in coordinates, then $x^\ell \le x \le x^r$ for all $z = \pr{x, y} \in K$. 
  Define $I = \brac{x^\ell, x^r}$.
    
    Fix $n \in \N$.
    For each $j \in \set{1, \ldots, M_n}$, choose points $z^\ell_j = \pr{x^\ell_j, y^\ell_j}$  and $z^r_j = \pr{x^r_j, y^r_j} \in K_j^n$ so that $x^\ell_j \le x \le x^r_j$ for all $z = (x,y) \in K_j^n$.
    That is, for each $j$, $P_x\pr{K_j^n} = \brac{x^\ell_{j}, x^r_j}$.
    Condition \eqref{connectorBound} guarantees that for $j \ne k$, either $x^r_j \le x^\ell_k$ or $x^r_k \le x^\ell_j$.
    Without loss of generality, the sets are ordered in the sense that $x^r_j \le x^\ell_{j+1}$ for all $j \in \set{1, \ldots, M_n -1}$.
    If we define $x_0^r = x^\ell$ and $x_{M_n+1}^\ell = x^r$, then we 
    can write
        \begin{align*}
            I = \bigcup_{j =1}^{M_n} \brac{x^\ell_{j}, x^r_j} \cup \bigcup_{k =0}^{M_n} \brac{x^r_{k}, x^\ell_{k+1}},
        \end{align*}
    where these intervals only overlap at their endpoints.
    
    Define a piecewise linear function $g_n \colon I \to \R$ as follows:
    \begin{equation*}
        g_n(x) = \begin{cases}
           y^\ell_1
           & x \in \brac{x^r_0, x^\ell_1} \\
            y^\ell_{j} + \frac{y^r_j - y^\ell_{j}}{x^r_j - x^\ell_{j}} \pr{x - x^\ell_{j}} & x \in \brac{x^\ell_{j}, x^r_j}, j \in \set{1, \ldots, M_n} \\
            y^r_{k} + \frac{y^\ell_{k+1} - y^r_{k}}{x^\ell_{k+1} - x^r_{k}} \pr{x - x^r_{k}} & x \in \brac{x^r_{k}, x^\ell_{k+1}}, k \in \set{1, \ldots, M_n -1} \\
            y^r_{M_n} 
            & x \in \brac{x^r_{M_n}, x^\ell_{M_n+1}}.
        \end{cases}
    \end{equation*}
    Since $\abs{y^r_j - y^\ell_{j}} \le \abs{P_y\pr{K^n_j}}$ while $\abs{x^r_j - x^\ell_{j}} = \abs{P_x\pr{K^n_j}}$, \eqref{convexBound} implies that $\abs{y^r_j - y^\ell_{j}} \le \lambda \abs{x^r_j - x^\ell_{j}}$.
    Similarly, because $\pr{x_k^r, y_k^r} \in K_k^n$ and $\pr{x_{k+1}^\ell, y_{k+1}^\ell} \in K_{k+1}^n$, condition \eqref{connectorBound} shows that $\abs{y^\ell_{k+1} - y^r_{k}} \le \lambda \abs{x^\ell_{k+1} - x^r_{k}}$.
    It follows that $g_n$ is an $\lambda$-Lipschitz function on $I$.

    Next we show that $\set{g_n}_{n=1}^\iny$ is Cauchy in the $C(I;\mathbb{R})$-norm, the uniform norm.
    Let $n > m \ge N$ and take $x \in I$.
    If there exists $j$ so that $(x, g_m(x)) \in K^m_j$, then, by nestedness and convexity, $(x, g_n(x)) \in K^m_j$ also. 
    The assumption \eqref{diamBoundAbove} then implies that $\abs{g_n(x) - g_m(x)} \le \si_m \le \si_N$.
    If $(x, g_m(x)) \notin K^m_j$ for any $j$, then there exists $k \in \set{0, \ldots, M_m}$ so that $x \in \brac{x_k^{r}, x_{k+1}^{\ell}}$. %$ =: \brac{x_k^{r, m}, x_{k+1}^{\ell, m}}$. 
    If $k = 0$ or $M_m$, then both $g_n$ and $g_m$ are locally constant and take values in $P_y(K_1^m)$ or $P_y(K_{M_m}^m)$, respectively.
    Thus, \eqref{diamBoundAbove} implies that, in these cases, $\abs{g_n(x) - g_m(x)} \le \si_N$.
    We now consider $k \in \set{1, \ldots, M_m-1}$.
    By definition, for $x \in \brac{x_k^{r}, x_{k+1}^{\ell}}$,
    \begin{align*}
        g_m(x) &= y^r_{k} + \frac{y^\ell_{k+1} - y^r_{k}}{x^\ell_{k+1} - x^r_{k}} \pr{x - x^r_{k}}
        = g_m(x^r_{k}) + \frac{g_m(x^\ell_{k+1}) - g_m(x^r_{k})}{x^\ell_{k+1} - x^r_{k}} \pr{x - x^r_{k}}.
    \end{align*}
    Since nestedness ensures that $g_n$ restricted to $\brac{x_k^{r}, x_{k+1}^{\ell}}$ is linear, then for $x \in \brac{x_k^{r}, x_{k+1}^{\ell}}$, we can write
    \begin{align*}
        g_n(x) &= g_n(x^r_{k}) + \frac{g_n(x^\ell_{k+1}) - g_n(x^r_{k})}{x^\ell_{k+1} - x^r_{k}} \pr{x - x^r_{k}}
    \end{align*}    
    and then
    \begin{align*}
        \abs{g_m(x) - g_n(x) }
        &= \abs{g_m(x^r_{k}) - g_n(x^{r}_{k}) 
        + \frac{g_m(x^\ell_{k+1}) - g_m(x^r_{k}) - g_n(x^{\ell}_{k+1}) + g_n(x^{r}_{k})}{x^{\ell}_{k+1} - x^{r}_{k}} \pr{x - x^{r}_{k}  }} \\
        &\le \abs{g_m(x^{r}_{k}) - g_n(x^{r}_{k})} \abs{\frac{x^{\ell}_{k+1} - x}{x^{\ell}_{k+1} - x^{r}_{k}} }
        + \abs{g_m(x^{\ell}_{k+1}) - g_n(x^{\ell}_{k+1})} \abs{\frac{x - x^{r}_{k} }{x^{\ell}_{k+1} - x^{r}_{k}}}. 
    \end{align*}
    Since $\pr{x_k^r, g_m(x_k^r)} \in K_k^m$, following the arguments in the previous case shows that $\pr{x_k^r, g_n(x_k^r)} \in K_k^m$ and we deduce that $\abs{g_n(x_k^r) - g_m(x_k^r)} \le \si_m \le \si_N$.
    Since $\pr{x_{k+1}^\ell, g_m(x_{k+1}^\ell)} \in K_{k+1}^m$, then we similarly get that $\abs{g_m(x^{\ell}_{k+1}) - g_n(x^{\ell}_{k+1})} \le \si_N$.
    Since $\abs{\frac{x^{\ell}_{k+1} - x}{x^{\ell}_{k+1} - x^{r}_{k}} } + \abs{\frac{x - x^{r}_{k} }{x^{\ell}_{k+1} - x^{r}_{k}}} = 1$, we conclude that $\abs{g_m(x) - g_n(x) } \le \si_N$ in this final case.
    Therefore, $\{g_n\}_{n=1}^\iny$ is a Cauchy sequence of continuous functions on compact $I$. 
    
    Let $\disp g:= \lim_{n \to \iny} g_n$. 
    We want to show that $g$ is an $\lambda$-Lipschitz function.
Given any $\epsilon > 0$, set $\de = \epsilon \abs{x_1 - x_2} > 0$ and choose $N \in \N$ so that whenever $n \ge N$, $\|g-g_n\|_{C(I;\mathbb{R})} < \de$.
Then
\begin{align*}
    \abs{g(x_1) - g(x_2)}
    &\le \abs{g(x_1) - g_N(x_1)} + \abs{g_N(x_1) - g_N(x_2)}+ \abs{g_N(x_2) - g(x_2)}
    < 2 \de + \lambda \abs{x_1 - x_2} \\
    &= \pr{\lambda + 2\epsilon} \abs{x_1 - x_2}.
\end{align*}
Since $\epsilon > 0$ was arbitrary, then $\abs{g(x_1) - g(x_2)} \le \lambda \abs{x_1 - x_2}$ and we conclude that $g \colon I \to \mathbb{R}$ is $\lambda$-Lipschitz.

For every $n \in \N$, let $\Ga_n$ to be the graph of $g_n$ over $I$.
That is, $\Ga_n = \set{\pr{x, g_n(x)} : x \in I}$.
Similarly, set $\Ga = \set{\pr{x, g(x)} : x \in I}$.
Condition \eqref{diamBoundAbove} and the Cauchy bound on $\set{g_n}_{n = 1}^\iny$ imply that for each $n$,
\begin{align*}
    E_n \su \Ga_n(\si_n) \su \Ga(2\si_n)
\end{align*}
which implies that 
\begin{equation*}
    \bigcap_{n=1}^{\infty} E_n \su \bigcap_{n=1}^{\infty} \Ga(2\si_n) 
    = \Ga. \qedhere
\end{equation*}
\end{proof}

\begin{remark}
An alternative approach to the above proof uses Arzel\`a-Ascoli Theorem.
However, to establish the uniform Cauchy bound and avoid passing to subsequences, we prefer this more elementary approach.
\end{remark}

\begin{remark}
\label{frameChange}
    As written, this proposition produces a graph of the form $y = g(x)$ over the standard frame in $\R^2$.
    If we replace each instance of $P_x$ and $P_y$ with $P_{\te}$ and $P_{\te + \frac \pi 2}$, respectively, for any $\te \in \brac{0, \pi}$, we can produce a graph of the form $t \, \tau_1^\te + g(t) \, \tau_2^\te$, where $\pr{\tau_1^\te, \tau_2^\te}$ is the frame corresponding to angle $\te$.
\end{remark}

Proposition \ref{prop graphconstruct} is used in four places to establish versions of Theorem \ref{mainthm}.
The first two applications of this result are to the case where $C = \mathcal{C}_4$, the 4-corner Cantor set; see Propositions \ref{construction1} and \ref{construction2}.
Theorems \ref{rotFreeGraph} and \ref{rotationalCase} rely on Proposition \ref{prop graphconstruct} to construct Lipschitz graphs that intersect the attractors of rotation-free and rotational iterated function systems, respectively, in a set of relatively high dimension.

\section{Motivating Example: The 4-corner Cantor Set} \label{Section:4corner}

Our motivating example in this project was the 4-corner Cantor set, $\mathcal{C}_4$.
Recall from above that $\disp \mathcal{C}_4 = \bigcap_{n = 0}^\iny C_n$, where each $C_n$ is a collection of $4^n$ cubes of side length $4^{-n}$.
However, $\mathcal{C}_4$ can also be realized as the attractor of the IFS $\disp \set{f_j}_{j = 1}^4$, where 
$$f_j(x)= \frac 1 4 x + z_j,$$ with 
$$z_1 = \pr{0,0}, \; z_2 = \pr{0,\tfrac 3 4}, \; z_3 = \pr{\tfrac 3 4,0}, \; z_4 = \pr{\tfrac 3 4,\tfrac 3 4}.$$
Since the fixed points of this IFS are $x_1 = \pr{0,0}$, $x_2 = \pr{0,1}$, $x_3 = \pr{1,0}$, and $x_4 = \pr{1,1}$, then the convex hull of its fixed points (and hence the convex hull of its attractor, see Lemma \ref{polygonLemma}) is $Q = \brac{0,1} \times \brac{0,1}$, the closed unit square.
Moreover, $\set{C_n}_{n=0}^\iny$ are the generations of $\disp \set{f_j}_{j = 1}^4$ with respect to $Q$.

As before, for $k \in \N$, we use the notation $j^{(k)} = \pr{j_1, j_2, \ldots, j_k} \in \set{1, 2, 3, 4}^k$ to denote a $k$-sequence of elements in $\set{1,2,3,4}$.
We partially order these vectors in the following way: 
We say that $j^{(k)} \prec i^{(\ell)}$ if there exists $n \le \min\set{k, \ell}$ so that $j_n < i_n$ while $j_m = i_m$ for all $m < n$.

With this notation, we write the iterated functions as 
\begin{equation*}
  f_{j^{(k)}} = f_{j_1} \circ f_{j_2} \circ \ldots \circ f_{j_k},  
\end{equation*}
and their associated cubes as
$$Q_{j^{(k)}} = f_{j^{(k)}}(Q) = f_{j_1} \circ f_{j_2} \circ \ldots \circ f_{j_k}(Q).$$

The cubes inherit the partial ordering from their associated sequences.
That is, we say that $Q_{j^{(k)}} \prec Q_{i^{(\ell)}}$ if and only if $j^{(k)} \prec i^{(\ell)}$.
Let $\te_0 = \arctan(1/2)$, the angle onto which all generations of $\mathcal{C}_4$ have a ``full projection'' (see Figure \ref{projPic}).
Observe that if $Q_{j^{(k)}} \prec Q_{i^{(\ell)}}$, then for any $x \in P_{\te_0} \pr{Q_{j^{(k)}}}$ and any $y \in P_{\te_0} \pr{Q_{i^{(\ell)}}}$, it holds that $x \le y$.

To construct graphs that coincide with a high-dimensional subset of $\mathcal{C}_4$, we construct nested sequences of subsets of $\set{C_n}_{n=1}^\iny$ that satisfy the hypotheses of Proposition \ref{prop graphconstruct} and have a high-dimensional attractor set.
That is, we choose particular subcollections of cubes along with an angle $\te \in \pb{\te_0, \frac \pi 4}$ onto which the projections of these cubes are well-separated.

We present two distinct graph constructions for the 4-corner Cantor set.
The first construction produces a Lipschitz function $g$, with graph $\Ga$, for which the dimension of the intersection with $\mathcal{C}_4$ is arbitrarily close to $1$ and 
$$\Lip(\Ga) \lesssim \pr{1 - \dim\pr{\mathcal{C}_4 \cap \Ga}}^{-2}.$$
In the second construction, we show that a simpler construction method can be used, but the cost is a much larger Lipschitz constant.
We include the second construction since it mimics the general methods that we will use in subsequent sections.

\subsection{The ad hoc graph construction algorithm}
\label{Cantor4Graph}

$\quad$ \\ 

Set $\mathcal{S}_1 = \set{\pr{1}, \pr{2}, \pr{4}}$.
Assuming that $\mathcal{S}_m$ has been defined, let
$$\mathcal{S}_{m+1} = \mathcal{S}_m \cup \bigcup_{\substack{k_j \in \set{1,3} \text{ for } j = 2, \ldots, m \\ k_{m+1} \in \set{2,4}}} \pr{3, k_2, \ldots, k_m, k_{m+1}}.$$
For example, we have 
\begin{align*}
\mathcal{S}_2 &= \mathcal{S}_1 \cup  \set{\pr{3,2}, \pr{3,4}}
= \set{\pr{1}, \pr{2}, \pr{3,2}, \pr{3,4}, \pr{4}} \\
\mathcal{S}_3 
&= \mathcal{S}_2 \cup \set{\pr{3, 1, 2}, \pr{3, 1, 4}, \pr{3,3,2}, \pr{3,3,4}} \\
&= \set{\pr{1}, \pr{2}, \pr{3, 1, 2}, \pr{3, 1, 4}, \pr{3,2}, \pr{3,3,2}, \pr{3,3,4}, \pr{3,4}, \pr{4}} ,
\end{align*}
and so on.
Note that the sequences in the sets have been listed according to the partial order $\prec$.
In particular, each $\mathcal{S}_m$ is well-ordered.
Note that $\abs{\mathcal{S}_m} = 2^{m} + 1$.

Given any $\mathcal{S}_m$, a set of finite sequences, define the associated sub-IFS $\mathcal{F}_m = \set{f_s}_{s \in \mathcal{S}_m}$.
Then $\mathcal{F}_m$ has similarity dimension $s_m < 1$ defined by the expression
\begin{equation}
\label{simDimEqn}
\frac{2}{4^{s_m}} + \frac 1 {4^{s_m}} \sum_{k=0}^{m-1} \pr{\frac{2}{4^{s_m}}}^k = 1.
\end{equation}
Thus $s_1 = \frac{\log 3}{\log 4}$, $s_2 = \frac{\log\pr{\frac{3 + \sqrt{17}}{2}}}{\log 4}$, $s_3 \approx \frac{\log \pr{3.8026}}{\log 4}$, and $s_m \uparrow 1$. 

\begin{lemma}[Similarity dimension bounds]
\label{simDimLem}
    For any $N \in \N$, there exists $c = c(N) > 0$ so that whenever $m \ge N$, it holds that $s_m \ge 1 - \frac c {2^m}$.
\end{lemma}

\begin{proof}
For every $m \in \N$, define $\eps_m = 4^{1 - s_m} -1 \in \pb{0,\frac 1 3}$ and observe that $\eps_m \downarrow 0$.
Therefore, if $m \ge N$, then $\eps_m \le \eps_N$.
The equation \eqref{simDimEqn} is equivalent to 
\begin{align*}
\tfrac{1 + \eps_m}{4}\brac{1 + \tfrac{1 + \eps_m}{2} + \ldots + \pr{\tfrac {1 + \eps_m} {2}}^{m-1} } = 1 - \tfrac {1 + \eps_m} 2
\iff & \pr{1 + \eps_m} \brac{1 - \pr{\tfrac {1 + \eps_m} 2}^m} = \pr{1 - \eps_m}^2 \\
\iff & \eps_m^2 -3 \eps_m + 2\pr{\tfrac {1 + \eps_m} 2}^{m+1} = 0.
\end{align*}
Then for some $c_1 \in \pr{1,2}$ depending on $\eps_N$ and $N$, it holds that
\begin{equation}
\label{epsnBound}
\begin{aligned}
\eps_m 
&= \tfrac 3 2 \pr{1 - \sqrt{1 - \tfrac 8 {9} \pr{\tfrac{1 + \eps_m}2}^{m+1}}}
= \tfrac 3 2 \brac{1 - \pr{1 - \tfrac 4 {9} \pr{\tfrac{1 + \eps_m}2}^{m+1} - \tfrac 8 {81} \pr{\tfrac{1 + \eps_m}2}^{2m+2} - \ldots}} \\
&\le \tfrac {2c_1} {3}  \pr{\tfrac{1 + \eps_m}2}^{m+1}.
\end{aligned}
\end{equation}
Now $\disp \tfrac {2c_1} {3}  \pr{\tfrac{1 + \eps_m}2}^{m+1} \le \tfrac{c_2}{m+1}$ if and only if
\begin{align*}
\pr{m + 1} \tfrac {2c_1}{3c_2} 
&\le  \pr{1 + \tfrac {1 - \eps_m}{1 + \eps_m}}^{m+1}
= 1 + \pr{m+1} \pr{\tfrac {1 - \eps_m}{1 + \eps_m}} + \tfrac{m\pr{m+1}}{2} \pr{\tfrac{1-\eps_m}{1+\eps_m}}^2 + \ldots
\end{align*}
which holds if $\disp c_2 \ge \tfrac {2c_1}{3}  \pr{\tfrac {1 + \eps_N}{1 - \eps_N}}$.
Therefore, substituting the bound $\eps_m \le \tfrac{c_2}{m+1}$ into the equation \eqref{epsnBound} shows that
\begin{align*}
\eps_m 
&\le  \tfrac {c_1} {3} 2^{-m} \pr{1 + \eps_m}^{m+1}
\le  \tfrac {c_1} {3} 2^{-m} \pr{1 + \tfrac{c_2}{m+1}}^{m+1}
\le  \tfrac {c_1 e^{c_2}} {3} 2^{-m},
\end{align*}
where we have used that
\begin{align*}
\pr{1 + \tfrac{c_2}{m}}^{m}
&= \sum_{k=0}^{m} \tfrac{m! }{k! \pr{m-k}!} \pr{\tfrac{c_2}{m}}^k
= \sum_{k=0}^{m} \pr{1-\tfrac 1 m} \ldots \pr{1-\tfrac{k-1}m}  \tfrac{c_2^k}{k!}
\le e^{c_2}.
\end{align*}
Recalling the definition of $\eps_m$, $4^{1 - s_m} \le 1 + \tfrac {c_1 e^{c_2}} {3} 2^{-m}$.
Therefore, for any $m \ge N$, there exists $c > 0$, depending only on $N$, for which $1 - s_m \le c 2^{-m}$. 
\end{proof}

\begin{figure}[ht]
\centering
\begin{tikzpicture}[scale = 0.82]
\fill[lightgray!50] (0,0) rectangle (4,4); 
\fill (0,0) rectangle (1,1); 
 \fill (0,3) rectangle (1,4); 
 \fill[gray] (3,0) rectangle (4,1); 
 \fill (3,3) rectangle (4,4); 
\draw[color=white] (0.5, 0.5) node {$1$};
\draw[color=white] (0.5, 3.5) node {$2$};
\draw[color=white] (3.5, 3.5) node {$3$};
\draw [fill=black] (1, 1) circle (2pt); 
\draw[color=black] (1.2, 1.2) node {$b_1$};
\draw [fill=black] (0, 3) circle (2pt); 
\draw[color=black] (0.1, 2.7) node {$a_2$};
\draw [fill=black] (1, 4) circle (2pt); 
\draw[color=black] (1.3, 3.8) node {$b_2$};
\draw [fill=black] (3, 3) circle (2pt); 
\draw[color=black] (3.1, 2.7) node {$a_3$};
\end{tikzpicture}
\qquad
\begin{tikzpicture}[scale = 0.82]
\fill[lightgray!50] (0,0) rectangle (4,4); 
\fill (0,0) rectangle (1,1); 
\fill (0,3) rectangle (1,4); 
\fill[gray] (3, 0) rectangle (3.25, .25);
\fill[gray] (3.75, 0) rectangle (4, .25);
\fill (3, .75) rectangle (3.25, 1);
\fill (3.75, .75) rectangle (4, 1);
\fill (3,3) rectangle (4,4); 
\draw[color=white] (0.5, 0.5) node {$1$};
\draw[color=white] (0.5, 3.5) node {$2$};
\draw[color=black] (3.12, 1.15) node {$3$};
\draw[color=black] (3.88, 1.15) node {$4$};
\draw[color=white] (3.5, 3.5) node {$5$};
\end{tikzpicture}
\qquad
\begin{tikzpicture}[scale = 0.82]
\fill[lightgray!50] (0,0) rectangle (4,4); 
\fill (0,0) rectangle (1,1); 
\fill (0,3) rectangle (1,4); 
\fill (3,3) rectangle (4,4); 
\fill (3, .75) rectangle (3.25, 1);
\fill (3.75, .75) rectangle (4, 1);
\fill[gray] (3, 0) rectangle (3.0625, .0625);
\fill[gray] (3.1875, 0) rectangle (3.25, .0625);
\fill (3, 0.1875) rectangle (3.0625, .25);
\fill (3.1875, 0.1875) rectangle (3.25, .25);
\fill[gray] (3.75, 0) rectangle (3.8125, .0625);
\fill[gray] (3.9375, 0) rectangle (4, .0625);
\fill (3.75, 0.1875) rectangle (3.8125, .25);
\fill (3.9375, 0.1875) rectangle (4, .25);
\end{tikzpicture}
\caption{
\label{EmnImages}
From left to right, the images of $E_1^1$, $E_1^2$, and $E_1^3$ are shown in black.
The numbering of cubes is indicated for $E_1^1$ and $E_1^2$.
In $E_1^1$, some of the corners are labelled.
The cubes that are not selected for each $E_1^m$ are shown in gray.}
\end{figure}
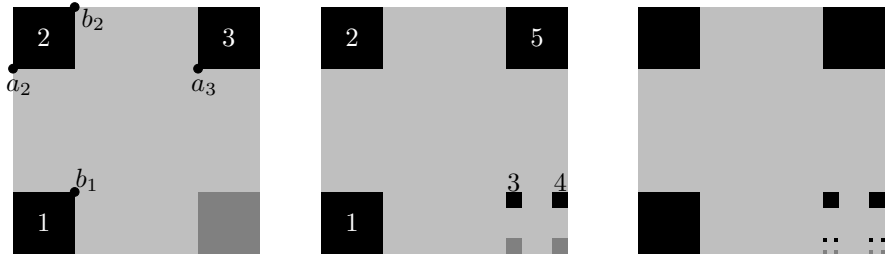

Fix some $m \in \N$.
Define the set $E_1 = E_1^m$ as 
$$E_1 = \bigcup_{s \in \mathcal{S}_m} f_s(Q) 
:= \bigcup_{i=1}^{2^m +1} Q_i^1,$$ 
a union of disjoint cubes that are ordered so that $Q_i^1 \prec Q_{i+1}^1$ for each $i \in \set{1,\ldots, 2^m}$.
Assuming that $E_n = E_n^m$ has been defined, let
$$E_{n+1} = \bigcup_{s \in \mathcal{S}_m} f_s(E_n) 
:= \bigcup_{i=1}^{\pr{2^m +1}^{n+1}} Q_i^{n+1}
\su E_n,$$ 
where the sets are disjoint and $Q_i^{n+1} \prec Q_{i+1}^{n+1}$ for each $i \in \set{1,\ldots, \pr{2^m+1}^{n+1} -1}$.
In this way, we produce the nested sequence of sets $\set{E^m_n}_{n=1}^\iny$ corresponding to the IFS $\mathcal{F}_m$.
Some images of $E^m_n$ are in Figure \ref{EmnImages}.
Define the limit set
$$E^m := \bigcap_{n=1}^\iny E_n^m$$
and note that since each $E_n \su C_n$, then $E^m \su \mathcal{C}_4$.

We use the sequences of sets $\set{E^m_n}_{n=1}^\iny$ along with Proposition \ref{prop graphconstruct} and Lemma \ref{simDimLem} to establish the following result.

\begin{proposition}[Ad hoc 4-corner graph construction]
\label{construction1}
    For every $\eps \in \pr{0, 1}$, there exists a Lipschitz graph $\Ga$ that satisfies
    $$\dim\pr{\mathcal{C}_4 \cap \Ga} \ge 1 - \eps$$
    and
    $$\Lip(\Gamma) \lesssim \eps^{-2}.$$
\end{proposition}

\begin{proof}
    Given $\eps \in \pr{0, 1}$, choose $m \in \N$ so that with $c$ from Lemma \ref{simDimLem}, it holds that $c \eps^{-1} \le 2^m < 2 c \eps^{-1}$.
    Using Lemma \ref{iterationOpenLemma} and Proposition \ref{openset}, it follows that $\dim E^m \ge 1 - \eps$.

Next we use Proposition \ref{prop graphconstruct} to construct a Lipschitz function $g^m$ with the property that $E^m \su \Ga^m := \text{graph}\pr{g^m}$.

We check that $\set{E_n^m}_{n=1}^\iny = \set{E_n}_{n=1}^\iny$ satisfies the hypotheses of Proposition \ref{prop graphconstruct}.
By construction, $\set{E_n}_{n=1}^\iny$ is nested and for each $n \in \N$, $\disp E_n = \bigcup_{i=1}^{\pr{2^m +1}^{n}} Q_i^{n}$, a finite union of disjoint cubes.
Since each $Q_i^n$ is a cube of side length $4^{-p}$ for some $p \in \set{n, \ldots, m n}$, then $\abs{P_\te\pr{Q_i^n}} \in \brac{4^{-p}, \sqrt 2 \cdot 4^{-p}}$.
In particular, \eqref{convexBound} holds with $\la = \sqrt 2$ and \eqref{diamBoundAbove} holds with $\si_n = \sqrt 2 \cdot 4^{-n}$.

By Remark \ref{frameChange}, we can change the frame so that in place of \eqref{connectorBound}, it suffices to show that there exists $\te_m \in \brac{0, \frac \pi 2}$ and $\la_m > 0$ so that for any $z_i \in Q_i^n$ and any $z_{j} \in Q_{j}^n$, it holds that
$$\frac{\abs{P_{\te_m + \frac \pi 2}\pr{z_{j} - z_i}}}{\abs{P_{\te_m}\pr{z_{j} - z_i}}} \le \la_m.$$
By self-similarity, we only need to check $n = 1$.
Let $a_i$ and $b_i$ denote the bottom-left and top-right corners, respectively, of the cube $Q_i^1$.
An inspection of the proof of Proposition \ref{prop graphconstruct} shows that testing $z_i = b_i$ and $z_j = z_{i+1} = a_{i+1}$ is sufficient for condition \eqref{connectorBound}.

Recall that $b_1 = \pr{\frac 1 4, \frac 1 4}$, $a_2 = \pr{0, \frac 3 4}$, $b_2 = \pr{\frac 1 4, 1}$, and $a_3 = \pr{\frac 3 4, \frac 3 {4^m}}$.
We define $\te_m$ so that the slope of the line from $b_1$ to $a_2$ is equal to the negative of the slope of the line from $b_2$ to $a_3$.
That is, with
\begin{align*}
    \tan \te_m 
    &= \frac{-3 + \frac {12}{4^m} + 5 \sqrt{1 - \frac{24}{5 \cdot 4^m} + \frac{36}{5 \cdot 4^{2m}}}}{4 - \frac 6 {4^m}} 
    \in \pb{\frac 1 2, 1}
\end{align*}
and then with 
\begin{align*}
    \la_m
    &=\frac{5 \cdot 4^m}{6} \pr{1 + \sqrt{1 - \frac{24}{5 \cdot 4^m} + \frac{36}{5 \cdot 4^{2m}}} - \frac {12}{5 \cdot 4^m}}
    \approx \frac 5 3 4^{m}
    \le \frac 5 3 \pr{2c}^2 \eps^{-2}
\end{align*}
it holds that
\begin{align*}
    \frac{P_{\te_m + \frac \pi 2}\pr{a_2 - b_1}}{P_{\te_m}\pr{a_2-b_1}}
    &= \frac{2 \cos \te_m + \sin \te_m}{- \cos{\te_m} + 2 \sin \te_m}
    = \la_m
     \\
    \frac{P_{\te_m + \frac \pi 2}\pr{a_3 - b_2}}{P_{\te_m}\pr{a_3 - b_2}}
    &= -\frac{\pr{2 - \frac 6 {4^m}} \cos \te_m + \sin \te_m }{\cos{\te_m} - \pr{2 - \frac 6 {4^m}} \sin \te_m} 
    = - \la_m
\end{align*}
and for all $i \in \set{3, \ldots, 2^m}$
\begin{align*}
    \abs{\frac{P_{\te_m + \frac \pi 2}\pr{a_{i+1} - b_i}}{P_{\te_m}\pr{a_{i+1}-b_i}}} \le \la_m.
\end{align*}
From here, an application of Proposition \ref{prop graphconstruct} shows that there exists an $\la_m$-Lipschitz function $g^m : P_{\te_m}\pr{\brac{0, 1}^2} \to \R$ with graph 
$$\Ga^m = \set{t \tau_1^m + g^m(t) \tau_2^m : t \in P_{\te_m}\pr{\brac{0, 1}^2}},$$ 
where $\pr{\tau_1^m,\tau_2^m}$ is the frame defined through $\te_m$.
Moreover, $E^m \su \Ga^m$.
Since $E^m \cap \mathcal{C}_4 = E^m$, then 
    \begin{align*}
        \dim\pr{\Ga^m \cap \mathcal{C}_4} \ge \dim\pr{E^m} \ge 1 - \eps.
    \end{align*}
Since $\la_m \lesssim \eps^{-2}$, the conclusion follows.
\end{proof}

With the notation from Proposition \ref{prop graphconstruct}, $\disp g^m = \lim_{n \to \iny} g^m_n$ and we use $\Ga^m_n$ to denote the graph of $g^m_n$ over the frame $\pr{\tau_1^m, \tau_2^m}$. 
Images of these the graphs in these sequences are in Figures \ref{L1jImages} -- \ref{L23jImages}.

\begin{figure}[ht]
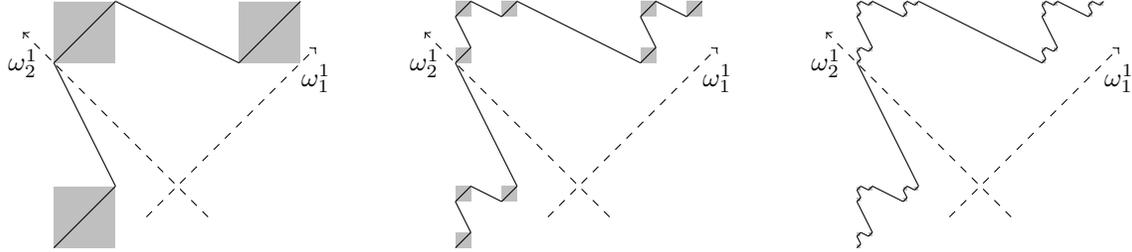

\centering

 \caption{
 \label{L1jImages}
 From left to right, the graphs $\Ga^1_1$, $\Ga^1_2$, and $\Ga^1_3$ are shown in black over the sets $E^1_1$, $E^1_2$, and $E^1_3$, respectively, shown in gray. 
 In each image, the vectors $\tau_1^1$ and $\tau^1_2$ are indicated with dashed lines and labelled.}
 \end{figure}

\begin{figure}[h!tbp]
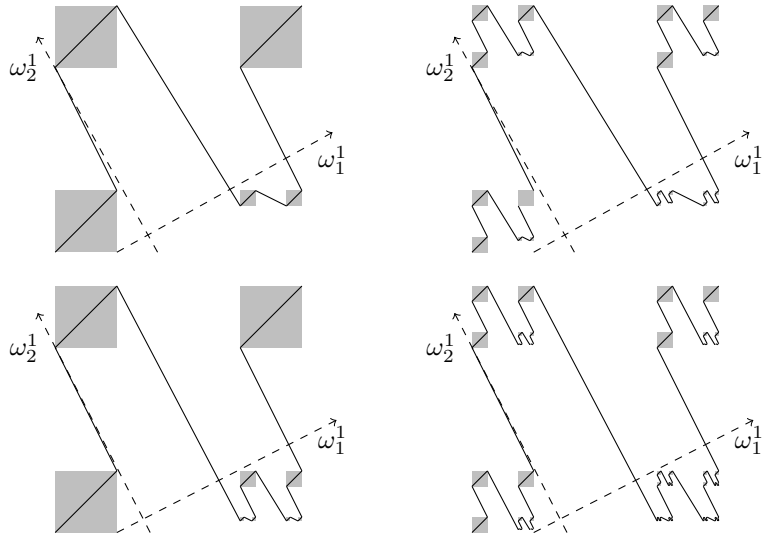
%{l}{0.65\textwidth}
\centering
%
\caption{
\label{L23jImages}
From left to right, the images of $\Ga^2_1$ and $\Ga^2_2$ (top row), then $\Ga^3_1$ and $\Ga^3_2$ (bottom row), the graphs of the functions that limit to $g^2$ (top) and $g^3$ (bottom) are shown in black.
Under each $\Ga^m_n$, the set $E^m_n$ is shown in gray.
In each image, the vectors $\tau_1^1$ and $\tau^1_2$ are indicated with dashed lines and labelled.
}
\end{figure}

\subsection{The generic graph construction algorithm}

$\quad$ \\ 

Here we demonstrate a simpler way to construct a Lipschitz graph that sees a high-dimensional subset of $\mathcal{C}_4$.
In the previous construction, we took different levels of iterations in different parts of the set.
In a sense, we increased the dimension of our graph intersection by zooming in on the bottom right cube.
For this construction, we take a uniform approach and extract exactly half (every other, when ordered) of the functions in the $m$-th iteration of the original IFS.
We include this construction because this method illustrates our idea for the more general setting.

Fix $m \in \N$ and define the collection of sequences
\begin{align*}
{\mathcal{S}}_m = \bigcup_{\substack{k_j \in \set{1,2,3,4} \text{ for } j = 1, \ldots, m-1 \\ k_{m} \in \set{1,3}}} \pr{k_1, \ldots, k_{m-1}, k_m}.
\end{align*}
Observe that ${\mathcal{S}}_m$ can be well-ordered using $\prec$ and that $\abs{{\mathcal{S}}_m} = \frac {4^m} 2 = 2^{2m-1}$.
Then set ${\mathcal{F}}_m = \set{f_{t}}_{t \in {\mathcal{S}}_m} = \set{f_i}_{i = 1}^{2^{2m-1}}$ to be the associated sub-IFS where the indexing indicates order.
The similarity dimension ${s}_m$ associated to ${\mathcal{F}}_m$ is defined by $\frac{4^m}{2} 4^{-m{s}_m} = 1$, i.e. ${s}_m = 1 - \frac 1 {2m}$.

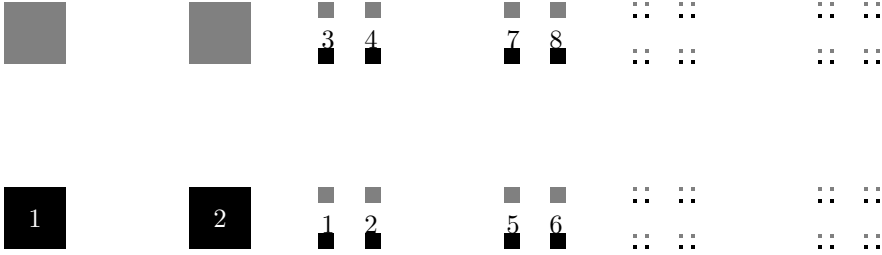
\begin{figure}[ht]
\centering
\begin{tikzpicture}[scale=0.82]
\fill[black]
 (0, 0) -- (0, 1) --
(0, 1) -- (1, 1) --
(1, 1) -- (1, 0) -- 
(1, 0) -- (0, 0) -- 
   cycle; 
\fill[gray]
 (0, 3) -- (0, 4) --
(0, 4) -- (1, 4) --
(1, 4) -- (1, 3) -- 
(1, 3) -- (0, 3) -- 
   cycle;
\fill[black]
 (3, 0) -- (3, 1) --
(3, 1) -- (3, 1) --
(4, 1) -- (4, 0) -- 
(4, 0) -- (3, 0) -- 
   cycle;
\fill[gray]
 (3, 3) -- (3, 4) --
(3, 4) -- (4, 4) --
(4, 4) -- (4, 3) -- 
(4, 3) -- (3, 3) -- 
   cycle;
\draw[color=white] (0.5, 0.5) node {$1$};
\draw[color=white] (3.5, 0.5) node {$2$};
\end{tikzpicture}
\qquad
\begin{tikzpicture}[scale=0.82]
\fill[black]
(0, 0) -- (0, .25) --
(0, .25) -- (0.25, .25) --
(0.25, .25) -- (0.25, 0) -- 
(0.25, 0) -- (0, 0) -- 
   cycle;
\fill[black]
 (0.75, 0) -- (0.75, .25) --
(0.75, .25) -- (0.75, .25) --
(1, .25) -- (1, 0) -- 
(1, 0) -- (0.75, 0) -- 
   cycle;
\fill[gray]
 (0, .75) -- (0, 1) --
(0, 1) -- (0.25, 1) --
(0.25, 1) -- (0.25, .75) -- 
(0.25, .75) -- (0, .75) -- 
   cycle;
\fill[gray]
 (0.75, .75) -- (0.75, 1) --
(0.75, 1) -- (1, 1) --
(1, 1) -- (1, .75) -- 
(1, .75) -- (0.75, .75) -- 
   cycle;
\fill[black]
(0, 3+0) -- (0, 3+.25) --
(0, 3+.25) -- (0.25, 3+.25) --
(0.25, 3+.25) -- (0.25, 3+0) -- 
(0.25, 3+0) -- (0,3+ 0) -- 
   cycle;
\fill[black]
 (0.75, 3+0) -- (0.75,3+ .25) --
(0.75, 3+.25) -- (0.75, 3+.25) --
(1,3+ .25) -- (1, 3+0) -- 
(1,3+ 0) -- (0.75, 3+0) -- 
   cycle;
\fill[gray]
 (0, 3+.75) -- (0, 3+1) --
(0, 3+1) -- (0.25, 3+1) --
(0.25, 3+1) -- (0.25,3+ .75) -- 
(0.25,3+ .75) -- (0,3+ .75) -- 
   cycle;
\fill[gray]
 (0.75,3+ .75) -- (0.75, 3+1) --
(0.75, 3+1) -- (1, 3+1) --
(1, 3+1) -- (1, 3+.75) -- 
(1, 3+.75) -- (0.75, 3+.75) -- 
   cycle;
\fill[black]
 (3, 0) -- (3, .25) --
(3, .25) -- (3.25, .25) --
(3.25, .25) -- (3.25, 0) -- 
(3.25, 0) -- (3, 0) -- 
   cycle;
\fill[black]
 (3.75, 0) -- (3.75, .25) --
(3.75, .25) -- (3.75, .25) --
(4, .25) -- (4, 0) -- 
(4, 0) -- (3.75, 0) -- 
   cycle;
\fill[gray]
 (3, .75) -- (3, 1) --
(3, 1) -- (3.25, 1) --
(3.25, 1) -- (3.25, .75) -- 
(3.25, .75) -- (3, .75) -- 
   cycle;
\fill[gray]
 (3.75, .75) -- (3.75, 1) --
(3.75, 1) -- (4, 1) --
(4, 1) -- (4, .75) -- 
(4, .75) -- (3.75, .75) -- 
   cycle;
\fill[black]
 (3, 3+0) -- (3, 3+.25) --
(3, 3+.25) -- (3.25,3+ .25) --
(3.25,3+ .25) -- (3.25,3+ 0) -- 
(3.25, 3+0) -- (3, 3+0) -- 
   cycle;
\fill[black]
 (3.75, 3+0) -- (3.75, 3+.25) --
(3.75,3+ .25) -- (3.75,3+ .25) --
(4, 3+.25) -- (4, 3+0) -- 
(4, 3+0) -- (3.75, 3+0) -- 
   cycle;
\fill[gray]
 (3, 3+.75) -- (3, 3+1) --
(3, 3+1) -- (3.25, 3+1) --
(3.25, 3+1) -- (3.25, 3+.75) -- 
(3.25, 3+.75) -- (3,3+ .75) -- 
   cycle;
\fill[gray]
 (3.75, 3+.75) -- (3.75, 3+1) --
(3.75, 3+1) -- (4, 3+1) --
(4, 3+1) -- (4, 3+.75) -- 
(4, 3+.75) -- (3.75,3+ .75) -- 
   cycle;
\draw[color=black] (0.15, 0.4) node {$1$};
\draw[color=black] (0.85, 0.4) node {$2$};
\draw[color=black] (0.15, 3.4) node {$3$};
\draw[color=black] (0.85, 3.4) node {$4$};
\draw[color=black] (3.15, 0.4) node {$5$};
\draw[color=black] (3.85, 0.4) node {$6$};
\draw[color=black] (3.15, 3.4) node {$7$};
\draw[color=black] (3.85, 3.4) node {$8$};
\end{tikzpicture}
\qquad
\begin{tikzpicture}[scale=0.82]
\fill[black]
(0.25*0, 0.25*0) -- (0.25*0, 0.25*0.25) --
(0.25*0, 0.25*0.25) -- (0.25*0.25, 0.25*0.25) --
(0.25*0.25, 0.25*0.25) -- (0.25*0.25, 0.25*0) -- 
(0.25*0.25, 0.25*0) -- (0.25*0, 0.25*0) -- 
   cycle;
\fill[black]
(0.25*0.75, 0.25*0) -- (0.25*0.75, 0.25*0.25) --
(0.25*0.75, 0.25*0.25) -- (0.25*0.75, 0.25*0.25) --
(0.25*1, 0.25*0.25) -- (0.25*1, 0.25*0) -- 
(0.25*1, 0.25*0) -- (0.25*0.75, 0.25*0) -- 
   cycle;
\fill[gray]
 (0.25*0, 0.25*0.75) -- (0.25*0, 0.25*1) --
(0.25*0, 0.25*1) -- (0.25*0.25, 0.25*1) --
(0.25*0.25, 0.25*1) -- (0.25*0.25, 0.25*0.75) -- 
(0.25*0.25, 0.25*0.75) -- (0.25*0, 0.25*0.75) -- 
   cycle;
\fill[gray]
 (0.25*0.75, 0.25*0.75) -- (0.25*0.75, 0.25*1) --
(0.25*0.75, 0.25*1) -- (0.25*1, 0.25*1) --
(0.25*1, 0.25*1) -- (0.25*1, 0.25*0.75) -- 
(0.25*1, 0.25*0.75) -- (0.25*0.75, 0.25*0.75) -- 
   cycle;
\fill[black]
(0.25*0, 0.25*3) -- (0.25*0, 0.25*3.25) --
(0.25*0, 0.25*3.25) -- (0.25*0.25, 0.25*3.25) --
(0.25*0.25, 0.25*3.25) -- (0.25*0.25, 0.25*3) -- 
(0.25*0.25, 0.25*3) -- (0.25*0, 0.25*3) -- 
   cycle;
\fill[black]
(0.25*0.75, 0.25*3) -- (0.25*0.75, 0.25*3.25) --
(0.25*0.75, 0.25*3.25) -- (0.25*0.75, 0.25*3.25) --
(0.25*1, 0.25*3.25) -- (0.25*1, 0.25*3) -- 
(0.25*1, 0.25*3) -- (0.25*0.75, 0.25*3) -- 
   cycle;
\fill[gray]
(0.25*0, 0.25*3.75) -- (0.25*0, 0.25*4) --
(0.25*0, 0.25*4) -- (0.25*0.25, 0.25*4) --
(0.25*0.25, 0.25*4) -- (0.25*0.25, 0.25*3.75) -- 
(0.25*0.25, 0.25*3.75) -- (0.25*0, 0.25*3.75) -- 
   cycle;
\fill[gray]
 (0.25*0.75, 0.25*3.75) -- (0.25*0.75, 0.25*4) --
(0.25*0.75, 0.25*4) -- (0.25*1, 0.25*4) --
(0.25*1, 0.25*4) -- (0.25*1, 0.25*3.75) -- 
(0.25*1, 0.25*3.75) -- (0.25*0.75, 0.25*3.75) -- 
   cycle;
\fill[black]
 (0.25*3, 0.25*0) -- (0.25*3, 0.25*0.25) --
(0.25*3, 0.25*0.25) -- (0.25*3.25, 0.25*0.25) --
(0.25*3.25, 0.25*0.25) -- (0.25*3.25, 0.25*0) -- 
(0.25*3.25, 0.25*0) -- (0.25*3, 0.25*0) -- 
   cycle;
\fill[black]
(0.25*3.75, 0.25*0) -- (0.25*3.75, 0.25*0.25) --
(0.25*3.75, 0.25*0.25) -- (0.25*3.75, 0.25*0.25) --
(0.25*4, 0.25*0.25) -- (0.25*4, 0.25*0) -- 
(0.25*4, 0.25*0) -- (0.25*3.75, 0.25*0) -- 
   cycle;
\fill[gray]
(0.25*3, 0.25*0.75) -- (0.25*3, 0.25*1) --
(0.25*3, 0.25*1) -- (0.25*3.25, 0.25*1) --
(0.25*3.25, 0.25*1) -- (0.25*3.25, 0.25*0.75) -- 
(0.25*3.25, 0.25*0.75) -- (0.25*3, 0.25*0.75) -- 
   cycle;
\fill[gray]
(0.25*3.75, 0.25*0.75) -- (0.25*3.75, 0.25*1) --
(0.25*3.75, 0.25*1) -- (0.25*4, 0.25*1) --
(0.25*4, 0.25*1) -- (0.25*4, 0.25*0.75) -- 
(0.25*4, 0.25*0.75) -- (0.25*3.75, 0.25*0.75) -- 
   cycle;
\fill[black]
(0.25*3, 0.25*3) -- (0.25*3, 0.25*3.25) --
(0.25*3, 0.25*3.25) -- (0.25*3.25, 0.25*3.25) --
(0.25*3.25, 0.25*3.25) -- (0.25*3.25, 0.25*3) -- 
(0.25*3.25, 0.25*3) -- (0.25*3, 0.25*3) -- 
   cycle;
\fill[black]
(0.25*3.75, 0.25*3) -- (0.25*3.75, 0.25*3.25) --
(0.25*3.75, 0.25*3.25) -- (0.25*3.75, 0.25*3.25) --
(0.25*4, 0.25*3.25) -- (0.25*4, 0.25*3) -- 
(0.25*4, 0.25*3) -- (0.25*3.75, 0.25*3) -- 
   cycle;
\fill[gray]
(0.25*3, 0.25*3.75) -- (0.25*3, 0.25*4) --
(0.25*3, 0.25*4) -- (0.25*3.25, 0.25*4) --
(0.25*3.25, 0.25*4) -- (0.25*3.25, 0.25*3.75) -- 
(0.25*3.25, 0.25*3.75) -- (0.25*3, 0.25*3.75) -- 
   cycle;
\fill[gray]
(0.25*3.75, 0.25*3.75) -- (0.25*3.75, 0.25*4) --
(0.25*3.75, 0.25*4) -- (0.25*4, 0.25*4) --
(0.25*4, 0.25*4) -- (0.25*4, 0.25*3.75) -- 
(0.25*4, 0.25*3.75) -- (0.25*3.75, 0.25*3 .75) -- 
   cycle;
\fill[black]
(3+0.25*0, 0.25*0) -- (3+0.25*0, 0.25*0.25) --
(3+0.25*0, 0.25*0.25) -- (3+0.25*0.25, 0.25*0.25) --
(3+0.25*0.25, 0.25*0.25) -- (3+0.25*0.25, 0.25*0) -- 
(3+0.25*0.25, 0.25*0) -- (3+0.25*0, 0.25*0) -- 
   cycle;
\fill[black]
(3+0.25*0.75, 0.25*0) -- (3+0.25*0.75, 0.25*0.25) --
(3+0.25*0.75, 0.25*0.25) -- (3+0.25*0.75, 0.25*0.25) --
(3+0.25*1, 0.25*0.25) -- (3+0.25*1, 0.25*0) -- 
(3+0.25*1, 0.25*0) -- (3+0.25*0.75, 0.25*0) -- 
   cycle;
\fill[gray]
 (3+0.25*0, 0.25*0.75) -- (3+0.25*0, 0.25*1) --
(3+0.25*0, 0.25*1) -- (3+0.25*0.25, 0.25*1) --
(3+0.25*0.25, 0.25*1) -- (3+0.25*0.25, 0.25*0.75) -- 
(3+0.25*0.25, 0.25*0.75) -- (3+0.25*0, 0.25*0.75) -- 
   cycle;
\fill[gray]
 (3+0.25*0.75, 0.25*0.75) -- (3+0.25*0.75, 0.25*1) --
(3+0.25*0.75, 0.25*1) -- (3+0.25*1, 0.25*1) --
(3+0.25*1, 0.25*1) -- (3+0.25*1, 0.25*0.75) -- 
(3+0.25*1, 0.25*0.75) -- (3+0.25*0.75, 0.25*0.75) -- 
   cycle;
\fill[black]
(3+0.25*0, 0.25*3) -- (3+0.25*0, 0.25*3.25) --
(3+0.25*0, 0.25*3.25) -- (3+0.25*0.25, 0.25*3.25) --
(3+0.25*0.25, 0.25*3.25) -- (3+0.25*0.25, 0.25*3) -- 
(3+0.25*0.25, 0.25*3) -- (3+0.25*0, 0.25*3) -- 
   cycle;
\fill[black]
(3+0.25*0.75, 0.25*3) -- (3+0.25*0.75, 0.25*3.25) --
(3+0.25*0.75, 0.25*3.25) -- (3+0.25*0.75, 0.25*3.25) --
(3+0.25*1, 0.25*3.25) -- (3+0.25*1, 0.25*3) -- 
(3+0.25*1, 0.25*3) -- (3+0.25*0.75, 0.25*3) -- 
   cycle;
\fill[gray]
(3+0.25*0, 0.25*3.75) -- (3+0.25*0, 0.25*4) --
(3+0.25*0, 0.25*4) -- (3+0.25*0.25, 0.25*4) --
(3+0.25*0.25, 0.25*4) -- (3+0.25*0.25, 0.25*3.75) -- 
(3+0.25*0.25, 0.25*3.75) -- (3+0.25*0, 0.25*3.75) -- 
   cycle;
\fill[gray]
 (3+0.25*0.75, 0.25*3.75) -- (3+0.25*0.75, 0.25*4) --
(3+0.25*0.75, 0.25*4) -- (3+0.25*1, 0.25*4) --
(3+0.25*1, 0.25*4) -- (3+0.25*1, 0.25*3.75) -- 
(3+0.25*1, 0.25*3.75) -- (3+0.25*0.75, 0.25*3.75) -- 
   cycle;
\fill[black]
 (3+0.25*3, 0.25*0) -- (3+0.25*3, 0.25*0.25) --
(3+0.25*3, 0.25*0.25) -- (3+0.25*3.25, 0.25*0.25) --
(3+0.25*3.25, 0.25*0.25) -- (3+0.25*3.25, 0.25*0) -- 
(3+0.25*3.25, 0.25*0) -- (3+0.25*3, 0.25*0) -- 
   cycle;
\fill[black]
(3+0.25*3.75, 0.25*0) -- (3+0.25*3.75, 0.25*0.25) --
(3+0.25*3.75, 0.25*0.25) -- (3+0.25*3.75, 0.25*0.25) --
(3+0.25*4, 0.25*0.25) -- (3+0.25*4, 0.25*0) -- 
(3+0.25*4, 0.25*0) -- (3+0.25*3.75, 0.25*0) -- 
   cycle;
\fill[gray]
(3+0.25*3, 0.25*0.75) -- (3+0.25*3, 0.25*1) --
(3+0.25*3, 0.25*1) -- (3+0.25*3.25, 0.25*1) --
(3+0.25*3.25, 0.25*1) -- (3+0.25*3.25, 0.25*0.75) -- 
(3+0.25*3.25, 0.25*0.75) -- (3+0.25*3, 0.25*0.75) -- 
   cycle;
\fill[gray]
(3+0.25*3.75, 0.25*0.75) -- (3+0.25*3.75, 0.25*1) --
(3+0.25*3.75, 0.25*1) -- (3+0.25*4, 0.25*1) --
(3+0.25*4, 0.25*1) -- (3+0.25*4, 0.25*0.75) -- 
(3+0.25*4, 0.25*0.75) -- (3+0.25*3.75, 0.25*0.75) -- 
   cycle;
\fill[black]
(3+0.25*3, 0.25*3) -- (3+0.25*3, 0.25*3.25) --
(3+0.25*3, 0.25*3.25) -- (3+0.25*3.25, 0.25*3.25) --
(3+0.25*3.25, 0.25*3.25) -- (3+0.25*3.25, 0.25*3) -- 
(3+0.25*3.25, 0.25*3) -- (3+0.25*3, 0.25*3) -- 
   cycle;
\fill[black]
(3+0.25*3.75, 0.25*3) -- (3+0.25*3.75, 0.25*3.25) --
(3+0.25*3.75, 0.25*3.25) -- (3+0.25*3.75, 0.25*3.25) --
(3+0.25*4, 0.25*3.25) -- (3+0.25*4, 0.25*3) -- 
(3+0.25*4, 0.25*3) -- (3+0.25*3.75, 0.25*3) -- 
   cycle;
\fill[gray]
(3+0.25*3, 0.25*3.75) -- (3+0.25*3, 0.25*4) --
(3+0.25*3, 0.25*4) -- (3+0.25*3.25, 0.25*4) --
(3+0.25*3.25, 0.25*4) -- (3+0.25*3.25, 0.25*3.75) -- 
(3+0.25*3.25, 0.25*3.75) -- (3+0.25*3, 0.25*3.75) -- 
   cycle;
\fill[gray]
(3+0.25*3.75, 0.25*3.75) -- (3+0.25*3.75, 0.25*4) --
(3+0.25*3.75, 0.25*4) -- (3+0.25*4, 0.25*4) --
(3+0.25*4, 0.25*4) -- (3+0.25*4, 0.25*3.75) -- 
(3+0.25*4, 0.25*3.75) -- (3+0.25*3.75, 0.25*3 .75) -- 
   cycle;
\fill[black]
(0.25*0, 3+0.25*0) -- (0.25*0, 3+0.25*0.25) --
(0.25*0, 3+0.25*0.25) -- (0.25*0.25, 3+0.25*0.25) --
(0.25*0.25, 3+0.25*0.25) -- (0.25*0.25, 3+0.25*0) -- 
(0.25*0.25, 3+0.25*0) -- (0.25*0, 3+0.25*0) -- 
   cycle;
\fill[black]
(0.25*0.75, 3+0.25*0) -- (0.25*0.75, 3+0.25*0.25) --
(0.25*0.75, 3+0.25*0.25) -- (0.25*0.75, 3+0.25*0.25) --
(0.25*1, 3+0.25*0.25) -- (0.25*1, 3+0.25*0) -- 
(0.25*1, 3+0.25*0) -- (0.25*0.75, 3+0.25*0) -- 
   cycle;
\fill[gray]
 (0.25*0, 3+0.25*0.75) -- (0.25*0, 3+0.25*1) --
(0.25*0, 3+0.25*1) -- (0.25*0.25, 3+0.25*1) --
(0.25*0.25, 3+0.25*1) -- (0.25*0.25, 3+0.25*0.75) -- 
(0.25*0.25, 3+0.25*0.75) -- (0.25*0, 3+0.25*0.75) -- 
   cycle;
\fill[gray]
 (0.25*0.75, 3+0.25*0.75) -- (0.25*0.75, 3+0.25*1) --
(0.25*0.75, 3+0.25*1) -- (0.25*1, 3+0.25*1) --
(0.25*1, 3+0.25*1) -- (0.25*1, 3+0.25*0.75) -- 
(0.25*1, 3+0.25*0.75) -- (0.25*0.75, 3+0.25*0.75) -- 
   cycle;
\fill[black]
(0.25*0, 3+0.25*3) -- (0.25*0, 3+0.25*3.25) --
(0.25*0, 3+0.25*3.25) -- (0.25*0.25, 3+0.25*3.25) --
(0.25*0.25, 3+0.25*3.25) -- (0.25*0.25, 3+0.25*3) -- 
(0.25*0.25, 3+0.25*3) -- (0.25*0, 3+0.25*3) -- 
   cycle;
\fill[black]
(0.25*0.75, 3+0.25*3) -- (0.25*0.75, 3+0.25*3.25) --
(0.25*0.75, 3+0.25*3.25) -- (0.25*0.75, 3+0.25*3.25) --
(0.25*1, 3+0.25*3.25) -- (0.25*1, 3+0.25*3) -- 
(0.25*1, 3+0.25*3) -- (0.25*0.75, 3+0.25*3) -- 
   cycle;
\fill[gray]
(0.25*0, 3+0.25*3.75) -- (0.25*0, 3+0.25*4) --
(0.25*0, 3+0.25*4) -- (0.25*0.25, 3+0.25*4) --
(0.25*0.25, 3+0.25*4) -- (0.25*0.25, 3+0.25*3.75) -- 
(0.25*0.25, 3+0.25*3.75) -- (0.25*0, 3+0.25*3.75) -- 
   cycle;
\fill[gray]
 (0.25*0.75, 3+0.25*3.75) -- (0.25*0.75, 3+0.25*4) --
(0.25*0.75, 3+0.25*4) -- (0.25*1, 3+0.25*4) --
(0.25*1, 3+0.25*4) -- (0.25*1, 3+0.25*3.75) -- 
(0.25*1, 3+0.25*3.75) -- (0.25*0.75, 3+0.25*3.75) -- 
   cycle;
\fill[black]
 (0.25*3, 3+0.25*0) -- (0.25*3, 3+0.25*0.25) --
(0.25*3, 3+0.25*0.25) -- (0.25*3.25, 3+0.25*0.25) --
(0.25*3.25, 3+0.25*0.25) -- (0.25*3.25, 3+0.25*0) -- 
(0.25*3.25, 3+0.25*0) -- (0.25*3, 3+0.25*0) -- 
   cycle;
\fill[black]
(0.25*3.75, 3+0.25*0) -- (0.25*3.75, 3+0.25*0.25) --
(0.25*3.75, 3+0.25*0.25) -- (0.25*3.75, 3+0.25*0.25) --
(0.25*4, 3+0.25*0.25) -- (0.25*4, 3+0.25*0) -- 
(0.25*4, 3+0.25*0) -- (0.25*3.75, 3+0.25*0) -- 
   cycle;
\fill[gray]
(0.25*3, 3+0.25*0.75) -- (0.25*3, 3+0.25*1) --
(0.25*3, 3+0.25*1) -- (0.25*3.25, 3+0.25*1) --
(0.25*3.25, 3+0.25*1) -- (0.25*3.25, 3+0.25*0.75) -- 
(0.25*3.25, 3+0.25*0.75) -- (0.25*3, 3+0.25*0.75) -- 
   cycle;
\fill[gray]
(0.25*3.75, 3+0.25*0.75) -- (0.25*3.75, 3+0.25*1) --
(0.25*3.75, 3+0.25*1) -- (0.25*4, 3+0.25*1) --
(0.25*4, 3+0.25*1) -- (0.25*4, 3+0.25*0.75) -- 
(0.25*4, 3+0.25*0.75) -- (0.25*3.75, 3+0.25*0.75) -- 
   cycle;
\fill[black]
(0.25*3, 3+0.25*3) -- (0.25*3, 3+0.25*3.25) --
(0.25*3, 3+0.25*3.25) -- (0.25*3.25, 3+0.25*3.25) --
(0.25*3.25, 3+0.25*3.25) -- (0.25*3.25, 3+0.25*3) -- 
(0.25*3.25, 3+0.25*3) -- (0.25*3, 3+0.25*3) -- 
   cycle;
\fill[black]
(0.25*3.75, 3+0.25*3) -- (0.25*3.75, 3+0.25*3.25) --
(0.25*3.75, 3+0.25*3.25) -- (0.25*3.75, 3+0.25*3.25) --
(0.25*4, 3+0.25*3.25) -- (0.25*4, 3+0.25*3) -- 
(0.25*4, 3+0.25*3) -- (0.25*3.75, 3+0.25*3) -- 
   cycle;
\fill[gray]
(0.25*3, 3+0.25*3.75) -- (0.25*3, 3+0.25*4) --
(0.25*3, 3+0.25*4) -- (0.25*3.25, 3+0.25*4) --
(0.25*3.25, 3+0.25*4) -- (0.25*3.25, 3+0.25*3.75) -- 
(0.25*3.25, 3+0.25*3.75) -- (0.25*3, 3+0.25*3.75) -- 
   cycle;
\fill[gray]
(0.25*3.75, 3+0.25*3.75) -- (0.25*3.75, 3+0.25*4) --
(0.25*3.75, 3+0.25*4) -- (0.25*4, 3+0.25*4) --
(0.25*4, 3+0.25*4) -- (0.25*4, 3+0.25*3.75) -- 
(0.25*4, 3+0.25*3.75) -- (0.25*3.75, 3+0.25*3 .75) -- 
   cycle;
\fill[black]
(3+0.25*0, 3+0.25*0) -- (3+0.25*0, 3+0.25*0.25) --
(3+0.25*0, 3+0.25*0.25) -- (3+0.25*0.25, 3+0.25*0.25) --
(3+0.25*0.25, 3+0.25*0.25) -- (3+0.25*0.25, 3+0.25*0) -- 
(3+0.25*0.25, 3+0.25*0) -- (3+0.25*0, 3+0.25*0) -- 
   cycle;
\fill[black]
(3+0.25*0.75, 3+0.25*0) -- (3+0.25*0.75, 3+0.25*0.25) --
(3+0.25*0.75, 3+0.25*0.25) -- (3+0.25*0.75, 3+0.25*0.25) --
(3+0.25*1, 3+0.25*0.25) -- (3+0.25*1, 3+0.25*0) -- 
(3+0.25*1, 3+0.25*0) -- (3+0.25*0.75, 3+0.25*0) -- 
   cycle;
\fill[gray]
 (3+0.25*0, 3+0.25*0.75) -- (3+0.25*0, 3+0.25*1) --
(3+0.25*0, 3+0.25*1) -- (3+0.25*0.25, 3+0.25*1) --
(3+0.25*0.25, 3+0.25*1) -- (3+0.25*0.25, 3+0.25*0.75) -- 
(3+0.25*0.25, 3+0.25*0.75) -- (3+0.25*0, 3+0.25*0.75) -- 
   cycle;
\fill[gray]
 (3+0.25*0.75, 3+0.25*0.75) -- (3+0.25*0.75, 3+0.25*1) --
(3+0.25*0.75, 3+0.25*1) -- (3+0.25*1, 3+0.25*1) --
(3+0.25*1, 3+0.25*1) -- (3+0.25*1, 3+0.25*0.75) -- 
(3+0.25*1, 3+0.25*0.75) -- (3+0.25*0.75, 3+0.25*0.75) -- 
   cycle;
\fill[black]
(3+0.25*0, 3+0.25*3) -- (3+0.25*0, 3+0.25*3.25) --
(3+0.25*0, 3+0.25*3.25) -- (3+0.25*0.25, 3+0.25*3.25) --
(3+0.25*0.25, 3+0.25*3.25) -- (3+0.25*0.25, 3+0.25*3) -- 
(3+0.25*0.25, 3+0.25*3) -- (3+0.25*0, 3+0.25*3) -- 
   cycle;
\fill[black]
(3+0.25*0.75, 3+0.25*3) -- (3+0.25*0.75, 3+0.25*3.25) --
(3+0.25*0.75, 3+0.25*3.25) -- (3+0.25*0.75, 3+0.25*3.25) --
(3+0.25*1, 3+0.25*3.25) -- (3+0.25*1, 3+0.25*3) -- 
(3+0.25*1, 3+0.25*3) -- (3+0.25*0.75, 3+0.25*3) -- 
   cycle;
\fill[gray]
(3+0.25*0, 3+0.25*3.75) -- (3+0.25*0, 3+0.25*4) --
(3+0.25*0, 3+0.25*4) -- (3+0.25*0.25, 3+0.25*4) --
(3+0.25*0.25, 3+0.25*4) -- (3+0.25*0.25, 3+0.25*3.75) -- 
(3+0.25*0.25, 3+0.25*3.75) -- (3+0.25*0, 3+0.25*3.75) -- 
   cycle;
\fill[gray]
 (3+0.25*0.75, 3+0.25*3.75) -- (3+0.25*0.75, 3+0.25*4) --
(3+0.25*0.75, 3+0.25*4) -- (3+0.25*1, 3+0.25*4) --
(3+0.25*1, 3+0.25*4) -- (3+0.25*1, 3+0.25*3.75) -- 
(3+0.25*1, 3+0.25*3.75) -- (3+0.25*0.75, 3+0.25*3.75) -- 
   cycle;
\fill[black]
 (3+0.25*3, 3+0.25*0) -- (3+0.25*3, 3+0.25*0.25) --
(3+0.25*3, 3+0.25*0.25) -- (3+0.25*3.25, 3+0.25*0.25) --
(3+0.25*3.25, 3+0.25*0.25) -- (3+0.25*3.25, 3+0.25*0) -- 
(3+0.25*3.25, 3+0.25*0) -- (3+0.25*3, 3+0.25*0) -- 
   cycle;
\fill[black]
(3+0.25*3.75, 3+0.25*0) -- (3+0.25*3.75, 3+0.25*0.25) --
(3+0.25*3.75, 3+0.25*0.25) -- (3+0.25*3.75, 3+0.25*0.25) --
(3+0.25*4, 3+0.25*0.25) -- (3+0.25*4, 3+0.25*0) -- 
(3+0.25*4, 3+0.25*0) -- (3+0.25*3.75, 3+0.25*0) -- 
   cycle;
\fill[gray]
(3+0.25*3, 3+0.25*0.75) -- (3+0.25*3, 3+0.25*1) --
(3+0.25*3, 3+0.25*1) -- (3+0.25*3.25, 3+0.25*1) --
(3+0.25*3.25, 3+0.25*1) -- (3+0.25*3.25, 3+0.25*0.75) -- 
(3+0.25*3.25, 3+0.25*0.75) -- (3+0.25*3, 3+0.25*0.75) -- 
   cycle;
\fill[gray]
(3+0.25*3.75, 3+0.25*0.75) -- (3+0.25*3.75, 3+0.25*1) --
(3+0.25*3.75, 3+0.25*1) -- (3+0.25*4, 3+0.25*1) --
(3+0.25*4, 3+0.25*1) -- (3+0.25*4, 3+0.25*0.75) -- 
(3+0.25*4, 3+0.25*0.75) -- (3+0.25*3.75, 3+0.25*0.75) -- 
   cycle;
\fill[black]
(3+0.25*3, 3+0.25*3) -- (3+0.25*3, 3+0.25*3.25) --
(3+0.25*3, 3+0.25*3.25) -- (3+0.25*3.25, 3+0.25*3.25) --
(3+0.25*3.25, 3+0.25*3.25) -- (3+0.25*3.25, 3+0.25*3) -- 
(3+0.25*3.25, 3+0.25*3) -- (3+0.25*3, 3+0.25*3) -- 
   cycle;
\fill[black]
(3+0.25*3.75, 3+0.25*3) -- (3+0.25*3.75, 3+0.25*3.25) --
(3+0.25*3.75, 3+0.25*3.25) -- (3+0.25*3.75, 3+0.25*3.25) --
(3+0.25*4, 3+0.25*3.25) -- (3+0.25*4, 3+0.25*3) -- 
(3+0.25*4, 3+0.25*3) -- (3+0.25*3.75, 3+0.25*3) -- 
   cycle;
\fill[gray]
(3+0.25*3, 3+0.25*3.75) -- (3+0.25*3, 3+0.25*4) --
(3+0.25*3, 3+0.25*4) -- (3+0.25*3.25, 3+0.25*4) --
(3+0.25*3.25, 3+0.25*4) -- (3+0.25*3.25, 3+0.25*3.75) -- 
(3+0.25*3.25, 3+0.25*3.75) -- (3+0.25*3, 3+0.25*3.75) -- 
   cycle;
\fill[gray]
(3+0.25*3.75, 3+0.25*3.75) -- (3+0.25*3.75, 3+0.25*4) --
(3+0.25*3.75, 3+0.25*4) -- (3+0.25*4, 3+0.25*4) --
(3+0.25*4, 3+0.25*4) -- (3+0.25*4, 3+0.25*3.75) -- 
(3+0.25*4, 3+0.25*3.75) -- (3+0.25*3.75, 3+0.25*3 .75) -- 
   cycle;
\end{tikzpicture}
\caption{
\label{EmnImages2}
From left to right, the images of $E^1_1$, $E^2_1$, and $E^3_1$ are shown.
The numbering of cubes is indicated in $E^1_1$ and $E^2_1$.
The omitted cubes are in gray.}
\end{figure}

Our procedure is the same as before in Section \ref{Cantor4Graph}:
We use the sub-IFS $\mathcal{F}_m$ to recursively define the nested sequence of sets $\set{E^m_n}_{n=1}^\iny$ and let $\disp E^m := \bigcap_{n=1}^\iny E_n^m$.
It holds that $E^m_n \su C_{mn}$, $E^m \su \mathcal{C}_4$, and $\dim E^m = 1 - \frac 1 {2m}$.

We let
$$E^m_{n} = \bigcup_{i=1}^{\pr{2^{2m -1}}^{n}} Q_i^{n}$$ 
where the sets are disjoint and $Q_i^{n} \prec Q_{i+1}^{n}$ for each $i \in \set{1,\ldots, \pr{2^{2m-1}}^{n} -1}$.
Images of some $E^m_n$ are in Figure \ref{EmnImages2}.

We use the sequences of sets $\set{E^m_n}_{n=1}^\iny$ along with Proposition \ref{prop graphconstruct} to establish the following result.

\begin{proposition}[Another 4-corner graph construction]
\label{construction2}
    For every $\eps \in \pr{0, \frac 1 2}$, there exists a Lipschitz graph $\Ga$ that satisfies
    $$\dim\pr{\mathcal{C}_4 \cap \Ga} \ge 1 - \eps$$
    and
    $$\Lip(\Gamma) \lesssim 2^{\frac 1 \eps}.$$
\end{proposition}

\begin{proof}
    Given $\eps \in \pr{0, 1}$, choose $m \in \N$ so $m - 1 < \frac 1 {2\eps} \le m$. By Lemma \ref{iterationOpenLemma} and Proposition \ref{openset},  $\dim E^m \ge 1 - \eps$.
    
To apply Proposition \ref{prop graphconstruct}, we check that $\set{E_n^m}_{n=1}^\iny$ satisfies the hypotheses.
As in the proof of Proposition \ref{construction1}, $\set{E_n^m}_{n=1}^\iny$ is nested, each $E_n^m \su C_{mn}$ is a finite union of disjoint cubes, each of side length $4^{-mn}$. 
The bound \eqref{convexBound} holds with $\la = \sqrt 2$, and \eqref{diamBoundAbove} holds with $\si_n = \sqrt 2 \cdot 4^{-nm}$.

In place of \eqref{connectorBound}, it suffices to show that there exists $\te_m \in \brac{0, \frac \pi 2}$ and $\la_m > 0$ so that for every $n\in \N$
\begin{equation}
\label{slopeCheck}
    \abs{\frac{P_{\te_m + \frac \pi 2}\pr{a_{i+1}^n - b_i^n}}{P_{\te_m}\pr{a_{i+1}^n-b_i^n}}} \le \la_m,
\end{equation}
where $a_i^n$ and $b_i^n$ denote the bottom-left and top-right corners, respectively, of the cubes $Q_i^n$.
Note that in contrast the the proof of Proposition \ref{construction1}, the angles between adjacent corners are not maintained through each generation, so we need to check every $n \in \N$.
In general, each of these line segments are parallel to vectors of the form $\pr{\frac 1 2, -1 + \frac{c}{4^i}}$ for $i = 1, \ldots, m$, and $\pr{\frac 1 2, -1 - \frac {2c} {4^i}}$ for $i = 2, \ldots, m$, where $c \in [3, 4)$.
By choosing
$$\te_m =  \frac \pi 2 - \frac 1 2 \pr{\arctan\pr{2 + \frac{12}{4^m}} + \arctan\pr{2 - \frac 6 {4^m}}},$$ 
\eqref{slopeCheck} holds with
$$\la_m = 
\frac{5}{18} \cdot 4^{m} + \frac 2 3 - 4^{1-m} + \sqrt{\pr{\frac{5}{18} \cdot 4^{m} + \frac 2 3 - 4^{1-m}}^2 + 1} 
\lesssim 2^{\frac 1 {\eps}}.$$
Repeating the arguments from the proof of Proposition \ref{construction1} leads to the conclusion.
\end{proof}

\begin{figure}[ht]
\centering

\caption{From left to right, the images of $\Ga^2_1$ (drawn over $E^2_1$), $\Ga^2_2$, and $\Ga^3_1$ (drawn over $E^3_1$).
$\Ga^m_n$ is the graph of $g^m_n$ and $\disp g^m = \lim_{n \to \iny} g_n^m$.}
\end{figure}

Although the Lipschitz bound here is worse than the previous construction, as shown in the next section, this procedure can be generalized to other iterated function systems.

\section{The Rotation-Free Case} 
\label{Section:norota}

In this section, we focus on iterated function systems that are rotation-free and we describe the constructions of graphs that intersect their attractors in a high-dimensional way.
To set notation, let $C$ denote the attractor of the IFS $\set{f_j}_{j=1}^N$ and let $\disp \set{C_n}_{n=0}^\iny$ denote its generations with respect to $K = \conv(C)$, as in Definition \ref{def generation}.
We use the notation $\disp \set{f_{j^{(m)}}}$ to denote the $m^{th}$ iteration of the IFS $\set{f_j}_{j=1}^N$ as described in \eqref{iteratedFunc}.

The first step in our construction is to show that for every generation $C_m$, there is an angle $\te \in \mathbb{S}^1$ onto which $C_m$ has a relatively large projection.
To argue that such an angle $\te$ exists, we use Mattila's lower bound on Favard length.
Once we have such an angle, we use a Vitali-type argument to extract a substantial subset $E_1 \su C_m$ with the property that the projections of connected components of $E_1$ are well-separated.

The set $E_1$ defines a sub-IFS $\set{\vp_\ell}_{\ell=1}^M \su \set{f_{j^{(m)}}}$ with attractor $E \su C$.
We let $\set{E_n}_{n=1}^\iny \su \set{C_{mn}}_{n=1}^\iny$ denote the generations of $\set{\vp_\ell}_{\ell=1}^M$ with respect to $K$, then we use these sets to build our graph.
The graph construction procedure described by Proposition \ref{prop graphconstruct} produces the Lipschitz graph.
We organize this section into subsections as follows:
First, we collect and prove a few results regarding Favard length, and then we describe how the sub-IFS is chosen and used to construct a graph.

\subsection{Favard length}

$\quad$ \\ 

Recall that the function $|\cdot|$ from the $\sigma$-algebra of Lebesgue measurable sets in $\R$ to $[0, \infty)$ represents the Lebesgue measure.

\begin{definition}[Favard Length]
\label{Flength}
    Let $E \su \mathbb{R}^2$ be a Borel set. 
    We define the \textbf{Favard length} by 
        \begin{align*}
            \Fav(E) := \fint_{\mathbb{S}^1} |P_{\theta}(E)| \,d\theta,
        \end{align*}
        where $P_\te$ denotes the projection onto the line of angle $\te$.
\end{definition}

The following result of Mattila, a lower bound on Favard length, allows us to choose good angles.

\begin{theorem}[\cite{mattila1990}, Theorem 4.1, Favard length lower bounds]
\label{MattilaLower}
Let $s \in (0, 1]$,  $\mu$ be a positive Borel measure on $\mathbb{R}^2$ with $\mu\pr{\mathbb{R}^2}=1$, $\operatorname{spt} \mu=E$, and for some $b \in \pr{0,\infty}$,
    \begin{align*}
        \mu \pr{B(x, r)} \leq b r^s \quad \text { for all } x \in \mathbb{R}^2, \, r \in \pr{0, \iny}.
    \end{align*}
Let $F \su \mathbb{S}^1$ be a Borel set and let $E(r)$ denote the closed $r$-neighborhood of $E$.
If $s<1$ and $0<r<\infty$, then there exists a constant $c_M = c_M(s) > 0$ so that
    \begin{align*}
        \int_F \abs{P_\te(E(r))} d \te
        \geq \frac{c_M}{b} \abs{F}^2 r^{1-s}.
    \end{align*}
If $s=1, b \geq 1$, and $0<r \leq \frac{1}{2}$, then there exists a constant $c_M = c_M(1) > 0$ so that
    \begin{align*}
        \int_F \abs{P_\te(E(r))} d \te
        \geq \frac{c_M}{b} \abs{F}^2 \pr{\log r^{-1}}^{-1}.
    \end{align*}
\end{theorem}

We introduce a Vitali-type lemma that will be used below to extract substantial subsets.

\begin{lemma}[Vitali-type lemma]
\label{VitaliLemma}
    Let $\mathcal{I}$ be a finite collection of closed intervals with $\abs{I} \ge \de > 0$ for all $I \in \mathcal{I}$.
    For any $\eps \ge 0$, there exists an $\eps$-separated subcollection $\mathcal{J} \su \mathcal{I}$ for which
    $$\bigcup_{I \in \mathcal{I}} I \su \bigcup_{J \in \mathcal{J}} \pr{3 + \frac \eps \de} J.$$
    When $\eps = 0$, the collection $\mathcal{J}$ is disjoint.
\end{lemma}

\begin{proof}
    Set $\mathcal{I}_0 = \mathcal{I}$.
    Select $J_1 \in \mathcal{I}_0$ so that $\abs{J_1} = \max \set{\abs{I} : I \in \mathcal{I}_0}$.
    Assuming that $J_1, J_2, \ldots, J_i$ have been selected, define 
    $$\mathcal{I}_{i} = \set{I \in \mathcal{I}_{i-1} :  \dist\pr{I, J_i} > \eps}.$$
    Note that since these collections of set are nested, then $\dist\pr{I, J_\ell} > \eps$ for every $I \in \mathcal{I}_{i}$ and every $\ell = 1, \ldots, i$.
    Then choose $J_{i+1} \in \mathcal{I}_i$ so that
    $$\abs{J_{i+1}} = \max \set{\abs{I} : I \in \mathcal{I}_i}.$$
    Accordingly, $\abs{J_{i+1}} \ge \abs{I}$ for every $I \in \mathcal{I}_i$.
    As this process must eventually stop, we have constructed a subcollection $\mathcal{J} = \set{J_\ell}_{\ell=1}^M$ of $\eps$-separated sets.
    
    It remains to show that $\disp \set{\pr{3 + \tfrac{\eps}{\de}} J_\ell}_{\ell=1}^M$ covers each $I \in \mathcal{I}$.
    %Let $I \in \mathcal{I}$.
    If $I \in \mathcal{J}$, then $I = J_\ell$ for some $\ell \in \set{1, \ldots, M}$.
    If $I \notin \mathcal{J}$, then there exists a smallest $\ell$ so that $I \in \mathcal{I}_{\ell - 1}$ but $I \notin \mathcal{I}_\ell$.
    Therefore, $\abs{I} \le \abs{J_{\ell}}$ and $\dist\pr{I, J_\ell} \le \eps$.
    Since $\abs{J_\ell} \ge \de$, then $I \su \pr{3 + \frac \eps \de} J_\ell$.
    In both cases, we see that
    \begin{equation*}
        I \su \bigcup_{\ell = 1}^M \pr{3 + \frac{\eps}{\de}} J_\ell. \qedhere
    \end{equation*}
\end{proof}

Now we make a comparison between the Favard length of $\de$-neighborhoods of the attractor and the Favard length of its generations.

\begin{lemma}[Favard length neighborhood comparison]
\label{FavardNbhdLemma}
    Let $\disp \set{f_j}_{j=1}^N$ be a rotation-free IFS with attractor $C$ and generations $\set{C_n}_{n=0}^\iny$ with respect to $K = \conv\pr{C}$.
    For any $n \in \N$, with
        \begin{equation}
        \label{deltaDef1}
            \delta_n := \inf_{\te \in \mathbb{S}^1} \abs{P_\te(K)} \left(\frac{r_N}{4N}\right)^n
        \end{equation} 
    it holds that
    $$\Fav(C(\delta_n)) \leq 10 \Fav(C_n).$$
\end{lemma}

\begin{proof}
If $\de_n = 0$, the result is immediate, so we may assume that $\de_n > 0$.
In this case, $K$ is $\nu$-nondegenerate, i.e. $\disp \nu = \inf_{\te \in \mathbb{S}^1} \abs{P_\te(K)} > 0$, and then $\disp \de_n = \nu \pr{\frac{r_N}{4N}}^n$.
    Since $C \su C_n$, recalling the definition of Favard length, it suffices to show that 
    for any $\theta \in \mathbb{S}^1$, 
        \begin{equation*}
            \abs{P_{\theta}(C_n(\de_n))} \leq 10 \abs{P_{\theta}(C_n)}.
        \end{equation*}
    Recall that 
    $$C_n = \bigcup_{j^{(n)}} f_{j^{(n)}}(K) = \bigcup_{j^{(n)}} K_{j^{(n)}}.$$
    That is, $C_n$ consists of $N^n$ translated and rescaled copies of $K$, where each rescaling is of the form $r_{j_1} \ldots r_{j_n}$.    
    Let $\theta \in \mathbb{S}^1$ and observe that
    $$P_{\theta}(C_n) 
    = \bigcup_{j^{(n)}} P_{\theta}\pr{K_{j^{(n)}}}
    =: \bigcup_{k=1}^{N^n} I_k.$$
    Assume that the closed intervals $I_k$ are indexed by increasing length and choose $m_0 \in \set{1, \ldots, N^n}$ so that
    \begin{align*}
        & |I_{k}| < \de_n \quad \text{ for all } k = 1, \ldots, m_0 - 1 \\
        & |I_{k}| \ge \de_n \quad \text{ for all } k = m_0, \ldots, N^n.
    \end{align*}
    For $k \le m_0 - 1$, the intervals are bounded from above and we see that 
    \begin{equation}
    \label{eq smallints}
        \begin{aligned}
        \left|\bigcup_{k=1}^{m_0-1} I_k(\delta_n) \right| 
        &\leq \sum_{k=1}^{m_0-1} |I_k(\delta_n)| 
        \leq \sum_{k=1}^{m_0-1} \pr{|I_k|+2\delta_n}
        < \sum_{k=1}^{m_0-1} 3 \de_n
        = 3 \pr{m_0 - 1} \de_n \\
        &< 3 N^n \nu \pr{\frac{r_N}{4N}}^n 
        = \frac{3}{4^n} \nu r^n_N 
        \leq \frac{3}{4^n} |I_{N^n}|
        \leq |P_{\theta}C_n|. 
        \end{aligned}
    \end{equation}
    For $k \ge m_0$, each $I_k$ is an interval of length $r_{j_1} \ldots r_{j_n} \abs{P_\te(K)} \ge \de_n$, so it follows that $I_k(\delta_n) \su 3 I_k$.
    An application of Lemma \ref{VitaliLemma} with $\mathcal{I}_0 := \set{3 I_k}_{k=m_0}^{N^n}$, $\de_n > 0$ as given, and $\eps = 0$ shows that there exists a disjoint subcollection $\mathcal{J} := \set{J_\ell}_{\ell = 1}^M \su \mathcal{I}_0$ for which
    $$\bigcup_{k=m_0}^{N^n} 3 I_k \su \bigcup_{\ell = 1}^M 3 J_\ell.$$
    As a consequence, we deduce that
    \begin{equation}
    \label{upperPnBound}
        \abs{\bigcup_{k=m_0}^{N^n} I_k(\de_n)} 
    \le \abs{\bigcup_{k=m_0}^{N^n} 3 I_k}
    \le \abs{\bigcup_{\ell = 1}^M 3 J_\ell} 
    \le 3 \sum_{\ell = 1}^M \abs{J_\ell}.
    \end{equation}
    As $\disp \set{J_\ell}_{\ell = 1}^M$ is disjoint, then so is $\disp \set{\tfrac 1 3 J_\ell}_{\ell = 1}^M$.
    Since $\disp \set{J_\ell}_{\ell = 1}^M \su \set{3 I_k}_{k=m_0}^{N^n}$ implies that $\disp \set{\tfrac 1 3 J_\ell}_{\ell = 1}^M \su \set{I_k}_{k=m_0}^{N^n}$, then 
    \begin{align*}
        \bigcup_{\ell=1}^M \frac 1 3 J_\ell
        \su \bigcup_{k=m_0}^{N^n} I_k  \su P_{\theta}(C_n)
    \end{align*}
    from which it follows that
    \begin{align*}
        \frac 1 3 \sum_{\ell=1}^M  \abs{J_\ell} 
        =\abs{\bigcup_{\ell=1}^M \frac 1 3 J_\ell} 
        \leq \abs{P_\te\pr{C_n}}.
    \end{align*}
    Combining \eqref{upperPnBound} with this bound shows that
    \begin{equation*}
        \abs{\bigcup_{k=m_0}^{N^n}I_k(\de_n)} 
        \le 3 \sum_{\ell = 1}^M \abs{J_\ell} 
        \le 9 \abs{P_\te\pr{C_n}} .
    \end{equation*}
    Finally, putting this bound together with \eqref{eq smallints} shows that
    \begin{equation*}
        \abs{P_{\theta}(C_n(\de_n))} \leq \abs{\bigcup_{k=1}^{m_0-1} I_k(\de_n)}+\abs{\bigcup_{k=m_0}^{N^n}I_k(\de_n)} 
        \le  10 \abs{P_\te\pr{C_n}} . \qedhere
    \end{equation*}    
\end{proof}

\subsection{Extracting a sub-IFS and building the graph}

$\quad$ \\ 

By combining Theorem \ref{MattilaLower} with Lemma \ref{FavardNbhdLemma} and Lemma \ref{VitaliLemma}, we construct a sub-IFS with a similarity dimension that is close to $1$ and whose generations have well-separated projections. 
We then show that the projection separation condition allows us to apply Proposition \ref{prop graphconstruct} and construct a graph that intersects the attractor in a high-dimensional way.

The following proposition describes how the sub-IFS is produced.

\begin{proposition}[Constructing a sub-IFS with substantial dimension]
\label{subIFSconstruction}
    Let $\disp \set{f_j}_{j=1}^N$ be a rotation-free, $\nu$-non-degenerate IFS that satisfies the open set condition with similarity dimension $1$.
    Let $C$ be the attractor of $\disp \set{f_j}_{j=1}^N$, let $b$ denote the Ahlfors upper constant of $C$, and set $K = \conv\pr{C}$.
 
    There exists $N_0(N, r_N, \nu) \geq 0$ such that for every $m \ge N_0$, there exists $\theta = \te(m) \in \mathbb{S}^1$ and an IFS $\set{\vp_\ell}_{\ell=1}^M \su \set{f_{j^{(m)}}}$ with the following properties:
    \begin{enumerate}
        \item The set $\disp E := \bigcup_{\ell = 1}^M \vp_\ell(K)$ has $M \le N^m$ connected components and $P_{\theta}(E)$ is a union of $M$ $\de_m$-separated intervals, where $\disp \de_m = \nu \pr{\frac{r_N}{4N}}^{m}$.
        \item The similarity dimension, $s = s(m)$, of $\set{\vp_\ell}_{\ell=1}^M$ satisfies
        \begin{equation}
        \label{simDimm}
            s \ge s_0 := 1 - \frac{\log m + \log c_1 }{m \log\pr{{r_N}^{-1}}},
        \end{equation}
        where $c_1 = \frac{16}{3c_M} b \log\pr{\frac{4N}{r_N}} \diam(K)$ and $c_M$ is the universal constant from Theorem \ref{MattilaLower}. 
    \end{enumerate}
\end{proposition}

\begin{proof} 
    Since $\disp \set{f_j}_{j=1}^N$ and $C$ satisfy the hypotheses of Theorem \ref{MattilaLower} with $s = 1$, then with $F = \mathbb{S}^1$, we deduce that for any $\de > 0$,
    $$
    \Fav(C(\de)) \ge \frac{c_M}{b \log \pr{\de^{-1}}}.
    $$
    For $m \in \N$ and $\de_m = \nu \pr{\frac {r_N}{4N}}^{m} > 0$ as in \eqref{deltaDef1}, an application of Lemma \ref{FavardNbhdLemma} shows that
    \begin{align*}
        \Fav(C_m) 
        \ge \frac{c_M}{ 10b\pr{ m \log(4Nr_N^{-1}) - \log \nu}}.
    \end{align*}
    As such, whenever $m \ge N_0(N, \nu):= \max\set{1, \frac{10 \log \nu}{9 \log\pr{4Nr_N^{-1}}}}$, there exists $\theta \in \mathbb{S}^1$ so that
    \begin{equation}
    \label{lowerProjBoundm1}
        \abs{P_\te(C_m)}
        \ge \frac{c_M}{b m\log\pr{4Nr_N^{-1}} }.
    \end{equation}
    Recall that $C_m$ can be written as
    $$C_m = \bigcup_{j^{(m)}} K_{j^{(m)}},$$
    where each set $K_{j^{(m)}}$ is a translated and $r_{j_1} \ldots r_{j_m}$-rescaled copy of $K$.
    For $\theta \in \mathbb{S}^1$, we have
    $$P_{\theta}(C_m) 
    = \bigcup_{j^{(m)}} P_{\theta}\pr{K_{j^{(m)}}}
    =: \bigcup_{j^{(m)}} I_{j^{(m)}}
    = \pr{\bigcup_{j^{(m)} \in \mathcal{L}_m} I_{j^{(m)}}} \cup \pr{\bigcup_{j^{(m)} \in \mathcal{H}_m} I_{j^{(m)}}},$$
    where we introduce 
    \begin{align*}
     &\mathcal{L}_m := \set{j^{(m)} \in \set{1, \ldots, N}^m : \abs{P_{\theta}\pr{K_{j^{(m)}}}} < \de_m} \\
     &\mathcal{H}_m := \set{j^{(m)} \in \set{1, \ldots, N}^m : \abs{P_{\theta}\pr{K_{j^{(m)}}}} \ge \de_m}.
    \end{align*}
    For $j^{(m)} \in \mathcal{L}_m$, the lengths of the intervals are bounded from above and we see that 
    \begin{equation*}
        \abs{\bigcup_{j^{(m)} \in \mathcal{L}_m} I_{j^{(m)}}}  \leq \sum_{j^{(m)} \in \mathcal{L}_m} \abs{I_{j^{(m)}}}
        < N^m \de_m
        = 4^{-m} \nu r^m_N 
       % \le 4^{-m} |I_{N^m}|
        \leq 4^{-m} |P_{\theta}C_m|,
    \end{equation*}
    from which it follows that
    \begin{align*}
        \abs{\bigcup_{j^{(m)} \in \mathcal{H}_m} I_{j^{(m)}}} \geq (1-4^{-m}) \abs{P_{\theta}(C_m)} 
        \geq \frac{3}{4} \abs{P_{\theta}(C_m)} .
    \end{align*}
    An application of Lemma \ref{VitaliLemma} shows that there exists a $\de_m$-separated subcollection $\disp \set{J_\ell}_{\ell = 1}^{M} \su \set{I_{j^{(m)}}}_{j^{(m)} \in \mathcal{H}_m}$ with the property that  
    \begin{align*}
        \bigcup_{{j^{(m)} \in \mathcal{H}_m}} I_{j^{(m)}} \su \bigcup_{\ell = 1}^{M} 4 J_\ell.
    \end{align*}
    In combination with \eqref{lowerProjBoundm1}, it follows that
    \begin{equation}
    \label{JellLength}
        \sum_{\ell =1}^{M} \abs{J_\ell}
        \ge \frac{3c_M}{16 bm \log\pr{4Nr_N^{-1}} }
        =: \frac{1}{c_0 m}.
    \end{equation}
    For each $\ell \in \set{1, \ldots, M}$, there exists $j^{(m)} \in \set{1, \ldots, N}^m$ so that $J_\ell = I_{j^{(m)}}$.
    Define each $\vp_\ell = f_{j^{(m)}}$ and note that $\vp_\ell(x) = \rho_\ell x + \zeta_\ell$, where $\rho_\ell =r_{j_1} \ldots r_{j_m}$. 
    With $\disp E := \bigcup_{\ell =1}^M \vp_\ell(K) \su C_m$, we see that $\disp P_\te \pr{E} = \bigcup_{\ell =1}^M P_\te \pr{\vp_\ell(K)} = \bigcup_{\ell =1}^M J_\ell$
    so that, by construction, $P_{\theta}(E)$ is a union of $M$ $\de_m$-separated intervals.
    
    With $s_0$ as in \eqref{simDimm}, observe that since $\rho_\ell \le r_N^m$ and $c_1 = c_0 \diam(K)$, then 
    \begin{align*}
        \pr{\frac 1 {\rho_\ell}}^{1-s_0}
        &= \pr{\frac 1 {\rho_\ell}}^{\frac{\log m + \log \pr{c_1} }{m \log\pr{\frac 1 {r_N}}}}
        = \pr{c_1 m }^{ \frac{\log\pr{\frac 1 {\rho_\ell}}}{\log\pr{\frac 1 {r_N^m}}}} 
        \ge {c_0} {\diam(K)} m .
    \end{align*}
    Therefore, using that $\abs{J_\ell} = \abs{P_\te(K)} \rho_\ell$ and \eqref{JellLength}, we see that
    \begin{align*}
        \sum_{\ell = 1}^M \rho_\ell^{s_0}
        = \sum_{\ell = 1}^M \pr{\frac 1{\rho_\ell}}^{1 - s_0} \rho_\ell
        \ge \frac {c_0 \diam(K) m } {\abs{P_\te(K)}} \sum_{\ell = 1}^M \abs{J_\ell}
        \ge \frac {\diam(K)} {\abs{P_\te(K)}}
        \ge 1.
    \end{align*}
    Thus, the IFS $\set{\vp_\ell}_{\ell=1}^M$ has similarity dimension $s \ge s_0$.
\end{proof}

\begin{remark}
\label{rem:P6.5}
If there exists $\te \in [0, \pi]$ so that $\abs{P_\te(C)} \ge \bar{c} > 0$, then for each $n \in \N$, $\abs{P_\te(C_n)} \ge \bar{c}$.
In this case, following the arguments in the proof of Proposition \ref{subIFSconstruction} leads to a similar result where now $\disp s \ge s_0 := 1 - \frac{\log \bar{c}_1}{m \log(r_N^{-1})}$.
If $\disp \set{f_j}_{j=1}^N$ is a rotation-free, $\nu$-non-degenerate IFS that satisfies the open set condition with similarity dimension greater than $1$, the existence of $\te$ and $\bar{c}$ is guaranteed.
In fact, there are such IFS with similarity dimension equal to $1$.
For example, for the 4-corner Cantor set, we can take $\te = \arctan(1/2)$ and $\bar{c} = \frac{3}{\sqrt 5}$.
\end{remark}

With Proposition \ref{subIFSconstruction} and Proposition \ref{prop graphconstruct}, we now establish the following theorem, a version of Theorem \ref{mainthm}.

\begin{theorem}[Theorem \ref{mainthm} in the rotation-free case]
\label{rotFreeGraph}
    Let $\disp \set{f_j}_{j=1}^N$ be a rotation-free, $\nu$-non-degenerate IFS that satisfies the open set condition with similarity dimension $1$.
    Let $C$ be the attractor of $\disp \set{f_j}_{j=1}^N$, let $b$ denote the Ahlfors upper constant of $C$, and set $K = \conv\pr{C}$.
    
    There exist constants $\eps_0\pr{N, r_N, b, \nu, \diam(K)} \in \pr{0, 1}$ and $c_0\pr{N, r_N} > 0$ so that for any $\eps \in \pb{0,\eps_0}$, there exists a Lipschitz graph $\Ga$ for which 
    $$\dim\pr{C \cap \Ga} \ge 1 - \eps$$ 
    and 
    $$\Lip(\Ga) \le \frac{\diam(K)}{\nu} \exp\brac{c_0 \eps^{-1} \log\pr{\eps^{-1}}}.$$ 
    The constant $c_0$ is given by $c_0 = \log\pr{4Nr_N^{-1}} \max\set{1, \frac{3}{\log\pr{r_N^{-1}}}}$.
\end{theorem}

\begin{proof}
    Let $N_0(N, r_N, \nu)$ be given by Proposition \ref{subIFSconstruction}. 
    Given $\eps \in \pr{0,1}$, choose $m \ge N_0(N, r_N, \nu)$ so that with $s_0(m)$ as defined in \eqref{simDimm}, $s_0(m) \ge 1 - \eps$.
    We can do this if we choose $m \in \N$ large enough so that
    \begin{equation}
    \label{mChoice}
        \frac{m\log\pr{r_N^{-1}}}{\log m + \log c_1 } \ge \frac 1 \eps.
    \end{equation}
    Define $\eps_0$ to satisfy $\eps_0 \log\pr{\eps_0^{-1}} \le c_1^{-1}$ 
    and $\eps_0 \le \min\set{\frac {\log\pr{r_N^{-1}}} 3, N_0^{-1}}$.
    If $\eps \le \eps_0$ and $c_2 = \max\set{1, \frac{3}{\log\pr{r_N^{-1}}}}$, then $\disp m =\lceil c_2 \eps^{-1} \log\pr{\eps^{-1}}\rceil \ge N_0$ is large enough to satisfy \eqref{mChoice}.
    To see this, note that
    \begin{align*}
        & \frac{c_2 \eps^{-1} \log\pr{\eps^{-1}}\log\pr{r_N^{-1}}}{\log \pr{c_2 \eps^{-1} \log\pr{\eps^{-1}}} + \log c_1 } 
        \ge \frac 1 \eps \\
        \iff & 
        \pr{c_2 \log\pr{r_N^{-1}}  - 2}\log\pr{\eps^{-1}} 
        \ge \log \pr{c_1 \eps \log\pr{\eps^{-1}}}
        + \log c_2.
    \end{align*}
    The first condition on $\eps_0$ ensures that $\log \pr{c_1 \eps \log\pr{\eps^{-1}}} \le 0$.
    Since $c_2 \log \pr{r_N^{-1}} \ge 3$, then the above inequality holds if $\log\pr{\eps^{-1}} \ge \log c_2$, or $c_2 \le \eps^{-1}$, which holds by the conditions on $c_2$ and $\eps_0$.
    
    An application of Proposition \ref{subIFSconstruction} produces an angle $\te \in \brac{0, \pi}$ and an IFS $\set{\vp_\ell}_{\ell=1}^M \su \set{f_{j^{(m)}}}$ with similarity dimension $s \ge 1 - \eps$.
    Let $E$ denote the attractor of $\set{\vp_\ell}_{\ell=1}^M$. 
    Then Lemma \ref{iterationOpenLemma} and Proposition \ref{openset} imply $\dim(E) \ge 1 - \eps$.
    By Corollary \ref{generCor}, $\disp E = \bigcap_{n=1}^\iny E_n,$ where the $n^{th}$ generation of $E$ with respect to $K$ is given by
    $$E_n := \bigcup_{\ell^{(n)}} \vp_{\ell^{(n)}}\pr{K} \su C_{mn}.$$
    
    Now we check that $\set{E_n}_{n=1}^\iny$ satisfies the hypotheses of Proposition \ref{prop graphconstruct}. 
    By construction, $\set{E_n}_{n=1}^\iny$ is nested and each $\disp E_n := \bigcup_{j=1}^{M^n} K^n_j$ consists of $M^n$ connected components, each a translated and rescaled copy of $K$.
    It follows  
    that \eqref{convexBound} holds with $\la = {\diam(K)} \nu^{-1}$.
    For each $n, j$, $\diam\pr{K_j^n} \le \diam(K) {r_N}^{m n}$ and then \eqref{diamBoundAbove} holds with $\si_n = \diam(K) {r_N}^{m n}$.

    Without loss of generality, let $\te = 0$.
    For some $n \in \N$, let $z_i \in K_i^n$ and $z_{j} \in K_{j}^n$ for $i \ne j$.
    There exists a largest $k \in \set{1, \ldots, n}$ so that $z_i$ and $z_j$ both belong to the same connected component of $E_{k-1}$.
    There is no loss of generality in assuming that $k = n$.
    After rescaling, we may assume that $n = 1$, i.e. let $z_i \in K_i^1$ and $z_j \in K_j^1$ belong to distinct connected components of $E_1$.
    With $E = E_1$, Proposition \ref{subIFSconstruction} shows that $\disp \abs{P_x(z_i - z_j)}\ge \de_m = \nu \pr{\frac{r_N}{4N}}^m$, while $\abs{P_y(z_i - z_j)}\le \diam(K)$. 
    It follows that for any $z_i \in K_i^n$ and any $z_{j} \in K_{j}^n$, $i \ne j$, 
    $$\frac{\abs{P_{y}\pr{z_{j} - z_i}}}{\abs{P_{x}\pr{z_{j} - z_i}}} 
    \le \frac{\diam(K)}{\nu \pr{\frac{r_N}{4N}}^m}
    =   \frac{\diam(K)}{\nu} e^{c_0 \eps^{-1} \log\pr{\eps^{-1}} },$$
    where $c_0 = c_2 \log\pr{4Nr_N^{-1}}$.
    The hypotheses of Proposition \ref{prop graphconstruct} have been verified.
    Let $I = P_x(K)$.
    An application of Proposition \ref{prop graphconstruct} shows that there exists a Lipschitz function $g \colon I \to \R$ with graph $\Gamma = \set{\pr{x, g(x)} : x \in I}$, $E \su \Gamma$, and $\disp \Lip(\Gamma) \leq \frac{\diam(K)}{\nu}e^{c_0 \eps^{-1} \log\pr{\eps^{-1}} }$.
    Since $E \cap C = E$, then 
    \begin{equation}
        \dim\pr{\Gamma \cap C} \ge \dim\pr{E} \ge 1 - \eps.\qedhere
    \end{equation}
\end{proof}

\begin{remark}
Following Remark \ref{rem:P6.5}, if there exists $\te \in [0, \pi]$ so that $\abs{P_\te(C)} \ge \bar{c} > 0$, then with $\disp s \ge 1 - \frac{\log \bar{c}_1}{m \log(r_N^{-1})}$, we can set $\eps_0 = \frac{\ln \bar{c}_1}{N_0 \ln(r_N^{-1})}$ and $m = \lceil \frac{\ln \bar{c}_1}{\eps \ln(r_N^{-1})}\rceil \ge N_0$.
Then we see that
\begin{align*}
\Lip(\Ga) \le
\frac{\diam(K)}{\nu \pr{\frac{r_N}{4N}}^m}
&%= \frac{\diam(K)}{\nu} \exp\brac{\frac{\ln c_1}{\eps \ln(r_N^{-1})}\ln \pr{4N r_N^{-1}}}
= \frac{\diam(K)}{\nu} \exp\pr{c_0 \eps^{-1}}.
\end{align*}
In comparison to the conclusion of Theorem \ref{rotFreeGraph}, this is an improvement.
Note that this bound is similar to the one given in Proposition \ref{construction2} for our generic graph construction for the 4-corner Cantor set.
\end{remark}

\section{The Rotational Case} 
\label{Section:rota}

In this section, we describe the constructions of graphs that intersect the attractors of general (rotational) iterated function systems in a high-dimensional way.
In comparison to the previous section, our procedure here is more delicate.
In the rotation-free case, a single scale with good projection properties gives rise to an arithmetic progression of scales with good projections, which allows us to build a Lipschitz graph using Proposition \ref{prop graphconstruct}.
With the presence of rotations, the projections become quite complicated, and choosing a good projection angle becomes more challenging.
As such, the construction of a Lipschitz graph in the presence of rotations requires new ideas.

Our notation here differs from the previous section.
We start with an IFS $\disp \set{f_k}_{k=1}^M$ with attractor $H$ and $K = \conv(H)$.
The first step, described in Subsection \ref{subsec UnifReduc}, is to extract a uniform sub-IFS, see Definition \ref{IFSDef}.
To do this, we follow the ideas in \cite{peres2009resonance} and produce an IFS $\disp \set{\vp_j}_{j=1}^N \su \set{f_{k^{(\kappa)}}}$ with attractor $C$, scale factor $r$, and rotation $A$.
We choose $\kappa \gg 1$ so that $r$ has appropriate bounds, and the dimension of $C$, which we denote by $\ga$, suitably depends on $\varepsilon$.
This subsection also establishes Ahlfors upper regularity and lower measure bounds for $C$, both depending explicitly on $\eps$.
In the second step, detailed in Subsection \ref{subsec bigproj}, we apply the quantitative projection bounds from \cite{peres2009resonance} to establish large subsets of ``good'' projections for each fixed scale. 
To ensure that there exists an angle that projects well at many scales, i.e. plays well with the rotations, we apply the Maximal Ergodic Theorem presented in subsection \ref{subsec Ergodic}.
In Subsection \ref{subsec rotgraphconstr}, we combine the tools and observations from the previous subsections to construct a nested collection $\set{E_n}_{n=1}^\iny$, where each $E_n$ is a subset of the $n$th generation of $\disp \set{\vp_j}_{j=1}^N$ with respect to $K$.
As in the previous case, the sets $\set{E_n}_{n=1}^\iny$ satisfy the hypotheses of the graph construction algorithm described by Proposition \ref{prop graphconstruct}, so we apply this result to prove the main theorem of this section.

\subsection{Reduction to a uniform iterated function system} 
\label{subsec UnifReduc}
 
$\quad$ \\    

In this subsection, we show that from a general IFS that satisfies the open set condition, we can produce a uniform IFS (UIFS) within a specified dimension range, provide bounds on the scale factor,  a bound on the upper Ahlfors constant, and a lower bound on the measure of the attractor of the UIFS.
Our starting point is the following proposition of Peres and Shmerkin \cite{peres2009resonance} which shows how to produce a UIFS from an iteration of a general IFS.

\begin{proposition}[Proposition 6 in \cite{peres2009resonance}, Reduction to a UIFS]
\label{PSProp6}
Let $\set{f_k}_{k=1}^M$, where each $f_k$ is of the form \eqref{IFSDef}, be an IFS that satisfies the open set condition and has attractor $H$.
For any $\eta>0$, there exists $\kappa \in \N$ and a UIFS $\disp \set{\vp_j}_{j=1}^L \su \set{f_{k^{(\kappa)}}}$ with attractor $\widetilde{C} \su H$ and $\disp \dim(\widetilde{C}) \in \brac{\dim(H)-\eta, \dim(H)}$.
\end{proposition}

The following lemma shows that once $\eta$ is sufficiently small, we can eliminate some of the functions in the IFS produced in the previous result to ensure that the dimension drops.

\begin{lemma}[Dimension drop lemma]
\label{PSPropDrop}
Let $\set{f_k}_{k=1}^M$, where each $f_k$ is of the form \eqref{IFSDef}, be an IFS that satisfies the open set condition and has attractor $H$.
There exists $\eta_0\pr{r_1, \ldots, r_M, M, \dim(H)} > 0$ so that for any $\eta \in (0, \eta_0]$, there exists $\kappa \in \N$ and a UIFS $\disp \set{\vp_j}_{j=1}^N \su \set{f_{k^{(\kappa)}}}$ with attractor $C$, such that $C \su H$, $\disp \dim(C) \in \brac{\dim(H)-\eta, \dim(H) - \frac \eta 2}$, and the scale factor $r$ satisfies
    \begin{equation}
    \label{rBound}
        c_1 \pr{c_2 \eta}^{\frac{c_3} \eta} 
        \leq r
        \le  \pr{\tfrac{3}{2} c_2 \eta}^{\frac{c_3}{3 \eta} },
    \end{equation}
     where 
     \begin{equation}
     \label{c123Defn}
        c_1 = \prod_{k=1}^M r_k, \quad
        c_2 = \frac{4}{3M}\sum_{k=1}^M r_k^{\dim(H)}\log\pr{r_k^{-1}}, \quad
        c_3 = \frac {3M} 2.
     \end{equation}
\end{lemma}

\begin{proof}
    By Proposition \ref{PSProp6}, there exists $\kappa \in \N$ and a UIFS $\set{\vp_j}_{j=1}^L \su \set{f_{k^{(\kappa)}}}$ with attractor $\widetilde{C} \su H$.
    The proof of Proposition \ref{PSProp6} in \cite[Proposition 6]{peres2009resonance} shows that $\vp_j(x) = r A x+z_j$ for $j = 1, \ldots, L$, with
    \begin{equation}
    \label{rLDef}
        r = \prod_{k=1}^M r_k^{v_\kappa} \quad \text{ and } \quad
        L \ge c_p \kappa^{- \frac{M-1}{2}} \prod_{k=1}^M r_k^{- \dim(H) v_\kappa}, 
    \end{equation}
    where $c_p > 0$ is a constant inherited from an application of the Central Limit Theorem and  $v_\kappa = \lceil \kappa r_k^{\dim(H)} \rceil$. 
    We show that by reducing the number of elements in this UIFS, we can get a sub-UIFS that satisfies the statement of the lemma.
    
    Define 
    \begin{align*}
        N =
        &\lfloor \kappa^{-\frac {2M-1} 4} \prod_{k=1}^M r_k^{- \dim(H) v_\kappa} \rfloor 
        = \kappa^{-\frac {2M-1} 4} \prod_{k=1}^M r_k^{- \dim(H) v_\kappa} - \eps,
    \end{align*}
    where $\eps \in [0, 1)$.
    If $\kappa \ge c_p^{-4}$, then $N \le L$ so that $\set{\vp_j}_{j=1}^{N} \su \set{\vp_j}_{j=1}^{L}$.
    If we let $C \su \widetilde{C}$ denote the attractor of the sub-UIFS $\set{\vp_j}_{j=1}^{N}$, then its similarity dimension is given by
    \begin{equation*}
        \dim(C)
        = \frac{\log N}{ \log \pr{\frac 1 r}}
        = \dim (H) - \frac{\frac{2M-1}{4} \log \kappa - \log \pr{1 - \eps_0}}{\sum_{k=1}^M v_\kappa \log\pr{r_k^{-1}}},
    \end{equation*}
    where we let $\disp \eps_0 = \eps_0(\kappa) = \eps \kappa^{\frac {2M-1} 4} \prod_{k=1}^M r_k^{\dim(H) v_\kappa} \ll 1$.
    Observe that whenever $\kappa \ge \frac{3M}{M-2}r_1^{- \dim(H)}$, we get
    \begin{align*}
        & \kappa r_k^{\dim(H)} \le v_\kappa \le \kappa r_k^{\dim(H)} +1 \le \frac{4M - 2}{3M}\kappa r_k^{\dim(H)} ,
    \end{align*}
    and whenever $\kappa \ge \frac 1 {\pr{1 - \eps_0(c_p^{-4})}^{4}}$, we get
    \begin{align*}
        \frac{2M-1}{4} \log \kappa
        \le \frac{2M-1}{4} \log \kappa - \log \pr{1 - \eps_0}
        \le \frac M 2 \log \kappa.
    \end{align*}
    Therefore, if $\kappa \ge \kappa_0 := \max\set{c_p^{-4}, \frac{3M}{M-2}r_1^{- \dim(H)}, \frac 1 {\pr{1 - \eps_0(c_p^{-4})}^{4}}}$, then $N \le L$ and it follows that
    \begin{align*}
        \frac{3M}{8T} \frac{\log \kappa}{\kappa}
        \le \dim(H) - \dim(C) 
        \le \frac{M}{2T} \frac{\log \kappa}{\kappa},
    \end{align*}
    where we introduce $\disp T = \sum_{k=1}^M r_k^{\dim(H)}\log\pr{r_k^{-1}}$.
   
    If we choose $\kappa \in \N$ satisfying $\disp \frac{M}{2T} \frac{\log \kappa}{\kappa} \in \brac{\frac {2\eta} 3, \eta}$, then $\disp \brac{\frac{3M}{8T} \frac{\log \kappa}{\kappa}, \frac{M}{2T} \frac{\log \kappa}{\kappa}} \su \brac{\frac \eta 2, \eta}$ which implies that $\disp \dim(H) - \dim(C) \in \brac{\frac \eta 2, \eta}$.
       
    Define the function $g \colon (e, \iny) \to \R$ by $g(x) = \frac x {\log x}$.
    Since $g'(x) = \frac{\log x - 1}{\pr{\log x}^2} > 0$ on its domain, then $g^{-1}\colon (e, \iny) \to \R$ is well-defined and increasing.
    Therefore, $\disp \frac{M}{2T} \frac{\log \kappa}{\kappa} \in \brac{\frac {2\eta} 3, \eta}$ is equivalent to
    \begin{equation}
    \label{kappaChoice}
      \kappa_1 \le \kappa \le \kappa_2,  
    \end{equation}
    where 
    \begin{equation}
    \label{kappDefs}
        \kappa_1 = g^{-1}\pr{\tfrac{M}{2T}\eta^{-1}} \quad \text{and} \quad \kappa_2 = g^{-1}\pr{\tfrac{3M}{4T}\eta^{-1}}.
    \end{equation}
    Now we to check that $[\kappa_1, \kappa_2] \cap \N \ne \varnothing$.
    By the Mean Value Theorem, there exists $\bar \kappa \in (\kappa_1, \kappa_2)$ so that
    \begin{align*}
        \kappa_2 - \kappa_1
        &= \frac{g(\kappa_2) - g(\kappa_1)}{g'(\bar \kappa)}.
    \end{align*}
    It follows that $\kappa_2 - \kappa_1 \ge 1$ if and only if  $\disp g'(\bar \kappa) \le g(\kappa_2) - g(\kappa_1) = \tfrac{M}{4T}\eta^{-1}$.
    As $g'(x) \le g'(e^2) = \frac 1 4$, then $\kappa_2 - \kappa_1 \ge 1$ whenever $\eta \le \frac M T$.
    In this case, $\brac{\kappa_1, \kappa_2} \cap \N \ne \varnothing$, so there exists $\kappa \in \brac{\kappa_1, \kappa_2} \cap \N$ for which $\disp \frac{M}{2T} \frac{\log \kappa}{\kappa} \in \brac{\frac {2\eta} 3, \eta}$.

    Since we also need $\kappa \ge \kappa_0$, it suffices to ensure that $\kappa_1 \ge \kappa_0$ for all admissible choices of $\eta$, i.e., we need $g^{-1}\pr{\tfrac{M}{2T} \eta^{-1}} \ge \kappa_0$, or 
    \begin{align*}
        \eta \le \eta_0 := \frac{M \log \kappa_0}{2T \kappa_0} < \frac M T.
    \end{align*}
    That is, if $\eta \le \eta_0$, then there exists $\kappa \in \N$ with $\kappa \ge \kappa_0$ for which \eqref{kappaChoice} holds, and then $\dim(C) \in \brac{\dim(H) - \eta, \dim(H) - \frac \eta 2}.$
    
    From the definition of $r$ in \eqref{rLDef} and the bounds on $\kappa$ in \eqref{kappaChoice} and \eqref{kappDefs}, we observe that 
    \begin{align*}
        e^{-S} \exp\brac{-T g^{-1}\pr{\tfrac{3M}{4T}\eta^{-1}}}
        \leq r
        \le  \exp\brac{-T g^{-1}\pr{\tfrac{M}{2T}\eta^{-1}}},
    \end{align*}
     where we let $\disp S = \sum_{k=1}^M \log(r_k^{-1})$. 
     Since $x \log x \le g^{-1}(x) \le 2 x \log x$ whenever $x \ge e$, then 
     \begin{align*}
         T g^{-1}\pr{\tfrac{M}{2T}\eta^{-1}}
         &\ge T \tfrac{M}{2T}\eta^{-1} \log\pr{\tfrac{M}{2T}\eta^{-1}}
         = \log\pr{\tfrac{M}{2T\eta}}^{\tfrac{M}{2\eta} } \\
         T g^{-1}\pr{\tfrac{3M}{4T}\eta^{-1}}
         &\le 2T \tfrac{3M}{4T}\eta^{-1} \log\pr{\tfrac{3M}{4T}\eta^{-1}}
         = \log\pr{\tfrac{3M}{4T\eta}}^{\tfrac{3M}{2\eta} }.
     \end{align*}
     Plugging this into the previous expression gives
     \begin{align*}
        e^{-S} \pr{\tfrac{4T}{3M} \eta}^{\tfrac{3M}{2} \eta^{-1} } 
        \leq r
        \le  \pr{\tfrac{2T}{M} \eta}^{\tfrac{M}{2} \eta^{-1} } ,
    \end{align*}
    which we can rewrite as \eqref{rBound}.
\end{proof}

The next result shows that a uniform sub-IFS inherits upper bounds on the Ahlfors upper constant, and lower bounds on the measure of the attractor.

\begin{lemma}[Upper Ahlfors regularity of a uniform sub-IFS]
\label{PSMMProp}
Let $\set{f_k}_{k=1}^M$ 
be an IFS that satisfies the open set condition, has attractor $H$, Ahlfors lower and upper constants $a$ and $b$, respectively, and overlapping index $\tau$.

If for some $\kappa \in \N$, $\set{\vp_j}_{j = 1}^N \su \set{f_{k^{(\kappa)}}}$ is a UIFS with scale factor $r$, attractor $C \su H$, and $\gamma :=\dim(C) < \dim(H)$, then for all $x \in C$ and all $\rho \in (0, 1)$, it holds that
    \begin{equation}
    \label{AhlforUpperSub}
        \mathcal{H}^{\gamma}\pr{ C \cap B(x, \rho)}
        \le \frac{\tau b}{a} \brac{\frac{8 \diam(K)}{\nu(K)}}^{\dim(H)}  r^{\gamma\dim(H)} \rho^{\gamma},
    \end{equation}
    and
    \begin{equation}
    \label{measureLowerBoundsub}
        \mathcal{H}^{\gamma}(C)
        \geq \frac{a}{\tau b } \brac{\frac{\nu(K)}{8}}^{\dim(H)} \brac{\frac r {\diam(K)}}^{\dim(H) - \ga}   ,  
\end{equation}
    where $K = \conv(H)$. 
\end{lemma}

\begin{proof}
    Since  
    $\set{\vp_j}_{j = 1}^N \su \set{f_{k^{(\kappa)}}}$ for some $\kappa \in \N$, 
    then Corollary \ref{generCor} shows that we may use $K = \conv(H)$ to generate $C$.
    That is, $\disp C = \bigcap_{n=1}^\iny C_n$, where each generation $C_n$ is defined as $\disp C_n = \bigcup_{j^{(n)}} K_{j^{(n)}}$ with $K_{j^{(n)}} := \varphi_{j^{(n)}}(K)$.
    
    If $B \su K$ is a maximal ball, then Lemma \ref{lem innerball} implies that $\diam\pr{B} \geq \frac {\nu(K)} 2$.
    For each $j^{(n)}$, set $B_{j^{(n)}} = \varphi_{j^{(n)}}(B)$ so that $B_{j^{(n)}} \su K_{j^{(n)}}$ and $\diam\pr{B_{j^{(n)}}} \geq \frac{\nu(K)}{2\diam(K)} \diam\pr{K_{j^{(n)}}}$.

    Since $\tau$ is the overlapping index of $\set{f_k}_{k=1}^M$, then by Lemma \ref{overIterationCor}, the overlapping index of the generations $\set{C_n}_{n=1}^\iny$ is at most $\tau$.
    It follows that $\set{B_{j^{(n)}}}$ can be partitioned into $\tau$ subsets, each of which contains disjoint convex sets.
    
    Given $\rho \in (0, 1)$, choose $k \in \N$ depending on $\rho$ so that 
    \begin{align*}
        r\rho\leq \diam\pr{K_{j^{(k)}}} < \rho.
    \end{align*}
    For $x \in C$, let $\mathcal{S}(x, \rho) = \{j^{(k)} \in \set{1, \ldots, N}^k : K_{j^{(k)}}\cap B(x, \rho) \neq \varnothing\}$.
   
%     Then we see that
%     \begin{align*}
%         \frac{\nu(K)}{\diam(K)} r \rho \cdot \#\mathcal{S}(x, \rho)
%         &\le \frac{\nu(K)}{\diam(K)} \sum_{j^{(k)} \in \mathcal{S}(x, \rho)} \diam\pr{K_{j^{(k)}}}
%         \leq 4 \sum_{j^{(k)} \in \mathcal{S}(x, \rho)} \text{rad}\pr{B_{j^{(k)}}} \\
%         &\le \frac{4}{a} \sum_{j^{(k)} \in \mathcal{S}(x, \rho)} \Ha^1\pr{H \cap B_{j^{(k)}}} 
%         \le \frac{4 \tau}{a} \Ha^1\pr{H \cap \bigcup_{j^{(k)} \in \mathcal{S}(x, \rho)}  B_{j^{(k)}}} \\
%         &\le \frac{4 \tau}{a} \Ha^1\pr{H \cap B(x, 2\rho)}
%         \le \frac{8 \tau b}{a} \rho,
%     \end{align*}
%     where we recall that $a$ and $b$ are the Ahlfors lower and upper constants of $H$.
%     Therefore, $\disp \#\mathcal{S}(x, \rho) \le \frac{8 \tau b \diam(K)}{a \nu(K)} r^{-1}$, and an application of Jensen's inequality shows that
% \begin{equation}
% \label{Kbounds}
%     \begin{aligned}
%         \sum_{j^{(k)} \in\mathcal{S}(x, \rho)} \brac{\diam\pr{K_{j^{(k)}}}}^{\gamma}  
%         &\le {\#\mathcal{S}(x, \rho)}^{1 - \ga} \brac{\sum_{j^{(k)} \in\mathcal{S}(x, \rho)} \diam\pr{K_{j^{(k)}}}}^{\gamma} \\
%         &\leq \frac{8 \tau b \diam(K)}{a \nu(K)}  r^{\gamma-1} \rho^{\gamma}.
%  \end{aligned}   
% \end{equation}

Setting $\dim(H) = s$, we see that
\begin{align*}
    \pr{r \rho}^{s} \cdot \#\mathcal{S}(x, \rho)
    &\le \sum_{j^{(k)} \in \mathcal{S}(x, \rho)} \brac{\diam\pr{K_{j^{(k)}}}}^{s}
    \leq \brac{\frac{4 \diam(K)}{\nu(K)}}^{s} \sum_{j^{(k)} \in \mathcal{S}(x, \rho)} \brac{\text{rad}\pr{B_{j^{(k)}}}}^{s} \\
    &\le \brac{\frac{4 \diam(K)}{\nu(K)}}^{s} \sum_{j^{(k)} \in \mathcal{S}(x, \rho)} \frac{\Ha^{s}\pr{H \cap B_{j^{(k)}}} }{a} \\
    &\le \brac{\frac{4 \diam(K)}{\nu(K)}}^{s} \frac{\tau}{a} \Ha^{s}\pr{H \cap \bigcup_{j^{(k)} \in \mathcal{S}(x, \rho)}  B_{j^{(k)}}} \\
    &\le \brac{\frac{4 \diam(K)}{\nu(K)}}^{s} \frac{\tau}{a} \Ha^{s}\pr{H \cap B(x, 2\rho)}
    \le \frac{\tau b}{a} \brac{\frac{8 \diam(K)}{\nu(K)}}^{s} \rho^{s},
\end{align*}
where we recall that $a$ and $b$ are the Ahlfors lower and upper constants of $H$.
In particular,
$$\#\mathcal{S}(x, \rho) \le \frac{\tau b}{a} \brac{\frac{8 \diam(K)}{\nu(K)}}^{s} r^{-{s}}$$
and an application of Jensen's inequality shows that
\begin{equation}
\label{Kbounds}
\begin{aligned}
    \sum_{j^{(k)} \in\mathcal{S}(x, \rho)} \brac{\diam\pr{K_{j^{(k)}}}}^{\gamma}  
    &\le {\#\mathcal{S}(x, \rho)}^{1 - \frac \ga {s}} \pr{\sum_{j^{(k)} \in\mathcal{S}(x, \rho)} \brac{\diam\pr{K_{j^{(k)}}}}^{s}}^{\frac \gamma {s}} \\
    &\leq \frac{\tau b}{a} \brac{\frac{8 \diam(K)}{\nu(K)}}^{s} r^{\ga - {s}} \rho^\ga.
\end{aligned}   
\end{equation}
Since
\begin{align*}
    \sum_{j_{k+1} = 1}^N \brac{\diam(K_{(j^{(k)}, j_{k+1})})}^\ga
    = \sum_{j = 1}^N \brac{r \diam(K_{j^{(k)}})}^\ga
    = \brac{\diam(K_{j^{(k)}})}^\ga,
\end{align*} 
then for any $\ell \geq k$,
    \begin{equation}
    \label{lkComp}
        \sum_{\substack{j^{(\ell)} \\ K_{j^{(\ell)}}\cap B(x, \rho) \neq \varnothing}} \brac{\diam(K_{j^{(\ell)}})}^{\gamma} \leq \sum_{j^{(k)} \in \mathcal{S}(x, \rho)} \brac{\diam(K_{j^{(k)}})}^{\gamma}. 
    \end{equation}
Combining \eqref{lkComp} with \eqref{Kbounds} then shows that \eqref{AhlforUpperSub} holds.

To establish the lower bound on the measure of the attractor, we make an argument similar to the proof of \cite[Theorem 4.5]{martin1988}:
Let $\set{A_\iota}$ be a cover for $C$ such that $\diam(A_\iota) \le \rho \le \frac 1 2$ for each $\iota$.
By compactness of $C$, there exists $L \in \N$ so that the finite subcollection $\set{A_i}_{i=1}^L$ covers $C$.

For each $A_i$, let $B_i = B(x_i, \rho_i) \supseteq A_i$ be a ball for which $\diam(B_i) = 2 \diam(A_i) \le 2 \rho$.
For each $\rho_i$, choose $k_i \in \N$ depending on $\rho_i$ so that $r \rho_i \le \diam(K_{j^{(k_i)}}) \le \rho_i$.
Let $k_0 = \max\set{k_1, \ldots, k_L}$.
Then using \eqref{Kbounds} and \eqref{lkComp}, we see that
\begin{align*}
    \frac{\tau b }{a} \brac{\frac{8\diam(K)}{\nu(K)}}^{s} r^{\ga - {s}} \sum_{i=1}^L \diam\pr{A_i}^\ga
    & = \frac{\tau b }{a} \brac{\frac{8\diam(K)}{\nu(K)}}^{s} r^{\ga - {s}}   \sum_{i=1}^L \rho_i^\ga \\
    &\ge \sum_{i=1}^L \sum_{j^{(k_i)} \in\mathcal{S}(x_i, \rho_i)} \brac{\diam\pr{K_{j^{(k_i)}}}}^{\gamma} \\
    &\ge \sum_{i=1}^L \sum_{\substack{j^{(k_0)} \\ { K_{j^{(k_0)}}} \cap B(x_i, \rho_i) \ne \varnothing} }\brac{\diam\pr{K_{j^{(k_0)}}}}^{\gamma} 
    \ge \sum_{j^{(k_0)}} \brac{\diam\pr{K_{j^{(k_0)}}}}^{\gamma} \\
    &= \sum_{j^{(k_0)}} \brac{r^{k_0}\diam(K)}^{\gamma}
    = \pr{N r^{\gamma}}^{k_0}\brac{\diam(K)}^{\gamma}
    = \brac{\diam(K)}^{\gamma},
\end{align*}
where the last inequality holds because $\set{B_i}_{i=1}^L$ covers $C$. 
Since $\set{A_i}_{i = 1}^L$ was an arbitrary cover for $C$, then \eqref{measureLowerBoundsub} follows. 
\end{proof}

To summarize this subsection, we have shown that from a general IFS that satisfies the open set condition, we can produce a UIFS within a specified dimension range, provide bounds on the scale factor, a bound on the upper Ahlfors constant, and a lower bound on the measure of the attractor of the UIFS.

\subsection{Existence of large projections}
\label{subsec bigproj}

$\quad$ \\ 

In this subsection, we show that from an IFS with a uniform scale factor, we can find sufficiently large subsets of its generations with nice projection properties.
Our starting point is the following result of Peres and Shmerkin \cite{peres2009resonance}. 

\begin{proposition}[Proposition 7 in \cite{peres2009resonance}, Separated projections]
\label{goodprojection}
 Given constants $\al_0 >1, \al_1, \al_2>0$ and $\ga \in \pr{0,1}$, there exists a constant $\delta > 0$ such that the following holds:
 
Fix $\rho>0$. 
Let $\mathcal{Q}$ be a collection of disjoint closed convex subsets of the unit ball such that each element contains a ball of radius $\al_0^{-1} \rho$ and is contained in a ball of radius $\al_0 \rho$. 
Suppose that $\mathcal{Q}$ has cardinality at least $\al_1^{-1} \rho^{-\gamma}$, yet any ball of radius $\ell \in(\rho, 1)$ intersects at most $\al_2(\ell / \rho)^\gamma$ elements of $\mathcal{Q}$. 

Then for any $\varepsilon>0$, there exists a set $J \su [0, \pi)$ with the following properties:
    \begin{enumerate}
        \item $\abs{[0, \pi) \backslash J}  \leq \varepsilon$
        \item if $\phi \in J$, then there exists a subcollection $\mathcal{Q}_\phi$ of $\mathcal{Q}$ of cardinality at least $\varepsilon \delta \# \mathcal{Q}$ such that the orthogonal projections of the sets in $\mathcal{Q}_\phi$ onto a line with direction $\phi$ are all $\rho$-separated;
        \item  $J$ is a finite union of open intervals.
    \end{enumerate}
\end{proposition}

\begin{remark}
\label{disjointInterior}
    An inspection of the proof of  \cite[Proposition 7]{peres2009resonance} shows that the disjointness requirement is not necessary.
    In particular, the result still holds under the assumption that the collection $\mathcal{Q}$ consists of convex sets with disjoint interiors.
    Additionally, the assumption that $\mathcal{Q}$ is contained in the unit ball can be replaced by an assumption that $\mathcal{Q}$ is contained in any compact set.

    Removing the disjointness condition transforms Proposition \ref{goodprojection} into something that is almost a generalization of Proposition \ref{subIFSconstruction}. 
    In fact, the heuristics of both statements are identical: leverage Mattila's lower bound on Favard  length \cite{mattila1990} to prove a discrete quantitative projection theorem. 
    This raises the question of whether Proposition \ref{goodprojection} can be used in place of Proposition \ref{subIFSconstruction}. 
    One issue with this replacement is that Proposition \ref{subIFSconstruction} handles the $\gamma=1$ case in Proposition \ref{goodprojection}. 
    Remark \ref{rem deltagamma} makes clear that this case is critical for Proposition \ref{goodprojection}. 
    However, we think that one could construct a limiting argument, but that would be unnecessarily powerful and complicated considering the simplicity of Proposition \ref{subIFSconstruction}.
\end{remark}

 \begin{remark}[Understanding $\delta$ as a function of $\gamma$]\label{rem deltagamma}
In the original proof of Proposition \ref{goodprojection} from \cite[Proposition 7]{peres2009resonance}, $\delta$ is defined with an explicit identity:
    \begin{align*}
        \delta= (5\alpha_0+5)^{-1} A_5^{-1}
    \end{align*}
for some non-explicit constant, $A_5$. 
We can track through their proof and show that $A_5 = c_e A_4$, where $c_e$ is a universal constant that comes from Theorem \ref{MattilaLower} and 
    \begin{align*}
        A_4 = \al_0^2 \al_1 \pr{2 \pi + \frac{e \al_0^2 \al_2}{1 - e^{-(1 - \ga)}}}
\le \frac{3 \al_0^4 \al_1 \al_2}{1 - e^{-(1 - \ga)}}.
    \end{align*}
That is, we can take 
    \begin{align*}
        \de = \frac{1 - e^{-(1 - \ga)}}{15 c_e \pr{\al_0 + 1}\al_0^4 \al_1 \al_2}.
    \end{align*}
 \end{remark}

In the following lemma, we show that we may apply Proposition \ref{goodprojection} to all of the generations of an IFS.
The constants that play the roles of $\al_0, \al_1, \al_2$ are independent of the generation, $n$.

\begin{lemma}[Application of Proposition 7 in \cite{peres2009resonance}]
\label{appGoodProj}
   Let $\set{\varphi_j}_{j=1}^N$ be an IFS that satisfies the open set condition, has uniform scale factor $r$, similarity dimension $\ga \in \pr{0, 1}$, attractor $C$, 
    and upper Ahlfors regularity constant $b_0$.
    Let $K$ be a compact, convex, $\nu$-non-degenerate set for which $C \su K$.
    Let $\set{C_n}_{n=0}^\iny$ denote the generations of $\set{\varphi_j}_{j=1}^N$ with respect to $K$, 
    and let $\tau_0$ denote the overlapping index of the generations. 
    
    There exists a constant $\de_0 > 0$ so that for every $n \in \N$, the following holds:
    For any $\varepsilon>0$, there exists a set $J \su [0, \pi)$ for which
    \begin{enumerate}
        \item $\abs{[0, \pi) \backslash J} \leq \varepsilon$
        \item If $\phi \in J$, then there exists a subset $C_{n,\phi} \su C_n$, with $N_\phi \ge \varepsilon \delta_0 r^{-n \ga}$ connected components, for which $P_\phi(C_{n,\phi})$ is a union of $N_\phi$ $r^n$-separated intervals;
    \end{enumerate}
    The constant $\de_0$ given in \eqref{delta0Defn} depends on $\ga$, $\nu(K)$, $\diam(K)$, $H^\ga(C)$, $b_0$, $\tau_0$, and is independent of $n$.
\end{lemma}

\begin{proof}
Fix $n \in \N$.
We have $\disp C_n = \bigcup_{j^{(n)}} K_{j^{(n)}}$, where $K_{j^{(n)}} := \varphi_{j^{(n)}}(K)$.
The claimed result follows from an application of Proposition \ref{goodprojection} with $\rho = r^n$ and a choice of $\mathcal{Q} \su \set{K_{j^{(n)}}}$ that satisfies the hypotheses. 
We first specify $\mathcal{Q}$.
Since the generations $\{C_n\}_{n=0}^\iny$ have overlapping index $\tau_0$, then the index set can be partitioned as $\disp \set{1, \ldots, N}^n = \bigcup_{k=1}^{\tau_0} \mathcal{T}_k$, where the sets in each collection $\set{K_{j^{(n)}}: j^{(n)} \in \mathcal{T}_k}$ have disjoint interiors. 
Choose a maximal $\mathcal{T} \in \set{\mathcal{T}_k}_{k=1}^{\tau_0}$ in the sense that $\# \mathcal{T} = \max\set{ \# \mathcal{T}_k : k \in \set{1, \ldots, \tau_0}}$ and define 
$$\mathcal{Q}:=\{K_{j^{(n)}}:j^{(n)} \in \mathcal{T} \}.$$

Now we check that $\mathcal{Q}$ satisfies the hypotheses of Proposition \ref{goodprojection}.
By construction, $\mathcal{Q}$ is a collection of closed, convex sets with disjoint interiors, which by Remark \ref{disjointInterior}, is sufficient.
With 
\begin{equation}
    \label{al0Defn}
    \al_0 = \al_0(K) := \max \set{\frac{\diam(K)} 2, \frac 4{\nu(K)}, 1} \ge 1,
\end{equation}
     we see that $K$ is contained in a ball of radius $\al_0$, and Lemma \ref{lem innerball} implies that $K$ contains a ball of radius $\al_0^{-1}$.
Since each $\varphi_{j}(x) = r A_j x + z_j$ is a linear contraction, then every $K_{j^{(n)}}$ 
is a translated, rotated, $r^n$-scaled copy of $K$. 
It follows that every element of $\mathcal{Q}$ contains a ball of radius $\al_0^{-1} r^n$ and is contained in a ball of radius $\al_0 r^n$. 

By the pigeonhole principle, $\# \mathcal{T} \ge \tau_0^{-1} N^n$.
Since the similarity dimension of $C$ is $\gamma$, then $N r^\ga= 1$ so that $N^n = r^{- n \ga}$.
Therefore, with $\al_1 = \tau_0$,
\begin{equation}
    \label{QCount}
    \# \mathcal{Q} \ge \al_1^{-1} \pr{r^{- n }}^\ga.
\end{equation}

Since $\set{\vp_j}_{j=1}^N$ has a uniform scale factor, then there exists $m_n > 0$ so that $m_n = \mathcal{H}^{\gamma}\pr{\varphi_{j^{(n)}}(C)}$ for every $j^{(n)}$.
By Proposition \ref{openset}, $\Ha^\ga(C) \in (0, \iny)$, and $\mathcal{H}^{\gamma}\pr{\varphi_{j^{(n)}}(C)\cap\varphi_{i^{(n)}}(C)} = 0$  whenever $i^{(n)} \neq j^{(n)}$.
Since $\disp C = \bigcup_{j^{(n)}} \varphi_{j^{(n)}}(C)$, we deduce that
\begin{align*}
    \Ha^\ga\pr{C} 
    = \Ha^\ga\pr{\bigcup_{j^{(n)}} \varphi_{j^{(n)}}(C)}
    &= \sum_{j^{(n)}} \Ha^\ga\pr{\varphi_{j^{(n)}}(C)}
    = N^n m_n
\end{align*}
from which it follows that
\begin{equation}
\label{mnDefn}
 m_n = \Ha^\ga\pr{C} r^{n \ga}.   
\end{equation}

Since $\varphi_{j^{(n)}}(C) \su C \su K$, then $\varphi_{j^{(n)}}(C) \su C \cap K_{j^{(n)}} \su K_{j^{(n)}}$.
By taking a union, we see that $\disp C = \bigcup_{j^{(n)}} \varphi_{j^{(n)}}(C) \su \bigcup_{j^{(n)}} K_{j^{(n)}} = C_n$.
On the other hand, since $\varphi_{j^{(n)}}(C) \su C \cap K_{j^{(n)}}$, then $C \cap K_{j^{(n)}} \ne \varnothing$, so for each $j^{(n)}$, there exists $x_{j^{(n)}} \in C \cap K_{j^{(n)}}$.
Because each $K_{j^{(n)}}$ is contained in a ball of radius $\al_0 r^n$, then $K_{j^{(n)}} \su B\pr{x_{j^{(n)}}, 2 \al_0 r^n}$.
By taking a union, it follows that $C_n \su C\pr{2 \al_0 r^n}$.
Therefore,
$$C \su C_n \su C\pr{2 \al_0 r^n}.$$

For $\ell \in (r^n, 1)$ and a point $x$, we want to count how many elements of $\mathcal{Q}$ intersect $B(x, \ell)$.
This is equivalent to estimating the number of elements in the index set 
$$\mathcal{T}(x, \ell) = \{j^{(n)}  \in \mathcal{T} : K_{j^{(n)}}\cap B(x, \ell) \neq \varnothing\}.$$
Assume that $B(x, \pr{1 + 2\al_0}\ell) \cap C = \varnothing$.
Since $C_n \su C\pr{2 \al_0 r^n} \su C\pr{2 \al_0 \ell}$, then $B(x, \ell) \cap C_n = \varnothing$ which implies that $\#\mathcal{T}(x, \ell) = 0$.
Therefore, there is no loss in assuming that $x$ is a point for which $B(x, \pr{1 + 2\al_0}\ell) \cap C \ne \varnothing$.
Because each $K_{j^{(n)}}$ is contained in a ball of radius $\al_0 r^n < \al_0 \ell$, then $\disp \bigcup_{j^{(n)} \in \mathcal{T}(x, \ell)} K_{j^{(n)}} \su B(x, \pr{1 + 2\al_0}\ell)$. 
And since $B(x, \pr{1 + 2 \al_0}\ell) \cap C \neq \varnothing$, then $\disp \bigcup_{j^{(n)} \in \mathcal{T}(x, \ell)} K_{j^{(n)}} \su B(y, (2+4\al_0)\ell)$ for some $y \in B(x, \pr{1 + 2 \al_0}\ell) \cap C$.
Therefore,
\begin{align*}
    m_n \#\mathcal{T}(x, \ell)
    &= \sum_{j^{(n)} \in \mathcal{T}(x, \ell)} \mathcal{H}^{\gamma}\pr{\varphi_{j^{(n)}}(C)}
    \le \sum_{j^{(n)} \in \mathcal{T}(x, \ell)} \mathcal{H}^{\gamma}\pr{C \cap K_{j^{(n)}}} \\
    &= \mathcal{H}^{\gamma}\pr{C \cap \bigcup_{j^{(n)} \in \mathcal{T}(x, \ell)} K_{j^{(n)}}}
    \le \mathcal{H}^{\gamma}\pr{C \cap B(y, (2+4\al_0)\ell)}
\end{align*}
and it follows from the upper Ahlfors regularity of $C$ and \eqref{mnDefn} that
\begin{align*}
    \#\mathcal{T}(x, \ell)
    &\le \al_2 \pr{\frac \ell{r^n}}^\ga,
\end{align*}
where $\disp \al_2 = \frac{b_0 (2+4\al_0)^\ga}{\mathcal{H}^{\gamma}(C) }$.

We may now apply Proposition \ref{goodprojection} to $\mathcal{Q}$ with $\rho = r^n$ to reach the conclusion.
That is, for any $\varepsilon>0$, there exists a set $J \su [0, \pi)$ for which $\abs{[0, \pi) \backslash J} \leq \varepsilon$, and if $\phi \in J$, then there exists a subset $C_{n,\phi} \su C_n$, that has at least $\eps \de \# \mathcal{Q} \ge \eps \de \tau_0^{-1} r^{- n \ga}$ components, where we have used \eqref{QCount}.
In particular, Remark \ref{rem deltagamma} shows that our claim holds with
\begin{equation}
\label{delta0Defn}
  \de_0 = \de \tau_0^{-1} 
  = \frac{\mathcal{H}^{\gamma}(C)}{15c_e \al_0^4 \pr{\al_0 + 1} \tau_0^2 b_0 } \frac{1 - e^{-(1 - \ga)}}{ (2+4\al_0)^\ga}.  
\end{equation}
We point out that $\de_0$ depends only on $C$ and its properties, and is independent of $n$.
\end{proof}

In summary, we have shown that the generations of a UIFS have good projection properties in the sense that one can always find a substantial set of angles onto which the projections contain lots of pieces.
Moreover, all these statements are quantitative.

\subsection{Ergodic theory} 
\label{subsec Ergodic}

$\quad$ \\ 

In this subsection, we establish a consequence of the Maximal Ergodic Theorem.
This result will be used to show that many generations of a given IFS have a desired projection property.

\begin{theorem}[Maximal Ergodic Theorem, Theorem 2.24 in \cite{einsiedler2013ergodic}]
\label{MET}
    Let $(X, \mathscr{B}, \mu)$ be a probability space with a measure-preserving transformation $T$.
    Let $g$ be a real-valued function in $L^1(\mu)$. 
    For any $\alpha \in \mathbb{R}$, define
        \begin{align*}
            E_{\alpha}= \left\{ x \in X ~:~ \sup_{n \geq 1} \frac{1}{n} \sum_{k=0}^{n-1} g(T^k x) > \alpha  \right\}.
        \end{align*}
    Then 
        \begin{align*}
            \alpha \mu(E_{\alpha}) \leq \int_{E_{\alpha}} g \,d\mu \leq \|g\|_1.
        \end{align*}
\end{theorem}

As an application of this result, we have the following.

\begin{lemma}[Density of rotation maps]
\label{goodscales}
    Let $J \su [0, \pi)$ satisfy $\abs{J} \ge \pi - \eps$ and let $\vp \in [0, \pi)$ be any angle.
    Define the rotation map $T_\vp: [0, \pi) \to [0, \pi)$ by 
    $$T_\vp(\te) = \left\{\begin{array}{ll} 
    \te + \vp & \te \in [0, \pi - \vp) \\
    \te + \vp - \pi & \te \in [\pi - \vp, \pi)
    \end{array}\right.$$ 
    and the associated density as
    $$D(n; \te) = \frac{\# \set{k \in \set{0, 1, \ldots, n-1} : T_\vp^k(\te) \in J }}{n}.$$
    There exists $\te \in [0, \pi)$ so that for every $n \in \N$, $\disp D(n; \te)\in \brac{1 - \frac \eps 2, 1}$ .
\end{lemma}

\begin{proof}
Let $X = [0, \pi)$ with probability measure $\mu=\frac 1 {\pi} \abs{\cdot}$ and let $\mathscr{B}$ denote the associated Borel algebra.
The rotation map $T = T_\vp$ is measure-preserving.
Since $\abs{J} \ge \pi - \eps$, then $\mu\pr{J} \ge 1 - \frac \eps \pi$.
Define $g = \chi_{J^c}$ and choose $\al = \frac \eps 2$ so that with the notation from Theorem \ref{MET},
\begin{align*}
    E_{\frac \eps 2}
    = \set{ x \in X ~:~ \sup_{n \geq 1} \frac{1}{n} \sum_{k=0}^{n-1} \chi_{J^c}(T^k x) > \frac \eps 2}.
\end{align*}
An application of Theorem \ref{MET} then implies that
\begin{align*}
    \mu(E_{\frac \eps 2}) 
    \leq \frac{2\mu(J^c)}{\eps} 
    \le \frac 2 \pi < 1.
\end{align*}
Therefore, $\mu\pr{E_{\frac \eps 2}^c} > 0$.
In particular, there exists $x \in E_{\frac \eps 2}^c$.
That is, for every $n \in \N$,
$$\sum_{k=0}^{n-1} \chi_{J^c}(T^k x) \le \frac{\eps n}2.$$
Since 
$$\sum_{k=0}^{n-1} \chi_{J^c}(T^k x) + \sum_{k=0}^{n-1} \chi_{J}(T^k x) = n,$$
then
$$\sum_{k=0}^{n-1} \chi_{J}(T^k x) \ge \pr{1 - \frac \eps 2} n$$
and the conclusion follows.
\end{proof}

\subsection{Extracting a substantial subset and building the graph}
\label{subsec rotgraphconstr}

$\quad$ \\ 

Using the results from the previous subsections, we now detail the extraction process of a substantial subset from the attractor of a rotational IFS. 
We then conclude by showing that the extracted subset is suitable for an application of Proposition \ref{prop graphconstruct}, and can therefore be covered by a Lipschitz graph.

\begin{proposition}[Construction of substantial subsets]
\label{EnCor}
    Let $\set{f_k}_{k=1}^M$, where each $f_k$ is of the form \eqref{IFSDef}, be a $\nu$-non-degenerate IFS that satisfies the open set condition with similarity dimension $1$.
    Let $H$ be the attractor of $\set{f_k}_{k=1}^M$, let $a$ and $b$ denote the Ahlfors lower and upper constants of $H$, respectively, let $\tau$ denote the overlapping index of $\set{f_k}_{k=1}^M$, and set $K = \conv(H)$.
    
    There exists a constant $\eps_0\pr{r_1, \ldots, r_M, M, a, b, \tau, K} \in (0,1)$ such that for any $\varepsilon \in (0, \varepsilon_0]$, there exists $\kappa \in \N$ and a UIFS $\disp \set{\vp_j}_{j=1}^N \su \set{f_{k^{(\kappa)}}}$ with scale factor
    \begin{equation}
    \label{rLBound}
        r \geq c_1 \eps^{\frac{20 M}{\varepsilon}},
    \end{equation}
    where $c_1$ is defined in \eqref{c123Defn}.
    Moreover, there exists $\theta \in \mathbb{S}^1$ and a nested collection $\set{E_n}_{n =1}^{\iny} \su H$, satisfying the following properties for all $n \in \mathbb{N}$:
   \begin{enumerate}
        \item \label{point1}
        $\disp E_n = \bigcup_{i=1}^{M_n} K_i^n$, where each $K_i^n=\varphi_{j^{(n)}}(K)$ for some $j^{(n)} \in \{1, ..., N\}^n$.
        \item  \label{point2} 
        If $K_i^{n-1}$ is a connected component of $E_{n-1}$, then $E_n \cap K_i^{n-1}$ has $N_n \le N$ connected components and $P_{\theta}(E_n \cap K_i^{n-1})$ is a union of $N_n$ $r^{n}$-separated intervals. 
        \item  \label{point3} 
        $E_n$ has $M_n \ge r^{-\pr{1 - \eps}n}$ connected components.
   \end{enumerate}
\end{proposition}
 
\begin{proof}
Set $\eps_0 = \min\set{10 \eta_0, 10 \eta_1, \pr{\frac{c_2}{10}}^{3}}$, where $\eta_0$ is from Lemma \ref{PSPropDrop}, $c_2$ is defined in \eqref{c123Defn}, $\eta_1 := \pr{\frac{2}{3 c_2}}^3 c_4$ with
\begin{align*}
    \disp c_4 := \pr{\frac{a \nu}{8 \tau^2 b  \al_0^2 \diam(K)} }^2 \frac{1 }{3 c_e \pr{\al_0 + 1}} \inf_{\ga \in \brac{0, 1} } \brac{\frac{\diam(K) }{ 4\al_0 + 2 }}^{\gamma}\frac{1 - e^{-(1 - \ga)}}{1-\gamma}>0,
\end{align*}
and $\al_0$ as in \eqref{al0Defn}.

Let $\eps \le \eps_0$ and set $\eta = \frac{\eps}{10}$.
Since $\eta \le \eta_0$, then an application of Lemma \ref{PSPropDrop} shows there exists a $\kappa \in \N$ and a UIFS $\set{\vp_j}_{j = 1}^N \su \set{f_{k^{(\kappa)}}}$ with attractor $C \su H$, $\disp \dim(C) =: \ga \in \brac{1 - \eta, 1 - \frac \eta 2}$, and a uniform scale factor $r$ that satisfies \eqref{rBound}.
In particular, the lower bound from \eqref{rBound} and the definition of $c_3$ in \eqref{c123Defn} show that 
    \begin{align*}
        r \ge c_1 \pr{c_2 \eta}^{\frac{c_3} \eta} = c_1 \pr{\frac{c_2}{10} \eps}^{\frac{15 M} \eps}.
    \end{align*}
Since $\disp \frac{c_2}{10} \ge \eps_0^{\frac 1 3} \ge \eps^{\frac 1 3}$, then \eqref{rLBound} follows.

We now construct the nested sequence of sets.
Lemma \ref{iterationOpenLemma} implies that $\set{\vp_j}_{j = 1}^N$ inherits the open set condition from $\set{f_k}_{k=1}^M$.
An application of Lemma \ref{PSMMProp} with $\dim(H) = 1$ shows that $\set{\vp_j}_{j=1}^N$ has upper Ahlfors regularity constant 
    \begin{equation}
    \label{b0Defn}
        b_0 = \frac{8 \tau b  \diam(K)}{a \nu} r^{\gamma-1}
    \end{equation}
and that the bound in \eqref{measureLowerBoundsub} holds.
Let $\set{C_n}_{n=0}^\iny$ denote the generations of $\set{\vp_j}_{j=1}^N$ with respect to $K$.

Lemma \ref{overIterationCor} implies that the generations $\set{C_n}_{n=0}^\iny$ have overlapping index $\tau_0 \le \tau$.
Therefore, we may apply Lemma \ref{appGoodProj} to $\set{\vp_j}_{j = 1}^N$ and $K$ with $n = 1$ to deduce that there exists a constant $\de_0 > 0$ and a set $J \su [0, \pi)$  such that $\abs{[0, \pi) \backslash J} \leq \varepsilon$ and the following holds:
\begin{equation}
\label{subsetsStatement}
\begin{aligned}
        & \text{For every $\phi \in J$, there exists a subset $C_{1, \phi} \su C_1$ with $N_\phi \ge \varepsilon \delta_0 r^{-\ga}$ connected} \\
        & \text{ components for which $P_\phi(C_{1, \phi})$ is a union of $N_\phi$ $r$-separated intervals.}
    \end{aligned}
\end{equation}
Using \eqref{delta0Defn}, we see that
\begin{align*}
    N_\phi 
    &\ge \eps \delta_0  r^{-\ga}
    = \frac{\mathcal{H}^{\gamma}(C)}{15 c_e \al_0^4 \pr{\al_0 + 1} \tau_0^2 b_0 } \frac{1 - e^{-(1 - \ga)}}{ (2+4\al_0)^\ga} \eps r^{-\ga} \\
    &\ge \pr{\frac{a \nu}{8 \tau^2 b \al_0^2 \diam(K)}}^2 \frac{1}{3 c_e  \pr{\al_0 + 1} } \brac{\frac{\diam(K)}{2+4\al_0}}^{\gamma} \pr{1 - e^{-(1 - \ga)}} 2\eta r^{2-3\ga} \\
    &\ge c_4 \pr{1 - \ga} 2 \eta r^{-1 + 3\pr{1-\ga}}
    \ge c_4 \eta^2 r^{- 1 + 3 \eta}
    = \frac{3c_2 }{2} \eta_1 \pr{\frac{3c_2 }{2} \eta}^2 r^{-1 + 3\eta} \\
    &\ge \pr{\frac{3c_2 }{2} \eta}^3 r^{-1 + 3\eta},
\end{align*}
where we have applied \eqref{b0Defn}, \eqref{measureLowerBoundsub}, and $\tau_0 \le \tau$, the definition of $c_4$, the bounds on $\ga$, the definition of $\eta_1$, and that $\eta \le \eta_1$.
From \eqref{rBound}, we see that $\frac{3c_2}{2}  \eta \ge r^{\frac{2 \eta}{M} }$ and since $M \ge 3$, we get
\begin{equation}
\label{NphiBound}
    N_\phi 
    \ge r^{\frac{6 \eta}{M}} r^{-1 + 3\eta}
    = r^{-1 + \frac \eps 2 \frac{3}{5}\pr{1 + \frac 2 M}} 
    \ge r^{-1 + \frac \eps 2}.
\end{equation}
   
    Let $A$ denote the rotation matrix for $\set{\vp_j}_{j=1}^N$. 
    The matrix $A$ induces a rotation map on $[0, \pi)$, $T_A$, in the sense of Lemma \ref{goodscales}. 
    For each $\te \in [0, \pi)$, define $\disp \mathcal{S}(\te) := \set{n \in \mathbb{N} ~:~ T_A^{n-1}(\theta) \in J}$.
    Since $\abs{J}\geq \pi -\varepsilon$, then Lemma \ref{goodscales} provides an angle $\theta \in [0, \pi)$ such that for every $n \in \N$, it holds that
    \begin{equation}
        \label{Sdense}
            \#(\mathcal{S(\te)} \cap \set{1, \ldots, n}) \geq \pr{1-\frac \varepsilon 2} n.
        \end{equation}

    For each $n \in \N$, we construct subsets $C_1(n) \su C_1$.
    The cases for $n \in \mathcal{S}(\te)$ and $n \notin \mathcal{S}(\te)$ are distinct.
    
    If $n \in \mathcal{S}(\te)$, then $P_{\theta} \circ A^{n-1} = P_{T^{n-1}_A(\theta)}$ and $\phi_n := T_A^{n-1}(\theta) \in J$.
    Therefore, by \eqref{subsetsStatement}, there exists $C_1(n) := C_{1, \phi_n} \su C_1$ with $N_n := N_{\phi_n}$ connected components for which $P_{\phi_n}(C_{1,\phi_n}) = P_{\theta}(A^{n-1}(C_{1}(n)))$ is a union of $N_{n}$ $r$-separated intervals.
    The bound \eqref{NphiBound} shows that $N_{n} \ge r^{-1 + \frac \eps 2}$.
    We may write
        \begin{align}
        \label{C1nDefn}
            C_{1}(n) := \bigcup_{j \in \mathcal{I}_{n}} \vp_j(K) \su C_1 ,
        \end{align}
    for some $\mathcal{I}_{n} \su \{1, ..., N\}$ with $\# {\mathcal{I}_{n}} = N_{n}$.
       
    If $n \in \N \setminus \mathcal{S}(\te)$, then we choose $C_{1}(n) \su C_1$ to have $N_n$ connected components for which the projection $P_{\theta}(A^{n-1}(C_{1}(n)))$ is a  union of $N_n$ $r$-separated intervals.
    Since $n \notin \mathcal{S}(\te)$, then $T_A^{n-1}(\theta) \notin J$, so we cannot guarantee that $N_n$ has a large lower bound, but we can ensure that $N_n \ge 1$.
    Again in this case, there is a non-empty subset of indices $\mathcal{I}_n \su \{1, ..., N\}$, where $\# {\mathcal{I}_n} = N_n$, such that \eqref{C1nDefn} holds.   
      
    For each $n \in \N$, define 
        \begin{align*}
            F_{n} 
            := \bigcup_{j^{(n-1)}} \vp_{j^{(n-1)}}(C_{1}(n)) 
            = \bigcup_{j^{(n-1)}} \bigcup_{j \in \mathcal{I}_n} \vp_{j^{(n-1)}}( \vp_j(K)) 
            \su C_{n}.
        \end{align*}
    An application of Corollary \ref{generCor} shows that $\set{C_n}_{n=0}^\iny$ is nested.
    In particular, $C_1 \su K$ so that $C_{1}(n) \su K$.
    Therefore, $\vp_{j^{(n-1)}}(C_{1}(n)) \su K_{j^{(n-1)}}$ and then $F_{n} \cap K_{j^{(n-1)}} = \vp_{j^{(n-1)}}(C_{1}(n))$. 
    For each $j^{(n-1)}$, there exists $w_{j^{(n-1)}} \in \mathbb{R}^2$ so that $\vp_{j^{(n-1)}}(x)= r^{n-1} A^{n-1}x + w_{j^{(n-1)}}$ and then
        \begin{align*}
            P_\te\pr{F_{n} \cap K_{j^{(n-1)}}}
            &= P_\te\pr{\vp_{j^{(n-1)}}(C_{1}(n))}
            = P_\te\pr{r^{n-1} A^{n-1}(C_{1}(n)) + w_{j^{(n-1)}}} \\
            &= r^{n-1} P_\te\pr{A^{n-1}(C_{1}(n))}
            + P_\te\pr{w_{j^{(n-1)}}}.
        \end{align*}
    Since $P_{\theta}(A^{n-1}(C_{1}(n)))$ is a union of $N_n$ $r$-separated intervals, then $P_\te\pr{F_{n} \cap K_{j^{(n-1)}}}$ is a union of $N_n$ $r^{n}$-separated intervals.     

    For each $n \in \N$, define $\disp E_n = \bigcap_{k = 1}^n F_k$.
    Since
    \begin{align*}
            F_{k} &= \bigcup_{j_1=1}^N \ldots \bigcup_{j_{k-1}=1}^N \bigcup_{j_k \in \mathcal{I}_k} \vp_{j_1} \circ \ldots \circ \vp_{j_{k-1}} \circ \vp_{j_k}(K),
        \end{align*}
    then
    \begin{align*}
        E_n = \bigcup_{j_1 \in \mathcal{I}_1} \bigcup_{j_2 \in \mathcal{I}_2}  \ldots \bigcup_{j_n \in \mathcal{I}_n} \vp_{j_1} \circ \vp_{j_2} \ldots \circ \vp_{j_n}(K)
        = \bigcup_{j^{(n)} \in \mathcal{M}_n} \vp_{j^{(n)}}(K),
    \end{align*}
    where $\mathcal{M}_n = \mathcal{I}_1 \times \mathcal{I}_2 \times \ldots \times \mathcal{I}_n \su \set{1, \ldots, N}^n$.
    In particular, each $E_n$ is as described in item \ref{point1}.
    
    For each $n \in \N$, since $E_n = F_n \cap E_{n-1}$, then each $E_n$ inherits the projective properties of $F_n$.
    That is, if $K_i^{n-1} = K_{j^{(n-1)}}$ is a connected component of $E_{n-1}$, then $E_n \cap K_i^{n-1} = F_n \cap K_i^{n-1}$, so $P_\te(E_n \cap K_i^{n-1})$ has $N_n$ $r^n$-separated intervals, establishing the item \ref{point2}.
    
    Each $E_n$ has $\disp M_n := \# \mathcal{M}_n = \prod_{k=1}^n N_k$ connected components. 
    Using that $N_k \ge r^{- \pr{1 - \frac \eps 2}}$ whenever $k \in \mathcal{S}(\te)$ and the density of $\mathcal{S}(\te)$ from \eqref{Sdense}, we see that
    \begin{align*}
        M_n 
        = \prod_{k=1}^n N_k
        \ge \prod_{k \in \mathcal{S}(\te) \cap \set{1, \ldots, n}} r^{- \pr{1 - \frac \eps 2}}
        \ge \brac{r^{- \pr{1 - \frac \eps 2}}}^{\pr{1 - \frac \eps 2} n}
        \ge r^{- \pr{1 - \eps} n},
    \end{align*}
    showing that item \ref{point3} also holds.
\end{proof}

Now we state and prove the main theorem for a general IFS.

\begin{theorem}[Theorem \ref{mainthm} in the Rotational Case]
\label{rotationalCase}
    Let $\set{f_k}_{k=1}^M$, where each $f_k$ is of the form \eqref{IFSDef}, be a $\nu$-non-degenerate IFS that satisfies the open set condition with similarity dimension $1$.
    Let $H$ be the attractor of $\set{f_k}_{k=1}^M$, let $a$ and $b$ denote the Ahlfors lower and upper constants of $H$, respectively, let $\tau$ denote the overlapping index of $\set{f_k}_{k=1}^M$, and set $K = \conv(H)$.
    
    There exists $\eps_0 = \eps_0\pr{r_1, \ldots, r_M, M, a, b, \tau, K} \in (0,1)$ so that for any $\eps \in \pb{0,\eps_0}$, there exists a Lipschitz graph $\Ga$ for which 
    $$\dim\pr{H \cap \Ga} \ge 1 - \eps$$ 
    and 
    $$\disp \Lip(\Ga) \le \frac{\diam(K)}{\prod_{k=1}^M r_k}  \max\set{\frac 1 {\nu}, 1}\exp\brac{20 M \eps^{-1} \log\pr{\eps^{-1}}}.$$ 
\end{theorem}

\begin{proof}
Let $\eps \in (0, \eps_0]$, where $\eps_0 > 0$ is given in Proposition \ref{EnCor}.
An application of Proposition \ref{EnCor} provides a $\kappa \in \N$ and a UIFS $\disp \set{\vp_j}_{j=1}^N \su \set{f_{k^{(\kappa)}}}$ with scale factor satisfying \eqref{rLBound}, and angle $\theta \in \mathbb{S}^1$, and a nested collection $\set{E_n}_{n =1}^{\iny} \su H$ that satisfies the properties from items \ref{point1} -- \ref{point3}.
Without loss of generality, we may assume that $\te = 0$ so that $P_{\theta} = P_x$, the projection onto the $x$-axis.

We check that $\set{E_n}_{n=1}^\iny$ satisfies the hypotheses of Proposition \ref{prop graphconstruct}. 
By construction, $\set{E_n}_{n=1}^\iny$ is nested and from item \ref{point1}, each $\disp E_n := \bigcup_{i=1}^{M_n} K^n_i$, where every connected component $K_i^n = \vp_{j^{(n)}}(K)$ is a translated and rescaled copy of $K$, hence closed and convex.
It follows that $r^n \nu \le \abs{P_\te\pr{K_i^n}} \le r^n \diam(K)$, and then \eqref{convexBound} holds with $\la = \frac{\diam(K)}{\nu}$.
For each $n, i$, $\diam\pr{K_i^n} \le \diam(K) {r}^{n}$ and then \eqref{diamBoundAbove} holds with $\si_n = \diam(K) {r}^{n}$.

It remains to check \eqref{connectorBound}.
For some $n \in \N$, pick $z_i \in K_i^n$ and $z_{j} \in K_{j}^n$ for $i \ne j$.
There exists a largest $k \in \set{1, \ldots, n}$ so that $K_i^n$ and $K_{j}^n$ both belong to the same connected component of $E_{k-1}$.
That is, $K_i^n \su K_{i'}^{k}$ and $K_j^n \su K_{j'}^{k}$ for $i' \ne j'$, while $K_{i'}^k, K_{j'}^k \su K_{q}^{k-1}$ for some $q \in \set{1, \ldots, M_{k-1}}$.
It follows that 
$$\abs{P_y(z_i - z_j)}\le \diam(K_{q}^{k-1}) \le \diam(K) r^{k-1},$$
while item \ref{point2} shows that 
$$\abs{P_x(z_i - z_j)} \ge r^k.$$
Therefore, \eqref{connectorBound} holds with $\la = \diam(K) r^{-1}$.
Since $r \ge c_1 \eps^{\frac{20 M}{\eps}}$ by \eqref{rLBound}, then both bounds \eqref{convexBound} and \eqref{connectorBound} hold with 
\begin{equation*}
    \lambda:= \frac{\diam(K)}{c_1} \max\set{\frac 1 {\nu}, 1} \eps^{- \frac{20M}{\eps}}.
\end{equation*}
Therefore, Proposition \ref{prop graphconstruct} is applicable and shows that there exists a Lipschitz graph, $\Gamma$, such that 
    \begin{align*}
        E := \bigcap_{n=1}^{\infty} E_n \su \Gamma
    \end{align*}
and $\Gamma$ has a Lipschitz constant that is bounded above by $\lambda$. 

Since $\set{E_n}_{n=1}^\iny \su H$, then $E \su H$ and $\dim\pr{\Ga \cap H} \ge \dim(E)$.
To estimate $\dim(E)$, Proposition \ref{prop dimboundcor} is applicable with $v_1 \cdots v_n = M_n$ and $D_n = d_n = r^n \diam(K)$ and $\disp b = \frac{\nu}{2\diam(K)}$.
Item \ref{point3} in Proposition \ref{EnCor} shows that $M_n \ge r^{-\pr{1 - \eps}n}$.
Therefore, Proposition \ref{prop dimboundcor} shows that
\begin{align*}
    \dim(E)  
    &\ge \liminf_{n \to \infty} \frac{\log( v_{1}v_{2} \cdots v_{n-1})}{-\log d_n}
    \ge \liminf_{n \to \infty} \frac{\log\brac{r^{-\pr{1 - \eps}(n-1)}}}{-\log \brac{r^n\diam(K) }} \\
    &= \liminf_{n \to \infty} \frac{\brac{\pr{1 - \eps}(n-1)}\log r}{n \log r + \log \brac{\diam(K)}}
    = 1 - \eps
\end{align*}  
and we conclude that $\dim\pr{\Ga \cap H} \ge 1 - \eps$.
\end{proof}

\appendix \section{The Measure of the 4-corner Cantor Set} 
\label{appendix}

Let $\mathcal{C}_4$ denote the 4-corner Cantor set.
Since projections decrease measure, we know that $\Ha^1(\mathcal{C}_4) \ge \Ha^1\pr{P_{\arctan(1/2)}\pr{\mathcal{C}_4}} = \frac{3}{\sqrt{5}}$.
See Figure \ref{projPic}.
On the other hand, if we choose the cubes involved in the construction of $\mathcal{C}_4$ as the covering in the definition of Hausdorff measure, then it follows that $\Ha^1(\mathcal{C}_4) \le \sqrt 2$.
Therefore, $\Ha^1(\mathcal{C}_4) \in \brac{\frac 3 {\sqrt 5}, \sqrt 2}$.

We include a computation of the exact Hausdorff measure of $\mathcal{C}_4$.
This result, originally due to Davies, was communicated in private communication to the second author by Kenneth Falconer. 
Other, less elementary proofs of this fact can be found in \cite{marion1979} and \cite{zhou1999hausdorff}.

\begin{proposition}[4-corner measure]
   $\Ha^1(\mathcal{C}_4) = \sqrt{2}$.
\end{proposition}

\begin{figure}[ht] 
\centering
\begin{tikzpicture}[scale = 0.82]
\fill[darkgray] (0,0) rectangle (1,1); 
 \fill[darkgray] (0,3) rectangle (1,4); 
 \fill[darkgray] (3,0) rectangle (4,1); 
 \fill[darkgray] (3,3) rectangle (4,4); 
 \draw[dotted] (-1, 0.5) -> (6, 4);
 \draw[dashed] (0, 0) -> (-0.4,0.8);
 \draw[dashed] (0, 3) -> (1,1);
 \draw[dashed] (1, 4) -> (3,0);
 \draw[dashed] (3, 3) -> (4,1);
 \draw[dashed] (4, 4) -> (4.4, 3.2);
 \draw[ultra thick] (-0.4,0.8) -> (4.4, 3.2);
 \draw (6, 3.25) node {$y = \frac 1 2 x + 1$};
 \draw (-0.5, 0.25) node {$\ell_0$};
 \draw (0.25, 2) node {$\ell_1$};
 \draw (2.5, 0.25) node {$\ell_2$};
 \draw (3.5, 1.5) node {$\ell_3$};
 \draw (4.5, 3.75) node {$\ell_4$};
\end{tikzpicture}
\qquad
\begin{tikzpicture}[scale = 0.82]
\fill[darkgray] (0,0) rectangle (0.25,0.25);
\fill[darkgray] (0,0.75) rectangle (0.25,1);
\fill[darkgray] (0.75,0) rectangle (1, 0.25);
\fill[darkgray] (0.75,0.75) rectangle (1,1);
\fill[darkgray] (3,0) rectangle (3.25,0.25);
\fill[darkgray] (3,0.75) rectangle (3.25,1);
\fill[darkgray] (3.75,0) rectangle (4, 0.25);
\fill[darkgray] (3.75,0.75) rectangle (4,1);
\fill[darkgray] (0,3) rectangle (0.25,3.25);
\fill[darkgray] (0,3.75) rectangle (0.25,4);
\fill[darkgray] (0.75,3) rectangle (1, 3.25);
\fill[darkgray] (0.75,3.75) rectangle (1,4);
\fill[darkgray] (3,3) rectangle (3.25,3.25);
\fill[darkgray] (3,3.75) rectangle (3.25,4);
\fill[darkgray] (3.75,3) rectangle (4, 3.25);
\fill[darkgray] (3.75,3.75) rectangle (4,4);
\draw[dashed] (0, 0) -> (-0.4,0.8);
\draw[dashed] (0.25, 0.25) -> (-0.1,0.95);
\draw[dashed] (0.75, 0) -> (0.2,1.1);
\draw[dashed] (1, 0.25) -> (0.5,1.25);
\draw[dashed] (0, 3) -> (1,1);
\draw[dashed] (0, 3.75) -> (1.1,1.55);
\draw[dashed] (0.25, 4) -> (1.4, 1.7);
\draw[dashed] (0.75, 3.75) -> (1.7, 1.85);
\draw[dashed] (1, 4) -> (3,0);
\draw[dashed] (3.25, 0.25) -> (2.3,2.15);
\draw[dashed] (3.75, 0) -> (2.6,2.3);
\draw[dashed] (4, 0.25) -> (2.9,2.45);
\draw[dashed] (3, 3) -> (4,1);
\draw[dashed] (3, 3.75) -> (3.5,2.75);
\draw[dashed] (3.25, 4) -> (3.8,2.9);
\draw[dashed] (3.75, 3.75) -> (4.1,3.05);
\draw[dashed] (4, 4) -> (4.4, 3.2);
\draw[ultra thick] (-0.4,0.8) -> (4.4, 3.2);
\draw[dotted] (-1, 0.5) -> (6, 4);
\draw (6, 3.25) node {$y = \frac 1 2 x + 1$};
\draw (-0.5, 0.25) node {$\ell_0$};
\draw (0.25, 2) node {$\ell_4$};
\draw (2.5, 0.25) node {$\ell_8$};
\draw (4.25, 1.5) node {$\ell_{12}$};
\draw (4.5, 3.75) node {$\ell_{16}$};
\end{tikzpicture}
\caption{
\label{projPic}
Each line $\ell_k$ meets the dotted line at a right angle.
The projection is indicated with a thick line showing that every generation of $\mathcal{C}_4$ projects onto the line $y = \frac 1 2 x + 1$ along the segment from $\pr{-\frac 1 {10}, \frac 1 5}$ to $\pr{\frac{11}{10}, \frac 4 5}$.}
\end{figure}
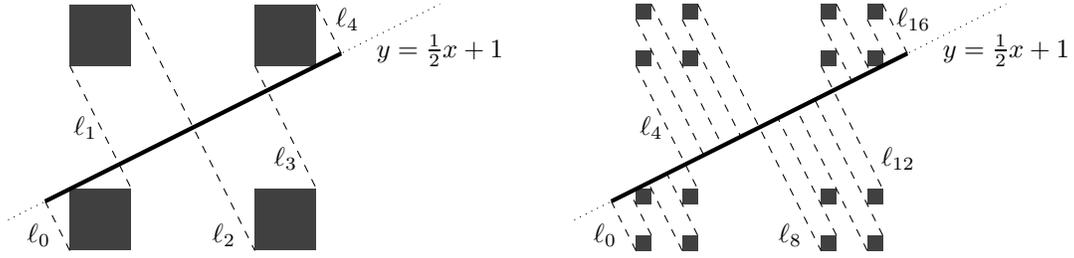

\begin{proof}
It suffices to show that for every $\delta >0$, $\Ha^1_\delta(\mathcal{C}_4) = \sqrt{2}$.

For any set in $E \su \R^2$, let $m^{\pm}E$ be the diameters of the projections of $E$ onto the lines $y = \pm x$, and let
\begin{equation*}
mE = \frac12 (m^+E + m^-E).
\end{equation*}
If $Q$ is any dyadic cube, denote by $Q^1, \dots, Q^4$ the grandchildren (two generations down) of $Q$ that are at the four corners. 
We want to show that
\begin{equation} 
\label{davies1}
mE \geq m(E \cap Q^1) + \cdots + m(E \cap Q^4).
\end{equation}
For any set $E \su \R^2$, there exists a rectangle $R$ with sides parallel to the lines $y = \pm x$ such that $E \su R$ and $m E = m R$.
Since $m\pr{R \cap Q^i} \ge m\pr{E \cap Q^i}$ for each $i$, then there is no loss of generality if we assume that $E$ itself is a rectangle with sides parallel to the line $y = \pm x$.

Suppose $E \cap Q^i \ne \varnothing$.
If $E^i$ is the smallest rectangle containing $E$ and $Q^i$, then $mE - m\pr{E \cap Q^i} = m E^i - m\pr{E^i \cap Q^i}$ and it follows that 
$$mE - \left[m(E \cap Q^1) + \cdots + m(E \cap Q^4)\right] \ge mE^i - \left[m(E^i \cap Q^1) + \cdots + m(E^i \cap Q^4)\right].$$
Therefore, there is no loss in assuming that if $E$ meets any square $Q^i$, then it contains it. 

\begin{figure}[ht]
\centering
\begin{tikzpicture}[scale = 0.82]
\fill[lightgray]  (-0.5, 0.5) -- (2,-2) --(6,2) --  (3.5, 4.5) -- (-0.5, 0.5) --
   cycle; 
\fill[darkgray] (0,0) rectangle (1,1); 
\fill[darkgray] (0,3) rectangle (1,4); 
\fill[darkgray] (3,0) rectangle (4,1); 
\fill[darkgray] (3,3) rectangle (4,4); 
\draw[dotted, thick, ->] (1.8,-2.2) -> (7, 3);
\draw[dotted, thick, ->] (2.2,-2.2) -> (-3, 3);
\draw (7.25, 2.25) node {$y = x$};
\draw (-3.25, 2.25) node {$y = -x$};
\draw[ultra thick] (2,-2) -> (3, -1);
\draw[ultra thick] (3.5, -0.5) -> (4.5, 0.5);
\draw[ultra thick] (5,1) -> (6, 2);
\draw[ultra thick] (2,-2) -> (1, -1);
\draw[ultra thick] (-0.5, 0.5) -> (0.5, -0.5);
\draw[ultra thick] (-0.6, 0.4) -> (0.4, -0.6);
\draw (4.5, 3) node {$E^3$};
\draw (5, 0.5) node {$\frac 1 2 \mu$};
\draw (3.5, -1) node {$\frac 1 2 \mu$};
\draw (3, -1.75) node {$m^+ Q^1$};
\draw (4.75, -0.25) node {$m^+ Q^3$};
\draw (6, 1.25) node {$m^+ Q^4$};
\draw (0.5, -1) node {$\frac 1 2 \mu$};
\draw (0.75, -1.75) node {$m^- Q^3$};
\draw (-1.5, -0.25) node {$m^- Q^1, m^- Q^4$};
\draw[color= white] (0.5, 0.5) node {$Q^1$};
\draw[color= white] (3.5, 0.5) node {$Q^3$};
\draw[color= white] (3.5, 3.5) node {$Q^4$};
\end{tikzpicture}
\caption{
\label{case3}
An illustration of the third case in the proof: The rectangle $E^3$ (in light gray) contains three of the four squares, $Q^1, Q^3, Q^4$.
The projections of each $Q^i$ onto the dotted lines ($y = \pm x$) are indicated with thick lines and their lengths are indicated.
The gaps are each of length $\frac 1 2 \mu$, where $\mu = m^\pm Q^i$.
}
\end{figure}
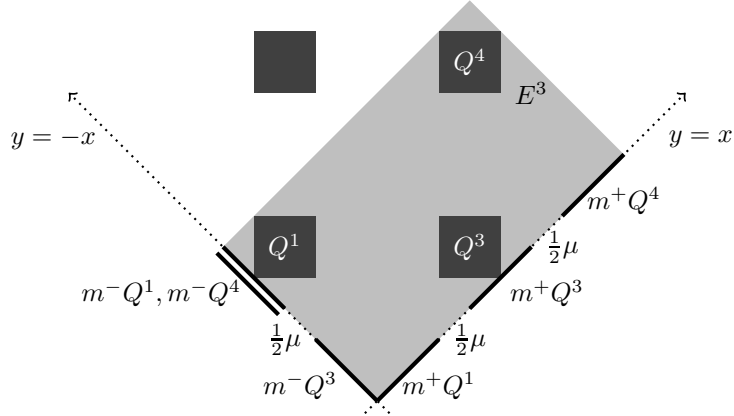

Let $\mu = m^\pm Q^i$ and note that $m Q^i = \mu$ for each $i$.

If $E$ contains only one $Q^i$, then let $E^1 \su E$ be the smallest rectangle with sides parallel to $y = \pm x$ that contains $Q^i$.
Since $m^\pm E^1 = \mu$, then $m E \ge m E^1 = m Q^i$.

If $E$ contains two cubes, $Q^i$ and $Q^j$, then let $E^2 \su E$ be the smallest rectangle with sides parallel to $y = \pm x$ that contains both $Q^i$ and $Q^j$.
In one case, $Q^j$ may be obtained by translating $Q^i$ along a line parallel to $y = \pm x$, and then $m^\pm E^2 = 4 \mu$ and $m^\mp E^2 = \mu$ so that $m E^2 = \frac 5 2 \mu$. 
Otherwise, $Q^j$ may be obtained by translating $Q^i$ along a line parallel to $y=0$ or $x=0$ in which case $m^\pm E^2 = \frac{5}{2}\mu$ and then $m E^2 = \frac{5}{2}\mu$.
It follows that,
$$m E \ge m E^2 = \frac 1 2 m^+ E^2 + \frac 1 2 m^- E^2 = \frac 5 2 \mu > 2 \mu = m Q^i + m Q^j.$$

If $E$ contains three cubes, $Q^i$, $Q^j$, and $Q^k$, then let $E^3 \su E$ be the smallest rectangle with sides parallel to $y = \pm x$ that contains all the cubes $Q^i$, $Q^j$, and $Q^k$.
Without loss of generality, $Q^j$ may be obtained by translating $Q^i$ along a line parallel to $y = x$, while $Q^k$ belongs to another corner.
It follows that $m^+ E^3 = 4 \mu$ while $m^- E^3 = \frac 5 2 \mu$.
Therefore, 
$$m E \ge m E^3 = \frac 1 2 m^+ E^3 + \frac 1 2 m^- E^3 = \frac{13}{4} \mu > 3 \mu = m Q^i + m Q^j + m Q^k.$$
This case is illustrated in Figure \ref{case3}.

Finally, if $E$ contains all four cubes, then with $E^4 \su E$ defined to be the smallest rectangle that contains all four cubes, we see that $m^\pm E^4 = 4 \mu$ and it follows that $m E \ge m E^4 = 4 \mu = m Q^1 + m Q^2 + m Q^3 + m Q^4$.
In all possible cases, we have established that \eqref{davies1} holds.

Recall that $\disp \mathcal{C}_4 = \bigcap_{n=1}^\iny C_n$.
If the squares of $C_n$ are denoted by $Q_n^i$, $i = 1 \dots,4^n$, then repeated applications of \eqref{davies1} show that 
\begin{equation} 
\label{davies2}
mE \geq \sum_{i=1}^{4^n} m (E \cap Q_n^i).
\end{equation}

We now proceed by contradiction. 
Assume that there exists a covering $\{U_j\}$ of $\mathcal{C}_4$ such that $\disp \sum_j \diam(U_j) < \sqrt{2}$. 
Because $m^{\pm}U_j \leq \diam (U_j)$, it follows that $mU_j \leq \diam (U_j)$ and so 
\begin{equation*}
\sum_j mU_j < \sqrt{2}.
\end{equation*}
We can assume that each $U_j$ is open (Theorem 4.4 in Mattila \cite{mattila}) and use that $\mathcal{C}_4$ is compact to conclude that $\{U_j\}_{j=1}^N$ is a finite collection.
In particular, we can find $n$ large enough so that for each $i = 1, \ldots, 4^n$, there exists an index $j \in \set{1, \ldots, N}$ so that $Q^i_n \su U_j$. 
Using \eqref{davies2}, we then see that
\begin{equation*}
\sum_{j=1}^{N} \sum_{i=1}^{4^n} m (U_j \cap Q^i_n) \leq \sum_{j=1}^{N} mU_j.
\end{equation*}
For each $i=1, \dots,4^n$ the term $mQ_n^i$ appears on the left hand side, and so 
\begin{equation*}
\sqrt{2} > \sum_j mU_j \geq \sum_{j=1}^{\infty} \sum_{i=1}^{4^n} m (U_j \cap Q^i_n) \geq \sum_{i=1}^{4^n} mQ_n^i = \sqrt{2},
\end{equation*}
which gives a contradiction and completes the proof.
\end{proof}

\bibliographystyle{alpha}
\bibliography{references}

\end{document}